\documentclass[11pt]{amsart}
\usepackage{amssymb,amsmath,amsthm,amsfonts,mathrsfs}
\usepackage[hmargin=3cm,vmargin=3.5cm]{geometry}
\usepackage{hyperref}
\usepackage{color}
\usepackage{graphicx}
\usepackage{stmaryrd}  

\newtheorem{theorem}{Theorem}[section]
\newtheorem{lemma}[theorem]{Lemma}
\newtheorem{remark}[theorem]{Remark}
\newtheorem{claim}[theorem]{Claim}
\newtheorem{definition}[theorem]{Definition}

\newtheorem{prop}[theorem]{Proposition}
\newtheorem{cor}[theorem]{Corollary}

\newtheorem{example}[theorem]{Example}

\let\oldtocsection=\tocsection
\let\oldtocsubsection=\tocsubsection
\renewcommand{\tocsection}[2]{\hspace{0em}\oldtocsection{#1}{#2}}
\renewcommand{\tocsubsection}[2]{\hspace{1em}\oldtocsubsection{#1}{#2}}

\usepackage{chngcntr}
\counterwithin{figure}{subsection}

\title{A deformation of Robert-Wagner foam evaluation and link homology}
 \author{Mikhail Khovanov}
 \address{Department of Mathematics, Columbia University, New York, NY 10027, USA}
 \email{\href{mailto:khovanov@math.columbia.edu}{khovanov@math.columbia.edu}}
 \author{Nitu Kitchloo}
 \address{Department of Mathematics, Johns Hopkins University, Baltimore, MD 21218, USA}
 \email{\href{mailto:nitu@math.jhu.edu}{nitu@math.jhu.edu}}

\date{April 30, 2020}

\begin{document}
\begin{abstract}
    We consider a deformation of the Robert-Wagner foam evaluation formula, with an eye toward a relation to formal groups. Integrality of the deformed evaluation is established, giving rise to state spaces for planar $GL(N)$ MOY graphs (Murakami-Ohtsuki-Yamada graphs). Skein relations for the deformation are worked out in details in the $GL(2)$ case. These skein relations deform $GL(2)$ foam relations of Beliakova, Hogancamp, Putyra and Wehrli. We establish the Reidemeister move invariance of the resulting chain complexes assigned to link diagrams, giving us a link homology theory.
\end{abstract}

\def\R{\mathbb R}
\def\Q{\mathbb Q}
\def\Z{\mathbb Z}
\def\N{\mathbb N}
\def\C{\mathbb C}
\def\CP{\mathbb P}
\renewcommand\SS{\ensuremath{\mathbb{S}}}
\def\l{\lbrace}
\def\r{\rbrace}
\def\o{\otimes}
\def\lra{\longrightarrow}
\def\Hom{\mathrm{Hom}}
\def\Id{\mathrm{Id}}
\def\mc{\mathcal}
\def\mf{\mathfrak}
\def\Ext{\mathrm{Ext}}
\def\MU{\mathrm{MU}}
\def\U{\mathrm{U}}
\def\U{\mathrm{U}}
\def\EU{\mathrm{EU}}
\def\BU{\mathrm{BU}}
\def\mfgl{\mathfrak{gl}}
\def\mfglN{\mathfrak{gl}_N}
\def\Foamt{\mathrm{Foam}_2}
\def\Free{\mathrm{Fr}}

\def\lra{\longrightarrow}
\def\kk{\mathbf{k}}  
\def\kR{\widetilde{R}} 
\def\lF{\langle F\rangle}  
\def\lG{\langle G\rangle} 
\def\gdim{\mathrm{gdim}}  
\def\rk{\mathrm{rk}}
\def\wchi{\widetilde{\chi}}
\def\wR{\widetilde{R}}
\def\ovmu{\overline{\mu}}
\def\Rminus{R^{-}}
\def\Rminuslong{\Z[\rho_0, \rho_1^{\pm 1}]\llbracket E_1, E_2 \rrbracket}
\def\ovR{\overline{R}}

\newcommand{\brak}[1]{\ensuremath{\left\langle #1\right\rangle}}
\newcommand{\oplusop}[1]{{\mathop{\oplus}\limits_{#1}}}
\newcommand{\ang}[1]{\langle #1 \rangle_{RW} } 
\newcommand{\angf}[1]{\langle #1 \rangle } . 
\newcommand{\pseries}[1]{\kk\llbracket #1 \rrbracket}
\newcommand{\addfigure}{\vspace{0.1in} \begin{center} {\color{red} ADD FIGURE} \end{center} \vspace{0.1in} }
\newcommand{\add}[1]{\vspace{0.1in} \begin{center} {\color{red} ADD FIGURE #1} \end{center} \vspace{0.1in} }

\def\NK#1{{\color{green}[NK: #1]}}
\def\MK#1{{\color{red}[MK: #1]}}%

\maketitle
\tableofcontents

%
%

\section{Introduction}
\label{sec:introduction}


\subsection{MOY graphs and quantum invariants for level one representation}

Foams are 2-dimensional combinatorial CW-complexes, often with extra decorations, embedded in $\R^3$. They naturally appear~\cite{Kh2,KRo2,MV1,MSV,QR,RWd} in the study of link homology theories that categorify quantum $\mathfrak{sl}_N$ or $\mfglN$ link invariants for level one representations when $N\ge 3$.

Reshetikhin-Turaev-Witten invariants~\cite{RT,W} of oriented links $L$ in the 3-sphere $\SS^3$ depend on the choice of a simple Lie algebra $\mf{g}$ and an irreducible representation of $\mf{g}$ associated to each component of $L$. When $\mf{g}=\mf{sl}_N$ and the  components are labelled by level one representations of $\mf{sl}_N$, the Reshetikhin-Turaev-Witten invariant $P(L)\in\Z[q,q^{-1}]$ can be written~\cite{MOY} as a linear combinations of terms $P(\Gamma)\in \Z_+[q,q^{-1}]$ over trivalent oriented planar graphs $\Gamma$ with edges labelled by integers between $1$ to $N$. $P(\Gamma)$ is known as the Murakami-Ohtsuki-Yamada or \emph{MOY invariant} of $\Gamma$.

An edge labelled $a$ corresponds to the identity intertwiner of $\Lambda_q^a V$,
the latter a quantum group representation which $q$-deforms
the $a$-th exterior power of the fundamental representation of the Lie algebra $\mf{sl}_N$.
At this point it's convenient to shift from $\mf{sl}_N$ to $\mfglN$, and view $\Lambda_q^a V$ as a representation of
$U_q(\mfglN)$ rather than that of $\mf{sl}_N$. This change will be more essential at the categorified level of homological invariants rather than for uncategorified quantum invariants, taking values in $\Z[q,q^{-1}]$.

Oriented labelled graphs $\Gamma$ are built out of trivalent vertices that correspond to suitably scaled inclusion and projection of $\Lambda_q^{a+b}V$ into and out of the tensor product $\Lambda_q^a V\otimes \Lambda_q^b V$, see Figure~\ref{fig:MOY_graphs}.

\begin{figure}[h]
\begin{center}
\includegraphics[scale=0.60]{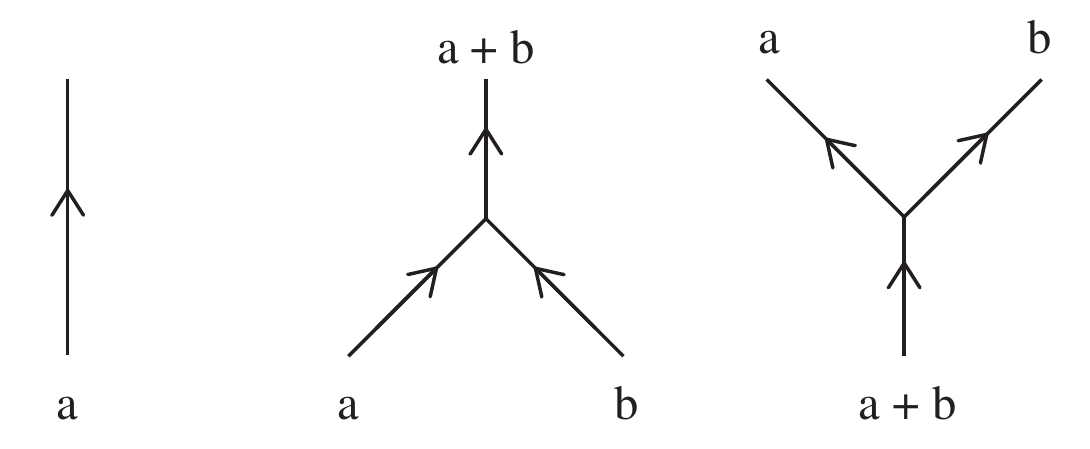}
\caption{\label{fig:MOY_graphs} Generating diagrams for $GL(N)$ MOY graphs. They correspond to the identity intertwiner on $\Lambda_q^a V$ and  projection and inclusion (up to scaling) between $\Lambda_q^a V\otimes \Lambda_q^b V $ and $\Lambda_q^{a+b}V$. }
\end{center}
\end{figure}
Quantum $\mfglN$ (or MOY) invariant of $\Gamma$ is given by a suitable convolution of these maps, which for closed graphs $\Gamma$ results in a Laurent polynomial $P(\Gamma)\in \Z[q,q^{-1}]$ with nonnegative coefficients, see~\cite{MOY} for integrality and~\cite[Appendix 2A]{RW2} for nonnegativity via a suitable state sum formula. Planar graph invariant $P(\Gamma)$ can be computed either via a state sum formula or inductively via skein relations.

As we mention earlier, $\Z[q,q^{-1}]$-linear combinations
of invariants $P(\Gamma)$ give quantum link invariants $P(L)$, when $\mf{g}=\mathfrak{gl}_N$ and components of $L$ are labelled by level one representations, that is, by $\Lambda^a_q V$, over different $a$'s.

The reason for the popularity of this specialization (from $\mf{g}$ to $\mfglN$ and to level one representations), especially with an eye towards categorification, is the relative simplicity of these formulas compared to the case of general $\mf{g}$ and its representations, where canonical choices of intertwiners associated to graph's vertices are harder to guess, spaces of these intertwiners may be  more than one-dimensional,  decomposition of a crossing into a linear combinations of planar graphs has more complicated coefficients or may be difficult to select, and evaluations of $P(\Gamma)$ lose positivity, acquire denominators and live in $\Q(q)$ rather than $\Z_+[q,q^{-1}]$. Any such complication makes  categorical lifting noticeably harder. An approach to categorification of the  Reshetikhin-Turaev-Witten link invariants for an arbitrary $\mf{g}$ and arbitrary representations has been developed by  Webster~\cite{We}. It's an open problem to find a foam-like interpretation of Webster link homology theories and refine them to achieve functoriality under link cobordisms.


\subsection{Foams and Robert-Wagner evaluation}

The key property of $P(\Gamma)$ is it having non-negative coefficients, that is, taking values in $\Z_+[q,q^{-1}]$, rather than just in $\Z[q,q^{-1}]$, where link invariants $P(L)$ live. In the lifting of
$P(L)$ to homology groups, state spaces $\brak{\Gamma}$ will be graded, with graded rank (as a free module over the graded ring $R_N$ of symmetric functions, see below) having non-negative coefficients, thus lying in $\Z_+[q,q^{-1}]$, Homology groups $H(L)$ come from complexes of state spaces $\brak{\Gamma}$, built from various resolutions $\Gamma$ of $L$.

Louis-Hadrien Robert and Emmanuel Wagner  discovered a remarkable evaluation formula for $GL(N)$ foams~\cite{RW1}. Their formula leads to a natural construction of homology groups (or state spaces) for each planar trivalent  MOY graph $\Gamma$ as above.

At the categorified level of this story, Robert-Wagner foam evaluation leads to a state space $\brak{\Gamma}$, a
graded module over the ring $R_N=\Z[x_1,\dots, x_N]^{S_N}$  of symmetric polynomials in $x_1, \dots, x_N$ with coefficients in $\Z$. Robert and Wagner prove~\cite{RW1} that the graded $R_N$-module is free
and finitely-generated, of graded rank $P(\Gamma)$.

Thus, graded rank of $R_N$-module $\brak{\Gamma}$  categorifies the quantum $\mfglN$ invariant (the Murakami-Ohtsuki-Yamada invariant) of these planar graphs. Forming suitable complexes out of these state spaces and taking homology groups leads to bigraded homology theories of links that categorify the HOMFLYPT polynomial and its generalizations to other quantum exterior powers of the fundamental representation~\cite{ETW}, see also earlier approaches~\cite{Y,Wu1,Wu2} to categorification of $\mfglN$ link homology with components colored by arbitrary level one representations.

We now recall the details of Robert-Wagner's foam invariant.
A $GL(N)$-foam $F$ is a two-dimensional piecewise-linear compact  $CW$-complex $F$ embedded in $\R^3$. Its facets are oriented in a compatible way and labelled by numbers from $0$ to $N$ called the \emph{thickness} of a facet (facets  of thickness 0 may be removed) with points of three types:
\begin{itemize}
    \item A regular point on a facet of thickness $a$.
    \item A point on a singular edge, which has a neighbourhood homeomorphic to the product of a tripod $T$ and an interval $I$. The three facets must have thickness $a,b,a+b$ respectively. One can think of thickness a,b facets as merging into the thick facet or vice versa, of the facet of thickness $a+b$ splitting into two thinner facets of thickness $a$ and $b$.
    \item A singular vertex where four singular edges meet. The six corners of the foam at the vertex have thickness $a,b,c,a+b,b+c,a+b+c$ respectively.
\end{itemize}
Neighbourhoods of these three types of points are depicted below.

\begin{figure}[h]
\begin{center}
\includegraphics[scale=0.75]{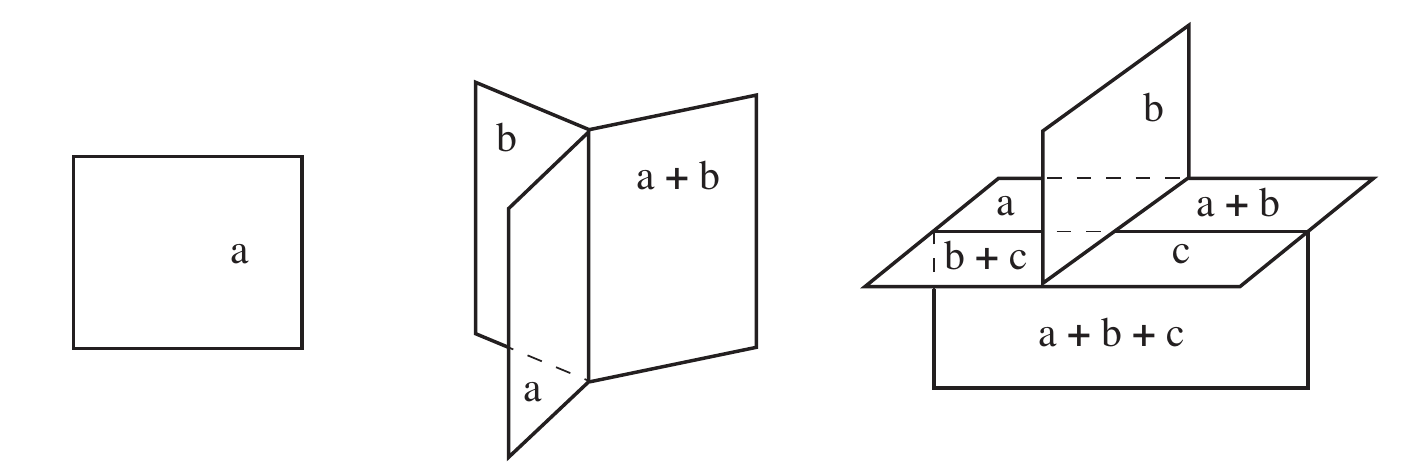}
\caption{Three types of points on a foam}
\end{center}
\end{figure}
Orientations of facets are compatible at singular edges, see Figure~\ref{fig:ori_conv} below.

A singular vertex can be viewed, see Figure~\ref{fig:sing_vertex}, as the singular point of the cobordism between two labelled trees that are the two splittings of an edge of thickness $a+b+c$ into edges of thickness $a,b,c,$ respectively. This is a kind of "associativity" cobordism, which is invertible when viewed as an appropriate module map between state spaces associated to MOY planar graphs in the foam theory.
\begin{figure}[h]
\begin{center}
\includegraphics[scale=0.60]{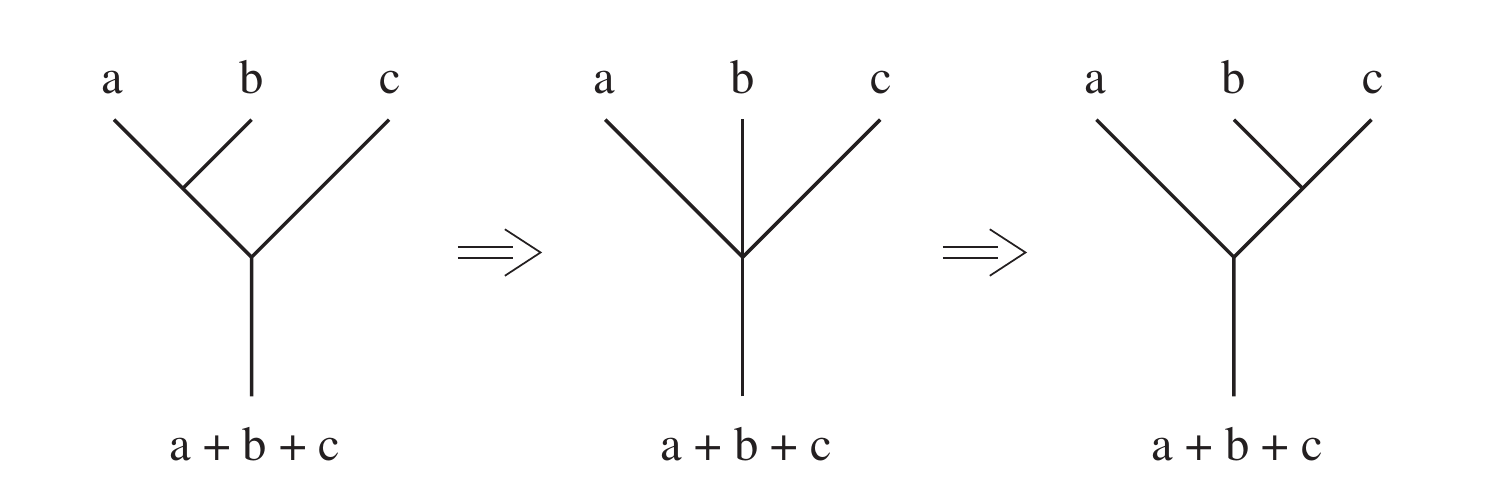}
\caption{\label{fig:sing_vertex} Cross-sections near a singular vertex}
\end{center}
\end{figure}
We follow the orientation conventions from~\cite{ETW}. They show compatible orientations on facets of thickness $a$ and $b$ attached along a singular edge to a facet of thickness $a+b$. The same diagram shows induced orientations on top and bottom boundaries of foam $F$. This convention will be used once we pass from closed foams to foams with boundary, viewed as cobordisms between $GL(N)$ MOY graphs.
\begin{figure}[h]
\begin{center}
\includegraphics[scale=0.7]{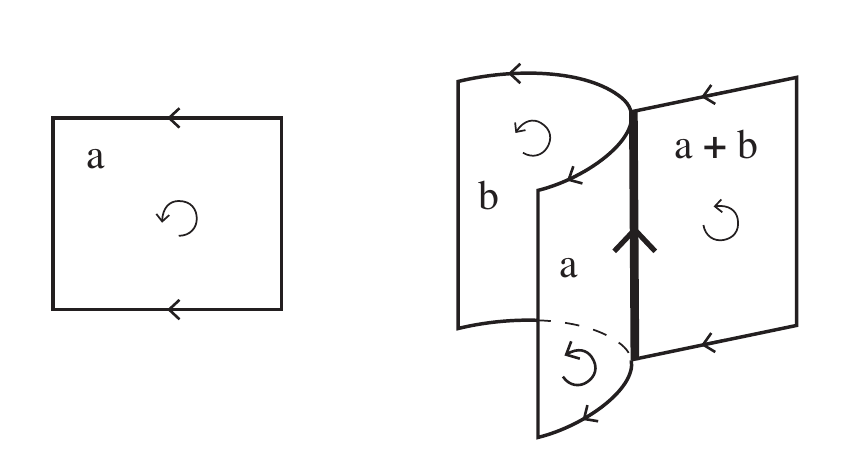}
\caption{\label{fig:ori_conv} Orientation conventions from~\cite[Figure 1]{ETW}. An orientation of a facet induces an orientation of its top boundary (if non-empty, for non-closed foams only) by sticking the first vector of the orientation basis up out of the foam. The remaining vector then induces an orientation of the boundary. For the bottom boundary the resulting orientation is reversed. To induce an orientation on a singular circle, approach it with an orientation basis from a thin facet and point the first vector into the thick facet. The second vector then defines an orientation of the singular circle (or a singular arc, if foam is not closed). This is the one convention we choose out of the four possible conventions for inducing orientations on the boundary and on singular lines, given an orientation of a facet. }
\end{center}
\end{figure}

 Facets $f$ of a foam $F$ are the connected components of the set $F\setminus s(F)$, where $s(F)$ is the set of the singular points of $F$. Thickness of $f$ is denoted $\ell(f)$. The set of facets of $F$ is denoted $f(F)$. A coloring $c$ of $F$ is a map
$c: f(F) \lra 2^{I_N}$ from the set of facets to the set of subsets of $I_N=\{1,\dots, N\}$ such that subset $c(f)$ has cardinality $\ell(f)$ and for any three facets $f_1, f_2, f_3$ attached to a singular edge with $\ell(f_3)=\ell(f_1)+\ell(f_2)$ equality $c(f_3)=c(f_1)\sqcup c(f_2)$ holds. In other words, the subset for $f_3$ is the union of subsets for $f_1$ and $f_2$. A foam may come with decorations (dots). A dot  on a facet $f$ of thickness $a$ represents a homogeneous symmetric polynomial $P_f$ in $a$ variables.

Any coloring $c$ gives rise to closed surfaces $F_i(c)$, $1\le i \le N$,  which are unions of facets $f$ such that $c(f)$ contains $i$.  One also forms symmetric differences $F_{ij}(c)=F_i(c)\Delta F_j(c)$, which are the unions of facets $f$ such that $c(f)$ contains exactly one element of the set $\{i,j\}$. Surfaces $F_{ij}(c)$, $i\not= j$ are closed orientable as well.

Rogert-Wagner evaluation $\ang{F,c}$ of a foam on a coloring $c$ is
\begin{equation}\label{eq:RW1}\ang{F,c} = (-1)^{s(F,c)}\frac{P(F,c)}{Q(F,c)} ,
\end{equation}
where
\begin{eqnarray*}
s(F,c) & = & \theta^+(c) +\sum_{i=1}^N i \chi(F_i(c))/2  \ , \\
\theta^+(c) & = & \sum_{i<j}\theta^+_{ij}(c), \\
P(F,c) & = & \prod_{f\in f(F)} P_{f}(c), \\
Q(F,c) & = & \prod_{1\le i<j\le N}
(x_i - x_j)^{\chi(F_{ij}(c))/2}
\end{eqnarray*}
Here $\theta^+_{ij}(c)$ counts the number of circular seams on the surface $F_{ij}(c)$ along which the cyclic order of the three attached facets (with $i$ but not $j$ in the coloring $c$,
with $j$ but not $i$ in the coloring, and with $i,j$ in the coloring) is one of the two types, called \emph{positive} type. Positivity is determined by the left hand rule with the direction of the thumb along the positive orientation, on turning the fingers of the left hand from the facet with color $i$ to the one with color $j$ for $i < j$.

\begin{figure}[h]
\begin{center}
\includegraphics[scale=0.70]{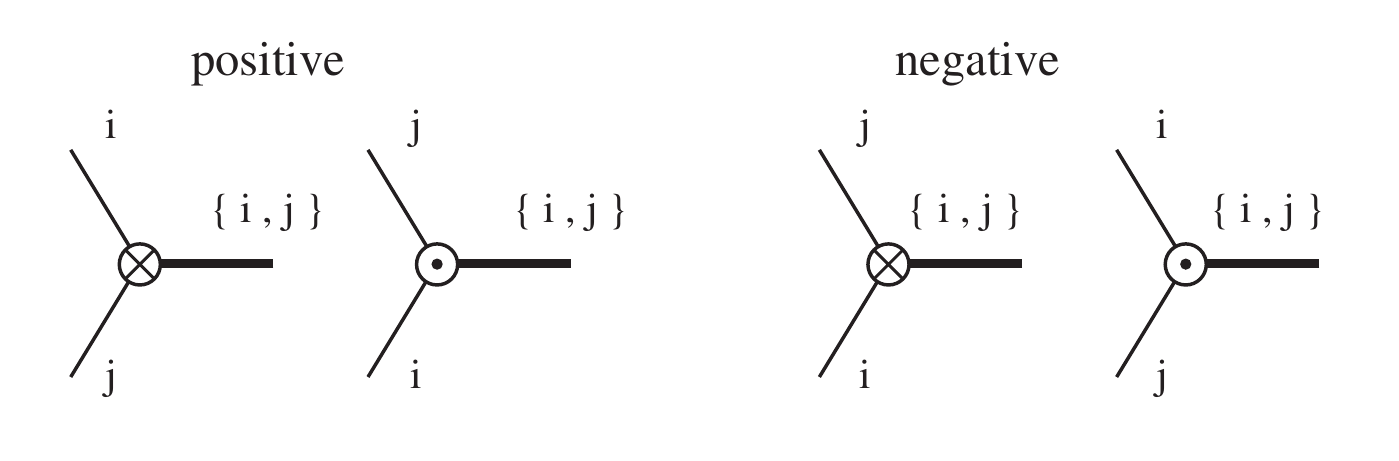}
\caption{\label{fig:sign_contr} $\theta^+_{ij}$ counts the number of positive $(i,j)$-circles, $i<j$}
\end{center}
\end{figure}

$\chi(S)$ denotes the Euler characteristic of the surface $S$.
Term $P_f(c)$ is the product of symmetric functions associated to the dots on the facet $f$, in variables $x_k$, where $k$ runs over the elements of $c(f)$.

Now define
\begin{equation} \label{eq:RW1_sum}
\ang{F} \ = \ \sum_c \ang{F,c},
\end{equation}
the sum over all colorings $c$ of $F$. We refer the readers to~\cite{RW1} for more details on $GL(N)$-foams and their evaluations.

Later in the paper we deform Robert-Wagner evaluation and for the most part work with the deformation. To keep notations light, we use Robert-Wagner's notation $\angf{F}$ to denote deformed evaluation and denote their  original one by $\ang{F}$.

One of the first key results of Robert and Wagner~\cite{RW1} is that $\ang{F}$ is a (symmetric) polynomial rather than just a rational function of $x_1, \dots, x_N$, thus an element of $R_N$.

Generic intersection of a foam $F$ with a plane $\R^2$ in $\R^3$ will result in an oriented planar graph $\Gamma$, which is exactly a $GL(N)$ MOY graph. It is straightforward to introduce foams with boundary.
With the evaluation $\ang{F}$ for closed foams at hand, one can now define the state space $\ang{\Gamma}$ of $\Gamma$ as a graded  $R_N$-module freely generated by symbols $\ang{F}$ of foams $F$ from the empty graph to $\Gamma$, modulo the relations that $\sum_k a_k \ang{F_k}=0$ for foams $F_k$ from $\emptyset$ to $\Gamma$ and $a_k \in R_N$ iff for any foam $G$ from $\Gamma$ to the empty graph $\sum_k a_k \ang{GF_k} =0$. Here $GF_k$ is a closed foam, the gluing or composition of $G$ and $F_k$ along $\Gamma$. Robert-Wagner state spaces (or homology) of graphs $\Gamma$ are then used as building blocks for link homology groups~\cite{ETW}.

As an informal remark, we want to point out that the foams considered in the Robert-Wagner construction~\cite{RW1} should really be called \emph{$GL(N)$-foams}. For \emph{$SL(N)$-foams} one would want to allow seamed edges along which three facets of thickness $a,b,c$ with $a+b+c=N$ or $a+b+c=2N$ meet and allow singular vertices along which such seam edges interact. Robert and Wagner~\cite{RW1} briefly discuss how to extend their evaluation to such foams.

Fundamental applications of Robert-Wagner foam evaluation are developed in~\cite{ETW,RW2,RW3}, with more clearly on the way, see also~\cite{KR,Bo}. Foam evaluation in the limit $N\to\infty$ and restricted to foams in special position provides a connection between foams and Soergel and singular Soergel bimodules~\cite{RW2,KRW}. Other approaches to Soergel and singular Soergel bimodules via foams~\cite{Vz,MV1,QR,RWd,Wd} do not use foam evaluation, utilizing instead matrix factorizations, more direct foam computations in $N=2,3$ cases, and other methods. Earlier, an extension of the Kapustin-Li formula was proposed for foam evaluation~\cite{KRo2}, but due to its more implicit nature was not easy to apply~\cite{MSV}.

Looking beyond foam evaluation, both foams as they are used in link homology and spin foams~\cite{Ba} have "foam" in their names, but we don't know if there is a relation between the two theories beyond this observation. Also see Natanzon~\cite{Nt} and the follow-up papers for yet another direction in the foam theory.


\subsection{Formal groups as a motivation}

In this paper we propose a deformation of the Robert-Wagner evaluation formula, motivated by algebraic topology and generalized cohomology theories related to formal groups. Link homology theories in the $SL(2)$ case have been lifted to spectra by Lipshitz and Sarkar~\cite{LS1,LS2} and Hu, Kriz and Kriz~\cite{HKK}. More recently, a lifting of bigraded $GL(N)$ link homologies as well as the triply-graded homology to  equivariant spectra has been constructed by the second author~\cite{K1,K2,K3}.

Application of generalized cohomology theories to these spectra results in new homological link invariants as well as cohomological operations on them. A purely combinatorial or algebraic description of these homological invariants is clearly desirable, and modifying foam theory and foam evaluation may be a natural first step in this direction.

$GL(N)$ foams are closely related to Grassmannians and partial and full complex flag varieties.
A family of cohomology theories known as complex oriented cohomology theories is related to these varieties as well, and to deformations of the formula
for the first Chern class in singular cohomology of the tensor product of line bundles $c_1(L_1\otimes L_2)=c_1(L_1) + c_1(L_2)$ to formulas
\[ c_1(L_1\otimes L_2) = \mathscr{F}(c_1(L_1), c_1(L_2))
\]
that hold for the first Chern class invariant in these geralized cohomology theories, where
\[ \mathscr{F}(x, y) = x + y + \sum_{i+j>1}  a_{i,j} x^i y^j
\]
is, in general, a power series in $x,y$ with coefficients in the ground ring.  Such a power series admits rich internal structure, making it a {\em Formal group law}. In Section~\ref{sec:formal} we shall study formal group laws in detail, but let us briefly point out some relevant structure in this introduction. Among the relations satisfied by $\mathscr{F}(x,y)$ is the associativity relation, which leads to polynomial relations on $a_{i,j}$ which admit a universal solution with one generator for each $k=i+j-1$, $k\ge 1$. This solution is hard to write down explicitly, and most manipulations with general formal group laws are implicit~\cite{Ha,St} (see section \ref{sec:formal} for examples).

With formal group law $\mathscr{F}(x,y)$ at hand, one defines $-_{\mathscr{F}} x$ or $[-1]x$ as
power series $-x+\dots$ which solves the equation $\mathscr{F}(x,[-1]x)=0$, and forms the power series $x -_{\mathscr{F}} y=\mathscr{F}(x,-_{\mathscr{F}} y)$, also denoted $x[-1]y$:
\[ x[-1]y \  = \ x -_{\mathscr{F}} y \  :=  \ \mathscr{F}(x,-_{\mathscr{F}} y).
\]
This expression deforms $x-y$, so that $x -_{\mathscr{F}} y = x-y + \mathrm{higher \ order \ terms}.$ One can show
$x -_{\mathscr{F}} y = (x-y)q(x,y)$ for an invertible element $q(x,y)$ of a suitable power series ring.  We write $x-y=p(x,y)(x -_{\mathscr{F}} y)$ where $p(x,y)q(x,y)=1,$ and use $p(x,y)=q(x,y)^{-1}$
in our computations.

From the standpoint of algebraic topology, $x-_{\mathscr{F}} y$ represents the Euler class (in the cohomology theory corresponding to the formal group law $\mathscr{F}(x,y)$) of the line bundle $\mathcal{L}_1 \otimes \mathcal{L}_2^\ast$, where $\mathcal{L}_1$ and $\mathcal{L}_2$ represent the tautological line bundles over the product space $\C\CP^{\infty} \times \C\CP^{\infty}$. In other words, the expression $q(x,y)$ should be interpreted as the {\em relative Euler class} for the bundle $\mathcal{L}_1 \otimes \mathcal{L}_2^\ast$, in the sense that one compares the Euler classes in the cohomology theory corresponding to $\mathscr{F}(x,y)$, to the standard Euler class in singular cohomology. In this context, products of the form $q(x_{i_1},x_{j_1})\, q(x_{i_2},x_{j_2})\ldots q(x_{i_k},x_{j_k})$ (which we will come across often in this paper) may be interpreted as the relative Euler class of the {\em direct sum} of the line bundles corresponding to each factor.

\medskip
Robert-Wagner foam evaluation formulas contain powers of $x_i - x_j$ in the denominator, and a natural idea would be to carefully replace them with  $x_i -_{\mathscr{F}} x_j$.  We pursue a variant of this idea in this paper. Similar replacements have already been considered for various formulas in the theory of symmetric functions, including the Weyl formula for the Schur function, see~\cite{NN1,NN2,Na} and references therein. Foam evaluation specializes to the Weyl formula for the Schur function in the case of the so-called theta-foam and its natural generalizations.

\medskip
On the algebraic topology side, the expressions $x_i-x_j$ for $1 \leq i \neq j \leq N$ have a natural meaning as the Euler classes in singular cohomology for the roots $\alpha$ of $GL(N)$ (we have included both positive and negative roots). In particular, the deformation $x_i -_{\mathscr{F}} x_j$ can be interpreted as the Euler class of $\alpha$ in an exotic cohomology theory corresponding to the formal group law $\mathscr{F}(x,y)$. We may therefore speculate that the corresponding deformed foam evaluation formula is obtained by applying an exotic cohomology theory to a (hitherto undefined) homotopy type. The existence of such a homotopy type for foam evaluations is very compelling given the results by the second author \cite{K1,K2,K3}. Since $GL(N)$ has $N(N-1)$ roots $\alpha$ representing the weights $x_i - x_j$ in the standard basis for $1 \leq i \neq j \leq N$, we see that our deformed evaluation formulas will be expressible in terms of $N(N-1)$ parameters given by the relative Euler classes $q(x_i,x_j)$. However, these extra parameters will satisfy certain constraints with coefficients in the algebra of symmetric power series in $N$-variables (which is the $GL(N)$-equivariant cohomology of a point). This suggests that the possible underlying homotopy type for foam evaluations is built from universal bundles using suitable subsets of roots of $GL(N)$.

\medskip
 The discussion above motivates our deformation using the language of formal group laws and related cohomology theories. We go into this further in section \ref{sec:formal}.  Interestingly however, although motivated by it, our deformation setup will end up not requiring all the constraints on the power series $q(x,y)$ imposed by a formal group law. For instance, we will not require associativity from our analogue of the power series $\mathscr{F}(x,y)$. We therefore take as Ansatz, the series $p(x,y)$, the inverse of $q(x,y)$, with arbitrary coefficients.  Most of the information in the coefficients of $q(x,y)$ will turn out to be redundant in our framework, at least in the $GL(2)$ case. However, it is conceivable that one may endow our constructions with the action of cohomology operations which are sensitive to more coefficients in the power series $q(x,y)$.


\subsection{Plan of the paper}
In section~\ref{sec:deformed}, motivated by analogies with formal group laws, we write down a multi-parameter deformation of the Robert-Wagner evaluation of closed $GL(N)$-foams and prove its integrality for any such foam. In section~\ref{sec:formal} we review formal group laws and corresponding generalizations of the divided difference operators. In section~\ref{sec:deform} we specialize to $N=2$ and study this deformation, which ultimately adds two more variables, of the $GL(2)$ foam evaluation. Skein  relations for the deformed $GL(2)$ foam evaluation  are derived in section~\ref{sec:skein_relations}. In section~\ref{sec:prefoams} we work out the ground ring $R$ for the deformed theory, which has four generators $E_1,E_2,$ $\rho_0,\rho_1$ of degrees $2,4,-2,0$, respectively. For comparison, the ground ring for the usual  $GL(2)$-equivariant link homology has generators $E_1, E_2$ (also denoted $h,t$, up to a minus sign).
In section~\ref{subsec:direct_sum_dec} we show, unsurprisingly, that the state spaces (or homology) of planar $GL(2)$ webs are free modules over the graded ring $R$ of rank $(q+q^{-1})^k$ over the ground ring $R$, where $k$ is the number of thin circles in a web. In section~\ref{sec:reidemeister} we extend the state spaces to homology groups of planar link diagrams and  show the invariance under the Reidemeister moves.

Specializing power series $p(x,y)$ to $p(x,y)=1$ recovers the $GL(2)$ foam theory of Beliakova, Hogancamp, Putyra, and Wehrli~\cite{BHPW}. Simplifying computations in our section~\ref{sec:deform} to this case gives a foam evaluation approach to their theory.

$GL(2)$ foam theory that comes from this deformation seems very similar to the $SL(2)$ theory as set up by Vogel~\cite{V} and extended by him to get a strong invariant of tangle cobordisms, without the sign indeterminacy. The relation is given by dropping double facets but remembering singular circles along which the facets attach to the thin surface of a $GL(2)$ foam.   $GL(2)$ foam theory also has an overlap with Ehrig-Stroppel-Tubbenhauer's generic $GL(2)$ foams~\cite{EST1}.

We've already mentioned connections to Clark-Morrison-Walter~\cite{CMW} and Caprau~\cite{Ca1,Ca2} who have achieved full functoriality of the $SL(2)$ link homology via diagrammatical calculi that employ singular circles on thin surfaces. These circles should be remnants of attached double facets. Caprau's theory is $GL(2)$ equivariant, with variables $h$ and $a$ in place of our $E_1$ and $E_2$.

Twisting of $SL(2)$ theories in Vogel~\cite{V} is related to the deformation via power series $p(x,y)$ in this paper. We plan to elucidate connections to Vogel~\cite{V}, Ehrig-Stroppel-Tubbenhauer~\cite{EST1}, and  to Turaev-Turner's rank two Frobenius algebra structures~\cite{TT} in a follow-up paper and also see whether the $p(x,y)$ deformation  corresponds to the  twisting~\cite{Kh4,V} in the $N=2$ and the general case.

\subsection{Acknowledgments}
M.K. was partially supported by NSF grants  DMS-1664240 and DMS-1807425 while working on this paper.
The authors are grateful to Yakov Kononov, Louis-Hadrien Robert and Lev Rozansky for valuable discussions and  would like to thank Elizaveta Babaeva\footnote{Elizaveta Babaeva, https://www.behance.net/lizababaiva} for help with producing figures for the paper.

%
%

\section{Deformed evaluation for $GL(N)$ foams} \label{sec:deformed}

The $GL(N)$ Robert-Wagner formula has denominators of the form $(x_i-x_j)^{\chi_{ij}(c)/2}$, where $\chi_{ij}(c)=\chi(F_{ij}(c))$ is the Euler characteristic of the bicolored surface $F_{ij}(c)$. The expression $x_i-x_j$ can be generalized to $x_i -_{\mathscr{F}} x_j = x_i[-1]x_j$, where $F$ is a formal group law. Unlike the additive case, when $x-y = - (y-x)$, most formal group laws do not satisfy $x [-1] y = - ( y [-1] x)$, while those that do are called \emph{symmetric}.  Converting
$(x_i-x_j)^{\chi_{ij}(c)/2}$ to
$(x_i [-1] x_j)^{\chi_{ij}(c)/2}$ to modify the Robert-Wagner formula may be possible, but it would not contain the opposite terms $x_j [-1] x_i$, that perhaps should be present to maintain some symmetry, despite us having fixed a set of positive roots $\{x_i-x_j\}_{i<j}.$

To distribute the exponent $\chi_{ij}(c)/2$ across both terms $x_i -_{\mathscr{F}} x_j$ and $x_j -_{\mathscr{F}} x_i$, we recall the relation~\cite[Lemma 2.7]{RW1} on Euler characteristics
\begin{equation} \label{eq-euler-char}
    \chi(F_{ij}(c)) \ = \ \chi(F_i(c)) + \chi(F_j(c)) - 2 \chi(F_{i\cap j}(c)),
\end{equation}
where $F_{i\cap j}(c)$ is the surface, possibly with boundary, consisting of the union of facets that contain colors $i$ and $j$,
$$ F_{i\cap j}(c) = F_i(c) \cap F_j(c).$$
The Euler characteristic $\chi(F_{i\cap j}(c))$ may be odd, due to the presence of boundary, but $2 \chi(F_{i\cap j}(c))$ is even, as are the other three terms in the formula (since the other three are Euler characteristics of closed surfaces).

Formula (\ref{eq-euler-char}) simply describes the Euler characteristic of the symmetric difference of two spaces, specialized to the case of surfaces $F_i(c)$ and $F_j(c)$ inside a foam. Using shorthand notations, we can rewrite it as
\begin{equation} \label{eq-euler-char-s}
    \chi_{ij}(c) \ = \ \chi_i(c) + \chi_j(c) - 2 \chi_{i\cap j}(c),
\end{equation}
where $\chi_{i\cap j}(c)= \chi(F_{i\cap j}(c))$.
We can now modify the evaluation formula by
changing the $(i,j)$ color pair contribution to the denominator to
$$(x_i [-1] x_j)^{\chi_i(c)/2} (x_j [-1] x_i)^{\chi_j(c)/2}$$
and multiplying the numerator by $(x_i-x_j)^{\chi_{i\cap j}(c)},$ to
define

\begin{equation} \label{eq:deformed_glN_1}
\angf{F,c} = (-1)^{\theta^+(c)}\prod_{f\in f(F)}{P_f(c)} \prod_{i<j}
\frac{(x_i - x_j)^{\chi_{i\cap j}(c)}}{(x_i [-1] x_j)^{\chi_i(c)/2} (x_j [-1] x_i)^{\chi_j(c)/2}}
\end{equation}
We denote modified evaluation by $\angf{F,c}$ and the original one in~\cite{RW1} by $\ang{F,c}$.

\medskip
Setting aside formal group laws at this point, let us now formally {\em define} $x[-1]y$ as follows.  Choose a commutative graded ring $\kk$ and homogeneous elements $\beta_{k,\ell}\in \kk$ in degree $-2(k+\ell)$ for all $k,\ell\in \Z_+=\{0,1,2, \dots \}$ such that $(k,\ell)\not= (0,0).$ The element
\begin{equation} \label{eq_p_series}
    p(x,y) \ = \ 1 + \sum_{(k,\ell)\not= (0,0)}
    \beta_{k,\ell} x^k y^{\ell}
\end{equation}
belongs to the power sum ring $\kk\llbracket x,y\rrbracket$. In general, $\beta_{k,\ell}\not= \beta_{\ell,k}.$ The element $p(x,y)$ has the inverse $q(x,y) = p^{-1}(x,y)\in \pseries{x,y}.$ Define
\begin{equation}
    x [-1] y \ = \ q(x,y)(x-y)  \ = \ p(x,y)^{-1}(x-y) \in \pseries{x,y}.
\end{equation}
Equivalently,
$x-y=p(x,y)(x [-1] y).$
Denote $p_{i,j}=p(x_i,x_j)$ or, interchangeably, $p_{ij}$, and $q_{ij}=q(x_i,x_j)=p_{ij}^{-1}$. Then
\begin{equation} \label{eq_minus_F}
   x_i [-1] x_j \ = \ (x_i - x_j) p_{ij}^{-1}.
\end{equation}

Note that $x [-1] y = - (y [-1] x)$ iff $p(x,y)=p(y,x)$ iff  $\beta_{\ell,k}=\beta_{k,\ell}$ for all $k,\ell$. We refer to this as the \emph{symmetric} case.

The universal case is that of the ring
\begin{equation}\label{eq_ring_kk}
\kk \ = \  \Z[\beta_{k,\ell}]
\end{equation}
over all $k,\ell$ as above ($k,\ell \in \Z+, (k,\ell)\not= (0,0)$). This ring is non-positively graded, with nontrivial homogeneous components in
even non-positive degrees $0, -2, -4, \dots$.
It's a graded polynomial ring with $k+1$ generators in degree $-2k$ over all $k\ge 1$.
The universal symmetric case is when $\beta_{\ell,k}=\beta_{k,\ell}$ are formal variables over all $0\le k \le \ell, (k,\ell)\not= (0,0)$.

Convert denominators in (\ref{eq:deformed_glN_1}) via (\ref{eq_minus_F}) and combine with a power of $x_i-x_j$ in the numerator to get
\begin{eqnarray*}
&  & \frac{(x_i - x_j)^{\chi_{i\cap j}(c)}}{(x_i [-1] x_j)^{\chi_i(c)/2} (x_j [-1] x_i)^{\chi_j(c)/2}} =
\frac{(x_i - x_j)^{\chi_{i\cap j}(c)}}{(x_i - x_j)^{\chi_i(c)/2}q_{ij}^{\chi_i(c)/2} (x_j - x_i)^{\chi_j(c)/2}q_{ji}^{\chi_j(c)/2}} =  \\
& & \frac{(-1)^{\chi_j(c)/2}p_{ij}^{\chi_i(c)/2} p_{ji}^{\chi_j(c)/2}}{(x_i-x_j)^{\chi_{ij}(c)/2}}
\end{eqnarray*}
Taking the product over all $1\le i < j \le N$, the minus signs will combine to $(-1)^{\sum_{j=1}^N(j-1)\chi_j(c)/2}$ and
\begin{equation} \label{eq:new_eval_1}
    \angf{F,c} = (-1)^{s'(F,c)}\prod_{f\in f(F)}{P_f(c)} \prod_{i<j}
\frac{p_{ij}^{\chi_i(c)/2} p_{ji}^{\chi_j(c)/2}}{(x_i - x_j)^{\chi_{ij}(c)/2}} ,
\end{equation}
where
\begin{equation}\label{eq:s_prime}
    s'(F,c) = \theta^+(c) + \sum_{j=1}^N (j-1)\chi_j(c)/2.
\end{equation}
We can then define, as before,
\begin{equation} \label{eq:new_eval_2}
\angf{F} = \sum_{c} \angf{F,c}.
\end{equation}
In the Robert-Wagner formula, due to $\chi_i(c)=\chi(F_i(c))$ being even, one can rewrite the sign as
\begin{equation}
(-1)^{s(F,c)} = (-1)^{\theta^+(c)+
\bigl( \sum^N_{i=1, i \  \mathrm{ odd}}\chi_i(c) \bigr) / 2 }
\end{equation}
versus
\begin{equation}
(-1)^{s'(F,c)} = (-1)^{\theta^+(c)+
\bigl( \sum^N_{i=1, i \  \mathrm{ even}}\chi_i(c) \bigr) / 2 }
\end{equation}
in (\ref{eq:s_prime}).
These two signs differ by
$(-1)^{\bigl( \sum^N_{i=1}\chi_i(c) \bigr) / 2} . $

The sum $\sum^N_{i=1}\chi_i(c)$ does not depend on the coloring $c$ of $F$ and can be computed as a sort of the Euler characteristic of $F$, denoted $\chi(F)$. An open facet $f$ of thickness $a$ contributes $a\chi(f)$, an edge where facets of thickness $a,b,a+b$ meet contributes $-(a+b)$, a vertex along which facets of thickness $a,b,c$ merge in two ways into a facet of thickness $a+b+c$ contributes $a+b+c$. In each case, a generalized 0-, 1- or 2-cell of a foam contributes its Euler characteristic times its thickness, defined as the number of surfaces $F_i(c)$, over all $i$, that contain that generalized cell.
Consequently, the sign difference in the two evaluations is by $(-1)^{\chi(F)/2}$, with
\begin{equation}
    \chi(F) = \sum_{i=1}^N \chi(F_i(c)), \ \ \mathrm{for \ any \ coloring \ } c .
\end{equation}

Thus, to recover the Robert-Wagner evaluation $\ang{F}$ from (\ref{eq:new_eval_1}) and (\ref{eq:new_eval_2}) one should specialize $p(x,y)=1$, so that $p_{ij}=1$ for all $i\not= j$ and scale by the sign,
\begin{equation}
    \ang{F} = (-1)^{\chi(F)/2} \angf{F}|_{p(x,y)=1}
\end{equation}
We keep the sign term $(-1)^{\chi(F)/2}$ so that, in $N=2$ case, the 2-sphere of thickness one carrying a single dot would evaluate to $1$ (upon specializing to $p_{ij}=1$), as in~\cite{BHPW}, rather than $-1$, as in~\cite{RW1}. Adding this sign term is a matter of preference, while we hope that the deformation via $p(x,y)$ will eventually prove significant.

\medskip
\noindent
Let us now show that the formula (\ref{eq:new_eval_2}) given by summing the expressions (\ref{eq:new_eval_1}) over all colorings gives rise to a symmetric power series that does not involve denominators. We begin with a simple lemma:

\begin{lemma} \label{Foam invariance}
Given a coloring $c$, let $p(c)$ denote the expression
\[ p(c) =  \prod_{1 \leq i < j \leq N} p_{ij}^{\chi_i(c)/2}p_{ji}^{\chi_j(c)/2}. \]
Let $c'$ be a coloring obtained from $c$ by a Kempe move relative to 1 and 2 along a connected surface $\Sigma_s$, then one has a relation
\[ p(c') = p(c) \, p_{1,2,s}, \]
where $p_{1,2,s}$ is the expression
\[ p_{1,2,s} := \frac{p_{12}^{\chi(F_2 \cap \Sigma_s)/2 - \chi(F_1 \cap \Sigma_s)/2}}{p_{21}^{\chi(F_2 \cap \Sigma_s)/2-\chi(F_1 \cap \Sigma_s)/2}} \prod_{2 < j \leq N} \frac{p_{1j}^{\chi(F_2 \cap \Sigma_s)/2 - \chi(F_1 \cap \Sigma_s)/2}}{p_{2j}^{\chi(F_2 \cap \Sigma_s)/2-\chi(F_1 \cap \Sigma_s)/2}}. \]
Notice, in particular, that $p_{1,2,s}$ is an invertible power series starting with $1$, and the transposition $\sigma$ that switches the variables $x_1$ and $x_2$ has the property
\[ \sigma(p_{1,2,s}) = \frac{1}{p_{1,2,s}}. \]
Furthermore, $p_{1,2,s}$ is of the form $p_{1,2,s} = 1 \mod (x_1 - x_2)$.
\end{lemma}

\begin{proof}
The relation between $p(c')$ and $p(c)$ is straightforward to verify, and is left to the reader. Now, since $p_{1,2,s}$ is an invertible power series, we may consider the expression
\[ p_{1,2,s} - \frac{1}{p_{1,2,s}} = \frac{1}{p_{1,2,s}}((p_{1,2,s})^2 - 1). \]
The above expression switches sign under the action of the transposition $\sigma$, and is therefore divisible by $(x_1-x_2)$. We conclude that $(x_1-x_2)$ divides the expression $(p_{1,2,s})^2-1$. Factoring this expression, se see that $(x_1-x_2)$ must divide $p_{1,2,s}-1$.
\end{proof}

\begin{remark}\label{deform_locality}
Notice that the definition of $p_{1,2,s}$ as a ratio of $p(c')$ and $p(c)$ can be extended to the case when $c'$ is a coloring obtained from $c$ by a $(1,2)$-Kempe move along several connected components \[ \Sigma_{\underline{s}} := \Sigma_{s_1} \sqcup \ldots \sqcup \Sigma_{s_k}. \]
This expression, $p_{1,2,\underline{s}}$ satisfies a locality property that is crucial for the following theorem
\[ p_{1,2,\underline{s}} = p_{1,2,s_1} \ldots p_{1,2,s_k}, \]
where $p_{1,2,s_i}$ denotes the ratio of $p(c'_i)$ and $p(c)$, with $c'_i$ obtained from $c$ by a $(1,2)$-Kempe move on $\Sigma_{s_i}$.
\end{remark}
\medskip
\noindent
The above lemma allows us to prove

\medskip
\begin{theorem} \label{Foam_denominator}
The $GL(N)$-foam evaluation $\angf{F}$ is a symmetric power series in the variables $x_1, x_2, \ldots, x_N$. In particular, $\angf{F}$ is free of denominators.
\end{theorem}
\begin{proof}
The proof of the above theorem is essentially a simple variation on the argument given in \cite[Proposition 2.18]{RW1}. Consider the expression $\angf{F}$. It is clear that it is symmetric in the variables $x_1, x_2, \ldots, x_N$ with possible denominators of the form $(x_i-x_j)^k$. By symmetry, the proof of the theorem will follow if we can show that the denominator $(x_1-x_2)$ does not appear in $\angf{F}$.

\medskip
Let us decompose the set of colorings of $\angf{F}$ into a collection of equivalence classes relative to the colors $1$ and $2$. Given a coloring $c$ of $F$, decompose $F_{12}(c)$ into connected components,
\[ F_{12}(c) = {\bf \Sigma} = \Sigma_1 \cup \Sigma_2 \cup \ldots \cup \Sigma_r. \]
The equivalence class of colorings $C_c$ that contains $c$ consists of colorings of $F$ that can be obtained from $c$ by performing Kempe moves about various connected components $\Sigma_s \subseteq F_{12}(c)$, where we recall that a Kempe move about $\Sigma_s$ switches the colors $1$ and $2$ of the facets in $\Sigma_s$.

As in \cite[Proposition 2.18]{RW1}, consider the expressions
\[ P_{F/{\bf \Sigma}}(F,c) = p(c) \prod_{f \,\,  \mbox{not a facet in} \, {\bf \Sigma}} P(c(f)), \quad \quad \widetilde{Q}(F,c) = \frac{Q(F,c) \prod_{s, k>2}(x_1-x_k)^{l_{\Sigma_s}(c,k)/2}}{(x_1-x_2)^{\chi_{12}(c)/2}} \]
where $p(c)$ is as defined in Lemma \ref{Foam invariance}, and the integers $l_{\Sigma_s}(c,k)$ are as defined in \cite[Lemma 2.10]{RW1}. Also, for $1 \leq s \leq r$, define
\[ T_s(F,c) = \widehat{P}_{\Sigma_s}(F,c) + (-1)^{\chi(\Sigma_s)/2} \, p_{1,2,s} \, \sigma(\widehat{P}_{\Sigma_s}(F,c)), \]
where $\sigma$ is the transposition that swaps $x_1$ and $x_2$, the term $p_{1,2,s}$ is as defined in Lemma \ref{Foam invariance}, and the expression $\widehat{P}_{\Sigma_s}(F,c)$ is defined as
\[ \widehat{P}_{\Sigma_s}(F,c) = \prod_{f \, \, \mbox{a facet in} \, \Sigma_s} P(c(f)) \, \prod_{2 < k \leq N} (x_1-x_k)^{l_{\Sigma_s}(c,k)/2}.\]
Using Lemma \ref{Foam invariance} and Remark \ref{deform_locality}, we may express the foam evaluation $\angf{F,c}$ on summing over the equivalence class $C_c$ as
\[ \sum_{c' \in C_c} \angf{F,c'} = (-1)^{s'(F,c)} \frac{P_{F/{\bf \Sigma}}(F,c)}{\widetilde{Q}(F,c)} \prod_{s=1}^r (x_1-x_2)^{-\chi(\Sigma_s)/2}T_s(F,c). \]
Since $\widetilde{Q}(F,c)$ is not divisible by $x_1-x_2$, it is enough for our purposes to show that the expression $(x_1-x_2)^{-\chi(\Sigma_s)/2}T_s(F,c)$ does not have a denominator given by a power of $x_1-x_2$. The only case that is relevant is when $\Sigma_s$ is a surface of genus zero. It is therefore sufficient to show that $T_s(F,c)$ is divisible by $(x_1-x_2)$ when $\Sigma_s$ is a surface of genus zero. In this case, we have
\[ T_s(F,c) = \widehat{P}_{\Sigma_s}(F,c) - \, p_{1,2,s} \, \sigma(\widehat{P}_{\Sigma_s}(F,c)). \]
By Lemma \ref{Foam invariance}, recall that $p_{1,2,s}$ is of the form $1 \mod (x_1-x_2)$. We therefore have
\[ T_s(F,c) = \widehat{P}_{\Sigma_s}(F,c) - \sigma(\widehat{P}_{\Sigma_s}(F,c)) \mod (x_1-x_2). \]
However, the expression $\widehat{P}_{\Sigma_s}(F,c) - \sigma(\widehat{P}_{\Sigma_s}(F,c))$ is also divisible by $x_1 - x_2$ since it switches sign under $\sigma$. It follows that $T_s(F,c)$ is divisible by $x_1-x_2$ whenever $\Sigma_s$ is a surface of genus zero. The proof of the theorem easily follows on summing $\angf{F,c'}$ over all the equivalence classes $C_c$.
\end{proof}

%
%

\section{Formal groups and generalized divided difference operators}
\label{sec:formal}

In this section we study formal group laws and their relationship to topology in some detail. Good references are \cite{Ha, St} and the references therein.
Due to standard conventions the choice of notation $R$ in this section conflicts with its use in the next section.

Let us begin by recalling the definition of a formal group law. A formal group law defined over a ring $R$ is a power series $\mathscr{F}(x,y)$ with coefficients in $R$ so that $\mathscr{F}(x,y)$ represents a commutative group structure on the formal affine line over $R$. In other words, one requires $\mathscr{F}(x,y)$ to satisfy the following three properties
\[\mathscr{F}(x,y) = \mathscr{F}(y,x) \quad \quad \mbox{commutativity} \]
\[ \mathscr{F}(0,x) = \mathscr{F}(x,0) = x \quad \quad \mbox{unitarity} \]
\[ \mathscr{F}(x,\mathscr{F}(y,z)) = \mathscr{F}(\mathscr{F}(x,y),z) \quad \quad \mbox{associtivity}. \]

\medskip
\begin{remark}
There is a universal ring known as the Lazard ring which is initial among all rings that support a formal group law. This ring $L$ can be defined to be generated by symbols $a_{i,j}$ where the universal formal group law has the form
\[\mathscr{F}(x,y) = x + y + \sum_{i,j > 0} a_{i,j} x^i y^j. \]
We then impose relations on the generators $a_{i,j}$ that are forced by the relations of commutativity and associativity (the relation for unitarity is built into the form of $\mathscr{F}(x,y)$). For instance, commutativity implies that $a_{i,j} = a_{j,i}$. The relation for associtivity is clearly more involved.
\end{remark}

\medskip
In topology, formal group laws appear when one describes the $E$-cohomolgy of a space $\BU(1)$, where $E$ is any complex oriented cohomology theory and $\BU(1)$ denotes the classifying space of the group $\U(1)$ (the space $\BU(1)$ is equivalent to the infinite projective plane $\C\CP^{\infty}$). More precisely, one starts with the observation that $E^*(\BU(1))$ can be expressed as $R\llbracket x \rrbracket$, with $R = E^*(pt)$ and $x$ being the first Chern class in cohomological degree 2. The abelian group structure on $\U(1)$ induces a map
\[ \BU(1) \times \BU(1) \longrightarrow \BU(1). \]
Evaluating this map in $E$-cohomology then gives rise to the underlying a formal group law $\mathscr{F}_E(x,y)$ for the complex oriented cohomology theory $E$:
\[ R\llbracket x \rrbracket = E^*(\BU(1)) \longrightarrow E^*(\BU(1) \times \BU(1)) = R\llbracket x, y \rrbracket, \quad \quad x \longmapsto \mathscr{F}_E(x,y). \]
In what follows therefore, we work in the graded setting. So $\mathscr{F}(x,y)$ will denote a formal group law over a graded power series ring $R\llbracket x,y \rrbracket$, where $R$ is a graded $\Z$-algebra, and the variables $x$ and $y$ are defined to have degree 2. We assume that $\mathscr{F}(x,y)$ is in homogeneous degree 2, namely
\[ \mathscr{F}(x,y) = x + y + \sum_{i,j > 0} a_{ij} x^i y^j, \quad \quad a_{ij} \in R^{2-2(i+j)}. \]

\medskip
\begin{definition}
The formal negative of the variable $x$ is defined to be the (unique) power series $[-1] x$ with the property
\[ \mathscr{F}(x, [-1] x) = \mathscr{F}([-1] x,  x) = 0. \]
The formal difference $x [-1] y$ is defined as $\mathscr{F}(x, [-1] y)$. It is a power series in two variables $x,y$ that has homogeneous degree 2.
\end{definition}

\medskip
\begin{example} \label{1}
Let $R$ be the $\Z$-algebra $\Z[\beta]$, where $\beta$ is in degree $-2$. The multiplicative formal group law $\mathscr{F}(x,y)$ and its formal difference is given by
\[ \mathscr{F}(x,y) = x + y - \beta xy, \quad \quad x [-1] y = \frac{x-y}{1-\beta y}, \quad \quad [-1](x) = -\sum_{i \geq 0} \beta^i x^{i+1}.\]
\end{example}

\medskip
\begin{example}\label{1b}
Let $R$ be the $\Z$-algebra $\Z[\beta^2]$ as before with $\beta^2$ in degree $-4$. The Lorentz or L-formal group law $\mathscr{F}(x,y)$ and its formal difference is given by
\[ \mathscr{F}(x,y) = \frac{x + y}{1+ \beta^2 xy}, \quad \quad x [-1] y = \frac{x-y}{1-\beta^2 xy}, \quad \quad [-1](x) = -x.\]
\end{example}

\medskip
\begin{example}\label{1c}
Let $R$ be the $\Z$-algebra $\Z[\frac{1}{2}, \beta^2]$ with $\beta^2$ in degree $-4$. The $\hat{A}$-formal group law $\mathscr{F}(x,y)$ and its formal difference is given by
\[ \mathscr{F}(x,y) = x\sqrt{1 + \beta^2 (y/2)^2} \, + \, y\sqrt{1+ \beta^2 (x/2)^2}, \]
\smallskip
\[ x [-1] y = x\sqrt{1 + \beta^2 (y/2)^2} \, - \, y\sqrt{1+ \beta^2 (x/2)^2}, \quad \quad [-1](x) = -x.\]
where the radicals are expressed as a power series (with coefficients in $\Z[\frac{1}{2}, \beta^2]$) by the formal application of the binomial expansion.
\end{example}

\medskip
\begin{example} \label{1d}
Let $R$ be the $\Z$-algebra $\Z[\frac{1}{2}, \epsilon, \delta]$ with the degree of $\delta$ being $-4$ and that of $\epsilon$ being $-8$. The Jacobi formal group law $\mathscr{F}(x,y)$ and its formal inverse is given by
\[ \mathscr{F}(x,y) = \frac{x\sqrt{J(y)} + y\sqrt{J(x)}}{1-\epsilon x^2 y^2}, \quad \mbox{where} \quad J(z) = 1 - 2\delta z^2 + \epsilon z^4. \]

\[ x [-1] y = \frac{x\sqrt{J(y)} - y\sqrt{J(x)}}{1-\epsilon x^2 y^2}, \quad \quad [-1](x) = -x.\]
Examples \ref{1b} and \ref{1c} are specializations of \ref{1d} at the ``cusps" described by $\epsilon = \beta^4, \delta = \beta^2$ and $\epsilon = 0, \delta = -\beta^2/8$ respectively.
\end{example}

\medskip
Let us return to the universal example. In other words, we consider the example of $R$ being the Lazard ring introduced earlier. On introducing a grading on the variables $x$ and $y$ so that the universal formal group law belongs in homogeneous degree 2, the Lazard ring naturally acquires a grading as described earlier. With this grading, the Lazard ring can be shown to be isomorphic to the graded coefficient ring of a complex oriented cohomology theory known as complex cobordism, $\MU$. In other words $R \cong \MU^\ast(pt)$ as a graded ring.

\medskip
By the definition of complex cobordism, the elements of $\MU^{-k}(pt)$ are cobordism classes of $k$-dimensional manifolds endowed with an almost complex structure on their stable normal bundle, and with the ring structure being induced by the cartesian product of manifolds. The ring $\MU^\ast(pt)$ can be shown to be a polynomial algebra over $\Z$, with one generator in each negative even degree. Working rationally, the generator in degree $-2n$ may be chosen to be the cobordism class of the complex projective space of dimension $2n$, denoted by $[\C \CP^n]$.

\medskip
Any formal group law over a $\Q$-algebra is isomorphic to the additive formal group law. This isomorphism is called the logarithm, written as $\log_\mathscr{F}(x)$, and is the unique power series with leading term being $x$, that interpolates the given formal group law $\mathscr{F}(x,y)$ with the additive one $\mathbb{G}_a(x,y) = x+y$.

\medskip
On extending scalars from $\Z$ to $\Q$, the logarithm in the universal case has an explicit description
\[ \log_{\MU} (x) = \sum_{k \geq 0} \frac{[\C\CP^k]}{k+1} x^{k+1}, \quad \quad \mbox{so that} \quad \quad \log_{\MU}(\mathscr{F}_{\MU}(x,y)) = \log_{\MU}(x) + \log_{\MU}(y) \]
An immediate corollary of the above description is the following example

\medskip
\begin{example}\label{1e}
Let $R$ be the $\Z$-algebra $\MU^\ast(pt)$. The universal formal group law $\mathscr{F}_{\MU}(x,y)$ and its formal difference is given by
\[ \mathscr{F}_{\MU}(x,y) =  \exp_{\MU}(\log_\mathscr{F}(x) + \log_\mathscr{F}(y)) = \exp_{\MU}(\sum_{k \geq 0} \frac{[\C \CP^k]}{k+1}(x^{k+1} + y^{k+1})), \]
\[ x [-1] y = \exp_{\MU}(\sum_{k \geq 0} \frac{[\C \CP^k]}{k+1}(x^{k+1} - y^{k+1})) = \exp_{\MU}((x-y) \sum_{k \geq 0} \frac{[\C\CP^k]}{k+1} \mbox{S}_k(x,y)), \]
where $\exp_{\MU}(z)$ is the compositional inverse of $\log_{\MU}(z)$, and $\mbox{S}_k(x,y)$ is the symmetric sum
\[ \mbox{S}_k(x,y) = x^k + x^{k-1}y + \cdots + x y^{k-1} + y^k. \]
Notice that even though the expressions for $\mathscr{F}_{\MU}(x,y)$ and $x [-1] y$ above appear to have denominators, these denominators cancel away in the ring $\MU^\ast(pt)$ once one expands the expression as a power series in $x$ and $y$.
\end{example}

\medskip
One may notice that the formal difference $x [-1] y$ in each of the above examples appears to be divisible by the expression $(x-y)$. In fact, this is always true as we now show

\medskip
\begin{claim} \label{2}
Given an arbitrary formal group law, there is a unique homogeneous degree 0 element $q(x,y) \in R\llbracket x,y \rrbracket$ so that
\[ x[-1]y = (x-y) \, \cdot q(x,y). \]
Furthermore, $q(x,y)$ is invertible and $q(x,y) \equiv 1\mod (y)$.
\end{claim}
\begin{proof}
Consider the formal expansion of the expression $(y+z)[-1]y := \mathscr{F}(y+z, [-1]y)$. On setting $z$ as $0$, we see that the expression vanishes. Therefore, it is divisible by $z$. Setting $z$ as $(x-y)$, we conclude that there is a power series $q(x,y)$ that satisfies the relation
\[ (x-y) \, q(x,y) := \mathscr{F}(y+x-y, [-1]y) = x[-1]y. \]
Since $(x-y)$ is not a zero divisor in $R\llbracket x,y \rrbracket$, we see that $q(x,y)$ is unique. Next, by setting $y$ as $0$, we see that $q(x,0)=1$. In particular, $q(x,y)$ has the form $1 \mod (y)$, and is therefore a unit. \end{proof}

\medskip
\begin{remark}
It is not hard to show using the definition of the power series $q(x,y)$ that the coefficients in its expansion generate the same sub algebra of $R$ as the coefficients of the formal group law $\mathscr{F}(x,y)$. To see this, first observe that the power series $[-1]y$ can be expressed in terms of the coeffiients of $q(x,y)$ using the fact that $[-1]y = -yq(0,y)$. Next, observe that $\mathscr{F}(x,y) = x[-1]([-1]y) = (x-[-1]y)q(x,[-1]y)$. These two observations together establish what we seek to show.
\end{remark}

\medskip
\begin{remark} \label{2b}
In example \ref{1d}, one may verify that $q(x,y)$ is the following (symmetric) expression
\[ q(x,y) = \frac{x+y}{x\sqrt{J(y)} + y\sqrt{J(x)}}. \]
In general however, $q(x,y)$ need not be symmetric in $x,y$ as is easily seen from example \ref{1}.
\end{remark}

\medskip
\begin{remark} \label{2a}
The universal example \ref{1e} allows us to deduce some interesting properties about $q(x,y)$. For instance, we see that $q(x,y)$ has the form
\[ q(x,y) = \sum_{k \geq 0} \frac{[\C\CP^k]}{k+1} \mbox{S}_k(x,y) + \sum_{n \geq 1} q_n(x,y) (x-y)^n, \]
where $q_n(x,y)$ are symmetric power series in $x$ and $y$. Note that each individual series $q_n(x,y)$ involves denominators. However, on setting $x = y$, those terms vanish and we obtain the interesting (universal) relation that does not involve denominators
\[ q(x,x) = \frac{d}{d x} \log_\mathscr{F}(x). \]
\end{remark}

\medskip
\begin{claim} \label{2c}
Assume that the $\Z$-algebra $R$ is torsion free. Then, given a formal group law $\mathscr{F}(x,y)$ over $R$, the power series $q(x,y)$ is symmetric if and only if $\log_\mathscr{F}(x)$ is an odd power series. Equivalently, $q(x,y)$ is symmetric if and only if $q(x,x)$ is an even power series. Note that these conditions are automatic if $R$ has no nontrivial elements in degrees $2 \mod 4$.
\end{claim}
\begin{proof}
The equivalence of the two conditions follows from remark \ref{2a} above. It remains to establish the first condition. Now $\log_\mathscr{F}(x)$ is an odd power series if and only if its compositional inverse $\exp_\mathscr{F}(x)$ is an odd power series. We will now proceed to show that symmetry of $q(x,y)$ is equivalent to $\exp_\mathscr{F}(x)$ being an odd power series. Using the universal example \ref{1e}, we see that the formal difference $x [-1] y$ for the formal group law $\mathscr{F}(x,y)$ has the form
\[ x [-1] y = \exp_\mathscr{F}(z), \quad \quad \quad \quad z = (x-y) s(x,y),\]
with $s(x,y)$ being a symmetric power series
\[ s(x,y) = \sum_{k \geq 0} \frac{l_k}{k+1} \mbox{S}_k(x,y), \quad \quad \mbox{where} \quad \quad \log_\mathscr{F}(x) = \sum_{k \geq 0} \frac{l_k}{k+1} x^{k+1}. \]
Note that $s(x,y)$ is invertible in $(R \otimes \Q) \llbracket x, y \rrbracket$, and so $z$ can be chosen to be a power series generator. We therefore have an inclusion
\[ (R \otimes \Q) \llbracket z \rrbracket \subset (R \otimes \Q) \llbracket x, y \rrbracket, \quad \quad z \longmapsto (x-y)s(x,y). \]
It follows that $q(x,y) \, s(x,y)^{-1} = \exp_\mathscr{F}(z)/z$ is symmetric if and only if $\exp_\mathscr{F}(z)/z$ is even, or that $\exp_\mathscr{F}(z)$ is odd.
\end{proof}

%
%

\medskip
Let us now study the divided difference operators in the context of formal group laws.

\medskip
\begin{definition} \label{3}
Consider the formal power series ring $R\llbracket x_1,x_2, \ldots, x_n \rrbracket$. Let $\alpha$ denote any pair $(i,j)$ for $1 \leq i < j \leq n$. We think of $\alpha$ as a positive root of $\U(n)$ so that the pairs $(i,j)$ are indexed by the set $\Delta_+$ of positive roots of $\U(n)$. Given a formal group law defined over $R$, we define the generalized divided difference operator $A_\alpha$ as the operator on $R\llbracket x_1,x_2, \ldots, x_n \rrbracket$
\[ A_\alpha (f) := \frac{f}{x_i [-1] x_j} + \frac{r_\alpha(f)}{x_j [-1] x_i}, \quad f \in R\llbracket x_1,x_2, \ldots, x_n \rrbracket, \]
where $r_\alpha$ is the reflection on $R\llbracket x_1,x_2, \ldots, x_n \rrbracket$ given by switching $x_i$ and $x_j$. Using elementary algebra, one can check that the operator $A_\alpha$ is well defined and does not involve denominators.
\end{definition}

\medskip
\begin{remark} \label{claim_pushpull}
If $E$ is a complex oriented cohomology theory with underlying formal group law $\mathscr{F}_E$ and the coefficients of a point being $E^*(pt) = R$, then the operators $A_\alpha$ have a natural meaning in terms of push-pull oprators on the $\U(n)$-equivariant $E$-cohomology ring of the flag variety $\U(n)/T$ (see \cite{BE}). More precisely, recall that the $\U(n)$-equivariant $E$-cohomology of a $\U(n)$-space $X$ is defined as the $E$-cohomology of the space $\EU(n) \times_{\U(n)} X$ with $\EU(n)$ being the principal contractible $\U(n)$-space. For $X = \U(n)/T$, the $\U(n)$-equivariant $E$-cohomology ring is isomorphic to $R\llbracket x_1, x_2, \ldots, x_n \rrbracket$, supporting the operator $A_\alpha$ that is defined as the pushforward in equivariant $E$-cohomology followed by the pullback: $\pi^\ast \circ \pi_\ast$, where $\pi$ denotes the $\U(n)$-equivariant fibration
\[ \pi : \U(n)/T \longrightarrow \U(n)/\U_\alpha(n), \]
with $\U_\alpha(n)$ being the maximal compact subgroup in the parabolic subgroup corresponding to the positive root $\alpha$.
\end{remark}

\medskip
\begin{claim} \label{4}
Given a root $\alpha \in \Delta_+$ defined by the pair $(i < j)$, let $q(\alpha)$ denote the unit $q(x_i, x_j)$ as defined in \ref{2}. Then the intersection of the kernels of all the operators $A_{\alpha_i}$, where $\alpha_i = x_i - x_{i+1}$ is a simple root, is a rank one free module over the ring of symmetric power series $R\llbracket x_1,x_2, \ldots, x_n \rrbracket^{\Sigma_n}$ generated by the unit $q(\Delta_+)$, where
\[ q(\Delta_+) := \prod_{\alpha \in \Delta_+} q(\alpha). \]
\end{claim}
\begin{proof}
Given a simple root $\alpha_i \in \Delta_+$, let us rewrite the action of $A_{\alpha_i}$ on $f$ as
\[ A_{\alpha_i}(f) = \frac{f}{x_i [-1] x_{i+1}} + r_i(\frac{f}{x_i [-1] x_{i+1}}). \]
Hence $f$ is in the kernel of $A_{\alpha_i}$ if the expression $\frac{f}{x_i [-1] x_{i+1}}$ switches sign under $r_i$. Notice that the ratio of any two such elements is invariant under $r_i$. On the other hand, by claim \ref{2}, it follows that given an $r_i$-invariant element $g$, the expression $q(\alpha_i) g$ is in the kernel of $A_{\alpha_i}$. In particular, we have shown that for a fixed $\alpha_i$, the kernel of $A_{\alpha_i}$ is {\em precisely} the rank one module of $r_i$ invariants generated by the unit $q(\alpha_i)$. Now let us fix $\alpha_i \in \Delta_+$, and consider
\[ \check{q}(\alpha_i) := \prod_{\beta \in \Delta_+, \, \beta \neq \alpha_i} q(\beta), \quad \quad \mbox{so that} \quad q(\alpha_i) \, \check{q}(\alpha_i) = q(\Delta_+). \]
Since $r_i$ permutes all positive roots besides $\alpha_i$, we see that $\check{q}(\alpha_i)$ is an $r_i$-invariant unit. Hence the kernel of $A_{\alpha_i}$ is a rank one free module of $r_i$-invariants generated by the element $q(\Delta_+)$. Taking intersection over all simple roots $\alpha_i \in \Delta_+$ we get the required result.
\end{proof}

\medskip
\begin{example} \label{5}
For the multiplicative formal group law of example \ref{1}, the intersection of the kernels of all the generalized divided difference operators $A_{\alpha_i}$ is a rank one free module over symmetric power series $R\llbracket x_1,x_2, \ldots, x_n \rrbracket^{\Sigma_n}$ generated by the unit $q(\Delta_+)$, where
\[ q(\Delta_+) = \frac{1}{(1-\beta x_2)(1-\beta x_3)^2 \ldots (1-\beta x_n)^{n-1}}. \]
Note that $q(\Delta_+)$ is not $\Sigma_n$-invariant for any $n > 1$.
\end{example}

\medskip
Let $D_\alpha$ denote the classical divided difference operator (i.e. the divided difference operator for the additive formal group law). Definition \ref{3} and claim \ref{2} imply that we have
\begin{equation} \label{eq:ADrelation}A_\alpha (f) = D_\alpha (q(\alpha)^{-1} f ), \quad \mbox{in other words} \quad A_\alpha = D_\alpha \circ Q(\alpha)^{-1}, \end{equation}
where $Q(\alpha)$ denotes the operator given by multiplication with $q(\alpha)$. In particular, the $q$-twisted operators $Q(\alpha_i) \circ A_{\alpha_i}$ satisfy the braid relations, and generate an algebra isomorphic to the nilHecke algebra, namely, the algebra generated by the operators $D_{\alpha_i}$. We also have

\medskip
\begin{theorem}\label{theorem_nilHecke}
Given a formal group law defined over $R$, let $\mathscr{A}(n)$ denote the algebra of operators on $R\llbracket x_1, x_2, \ldots, x_r \rrbracket$ generated by multiplication operators, and the generalized divided difference operators $A_{\alpha_i}$ for $1 \leq i < n$. Then $\mathscr{A}(n)$ is identically the same as the (completed) affine nilHecke algebra over the ground ring $R$. In other words, $\mathscr{A}(n)$ agrees with the algebra generated by the operators $D_{\alpha_i}$ and multiplication operators with respect to $R\llbracket x_1, x_2, \ldots, x_r \rrbracket$. In particular, $\mathscr{A}(n)$ is a free (left or right) module of rank $n!$ over the subalgebra $R\llbracket x_1, x_2, \ldots, x_n \rrbracket$. Alternatively, $\mathscr{A}(n)$ is a matrix algebra of rank $(n!)^2$ over the subalgebra $R\llbracket x_1, x_2, \ldots, x_n \rrbracket^{\Sigma_n}$. It follows that $R\llbracket x_1, x_2, \ldots, x_n \rrbracket^{\Sigma_n} \subset \mathscr{A}(n)$ is the center.
\end{theorem}
\begin{proof}
By (\ref{eq:ADrelation}), we see that the operators $A_{\alpha_i}$ are of the form $D_{\alpha_i} \circ Q(\alpha_i)^{-1}$, where $Q(\alpha_i)$ is the invertible multiplication operator corresponding to $q(\alpha_i)$. It follows that the operators $A_{\alpha_i}$ generate the same algebra as $D_{\alpha_i}$ when extended with multiplication operators, which is the affine nilHecke algebra by definition (once we complete polynomials to power series). The rest of the claim follows from well-known results on the affine nilHecke algebra.
\end{proof}

\begin{remark} \label{remark_surprise}
The above theorem may come as a surprise to the reader, since it has been known for some time that for an arbitrary compact Lie group $G$, the push-pull operators $A_{\alpha_i}$ defined as in remark \ref{claim_pushpull} and acting on the equivariant cohomology $E^\ast_{G}(G/T)$, do not satisfy the braid relations, unless the formal group law $\mathscr{F}_E$ underlying the cohomology theory $E$ is highly restrictive (see theorem 3.7 in  \cite{BE}, see also \cite{HMSZ}). It is possible that theorem \ref{theorem_nilHecke}  only holds for the compact Lie group $G = \U(n)$, though we have not verified this. It is important to note that the classes $q(\alpha_i)$ that allow for the proof of the above theorem have been studied before (see \cite{C}, \cite{Na}), though the main observation of theorem \ref{theorem_nilHecke} appears to be new.
\end{remark}

%
%

\section{Deformed $GL(2)$ foam evaluation}
\label{sec:deform}

\subsection{$GL(2)$ foams and their colorings}

The original formulation~\cite{Kh1} of $SL(2)$ link homology did not use foams. Hints at foams appeared in the work of Clark, Morrison, and Walker~\cite{CMW} and Caprau~\cite{Ca1,Ca2}, who used disorientation lines on surfaces involved in the construction of $SL(2)$ homology to control minus signs that appear throughout the theory. This allowed them to establish full functoriality of the theory under cobordisms rather than the functoriality up to an overall minus sign, as shown in the earlier work~\cite{J,Kh3,BN}. One can think of disorientation lines as remnants of the 2-facets of $GL(2)$ foams along which they were attached to the 1-facets.

Earliest constructions of $SL(3)$ and $SL(N)$ link homology for $N>3$ used foams explicitly~\cite{Kh2} and implicitly~\cite{KRo2}.

Blanchet~\cite{B} pioneered the use of foams for the $SL(2)$ (more precisely, $GL(2)$) homology theory. A detailed exploration of various flavours of $GL(2)$ foams and applications can be found in~\cite{EST1,EST2}.

We find its useful to follow the $GL(2)$ foam calculus of Beliakova, Hogancamp, Putyra, Wehrli~\cite{BHPW}. That's the calculus deformed in this section.

We consider $GL(2)$ foams (or, simply, foams) in this paper. A closed $GL(2)$ foam $F$ is a combinatorial compact two-dimensional CW-complex embedded in $\R^3$ (or $\SS^3$). The only allowed singularities of the CW-complex are \emph{singular circles}, such that any point on the circle has a neighbourhood homeomorphic to the product of the tripod and the interval.

The set of points of $F$ on its singular circles is denoted $s(F)$, and connected components of $F\setminus s(F)$ are called \emph{facets}.
Facets of $F$ are subdivided into 1-facets and 2-facets. One-facets are also called  \emph{thin} facets, two-facets are also called \emph{double} or \emph{thick} facets. We require that along each singular circle two 1-facets and one 2-facet meet, see Figure~\ref{fig:gl2_sing_circle}.

\begin{figure}[h]
\begin{center}
\includegraphics[scale=0.70]{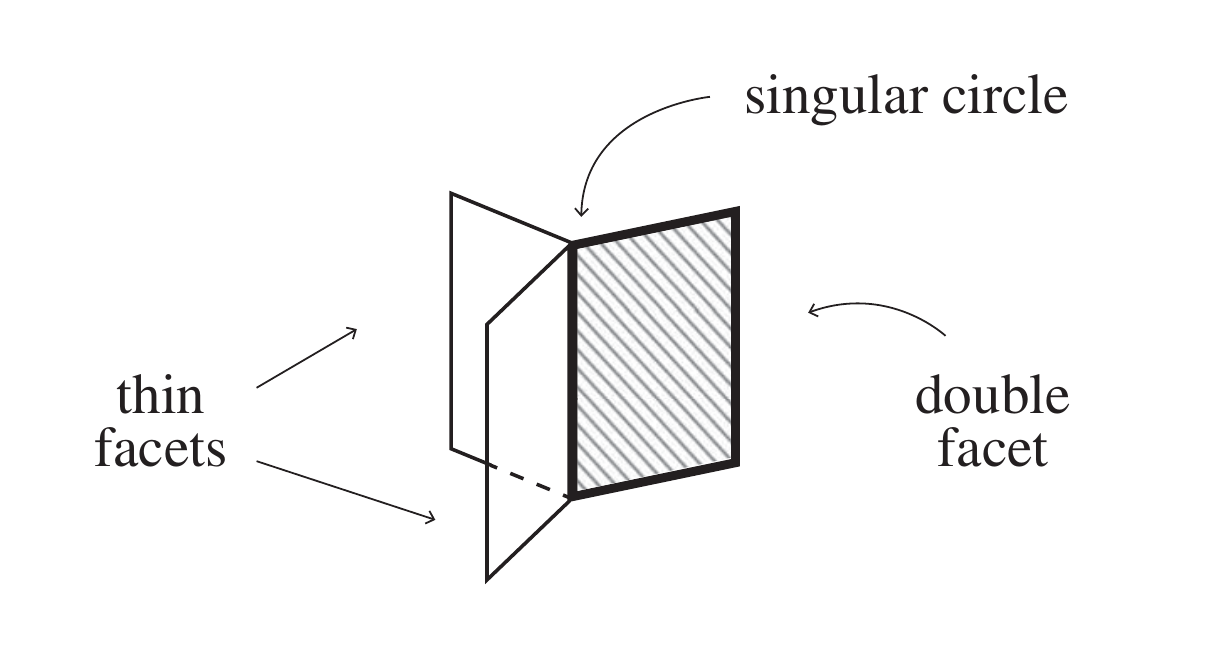}
\caption{\label{fig:gl2_sing_circle} Part of a singular circle and its neighbourhood, with two thin and one double facet.}
\end{center}
\end{figure}

This implies, in particular, that no 'monodromy' is possible along any singular circle, so it has a neighbourhood in $F$ homeomorphic to the product of $\SS^1$ and a tripod with 'two thin legs and a double leg'.

Each facet is oriented in such a way that all three facets along any singular circle induce a compatible orientation on this circle, see Figure~\ref{fig:gl2_orientations}. In most diagrams that follow, it's clear whether a facet is thin or double, and we usually omit the corresponding label $1$ or $2$.

\begin{figure}[h]
\begin{center}
\includegraphics[scale=0.70]{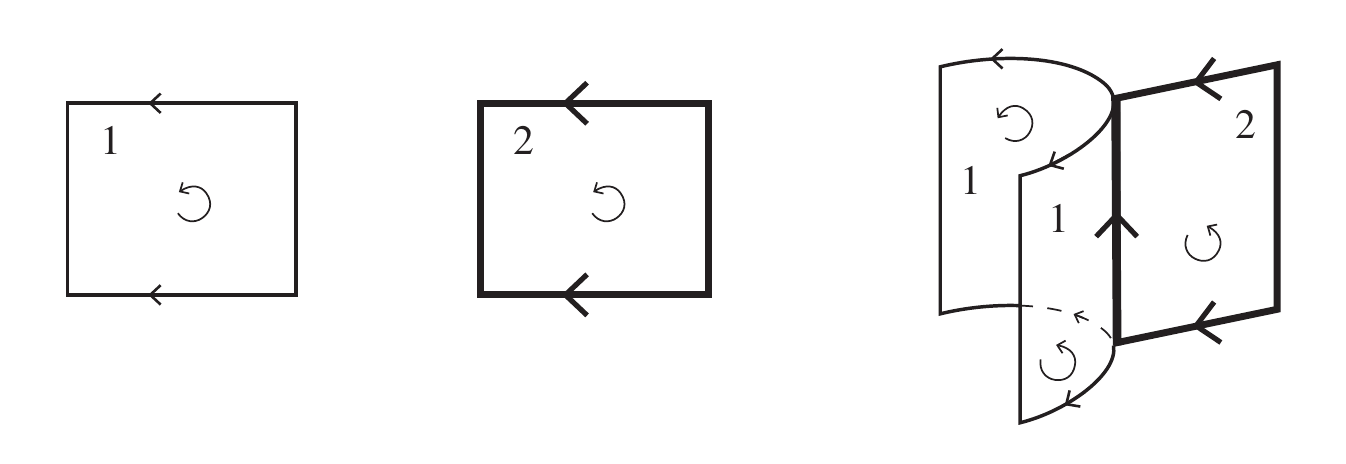}
\caption{\label{fig:gl2_orientations} Left: oriented think and thick facets, with induced orientations on the top and bottom boundary. Right: compatibility between orientations of a singular circle and adjacent thin facets and the double facet. Also shown induced orientations of the top and bottom foam's boundary. They'll be needed when we pass from closed foams to foams with boundary in Section~\ref{subsec:direct_sum_dec}}
\end{center}
\end{figure}

Thin facets may carry dots, which can move freely along a facet, but cannot jump to an adjacent facet. If a facet carries $n$ dots, we may record them as a single dot with label $n$. It's possible to allow similar decorations on 2-facets, namely symmetric polynomials in two variables, but we avoid doing so in the paper, instead moving any such decoration from
a 2-facet to the coefficient of the foam.

\emph{Remark:} Unlike $GL(N)$ foams for $N\ge 3$ and $SL(N)$ foams for $N\ge 4$, $GL(2)$ foams can't have singular vertices.

A \emph{coloring} (or \emph{admissible coloring}) of
a foam $F$ is a map $c$ from the set of its thin facets to the set $\{1,2\}$ such that along any singular circle, the two thin facets are mapped to different numbers. It's convenient to extend $c$ to double facets, coloring each double facet by the set $\{1,2\}$. This produces the flow condition, that the union of colors of 1-facets along each singular circle is the color of the double facet, that is, the entire set $\{1,2\}$.

Notice that $F_{12}(c)$ does not depend on the coloring $c$ and is a closed surface which is the union of closures of 1-facets of $F$. We denote it by $F_{12}$ and
call the \emph{thin surface} of $F$. Likewise, $F_{1\cap 2}(c)$ does not depend on $c$ and is
the union of closures of 2-facets of $F$.
We denote it by $F_{1\cap 2}$ and call the \emph{double surface} of $F$. The boundary of $F_{1\cap 2}$ is exactly the set of singular circles of $F$.

Often it's convenient to identify a facet $f$ with its closure $\overline{f}$ in $F$. In particular, the Euler characteristics of $f$ and $\overline{f}$ are equal, since the two spaces differ only by a union of circles, which is the boundary of $\overline{f}$. From now on, unless otherwise specified, by a facet we mean a closed facet.

 Surface $F_{12}$ has finitely many connected components $\Sigma_1, \dots, \Sigma_n$. Each component may contain one or more singular circles. The union of these singular circles is zero when viewed as an element of $H_1(\Sigma_k,\Z/2)$, for any $k$, due to our orientation requirements on $F$. In particular, each $\Sigma_k$ admits exactly two checkerboard colorings of its regions, where along each singular circle in $\Sigma_k$ the coloring is reversed.

A choice of such coloring for each $\Sigma_k$ is equivalent to a coloring of $F$. Hence, $F$ has $2^n$ colorings, where $n$ is the number of connected components of $F_{12}$.

Quantity $\theta^+(c)=\theta_{12}^+(c)$ counts the number of positive circles for a coloring $c$, see Figure~\ref{fig:theta_signs}.

\begin{figure}[h]
\begin{center}
\includegraphics[scale=0.70]{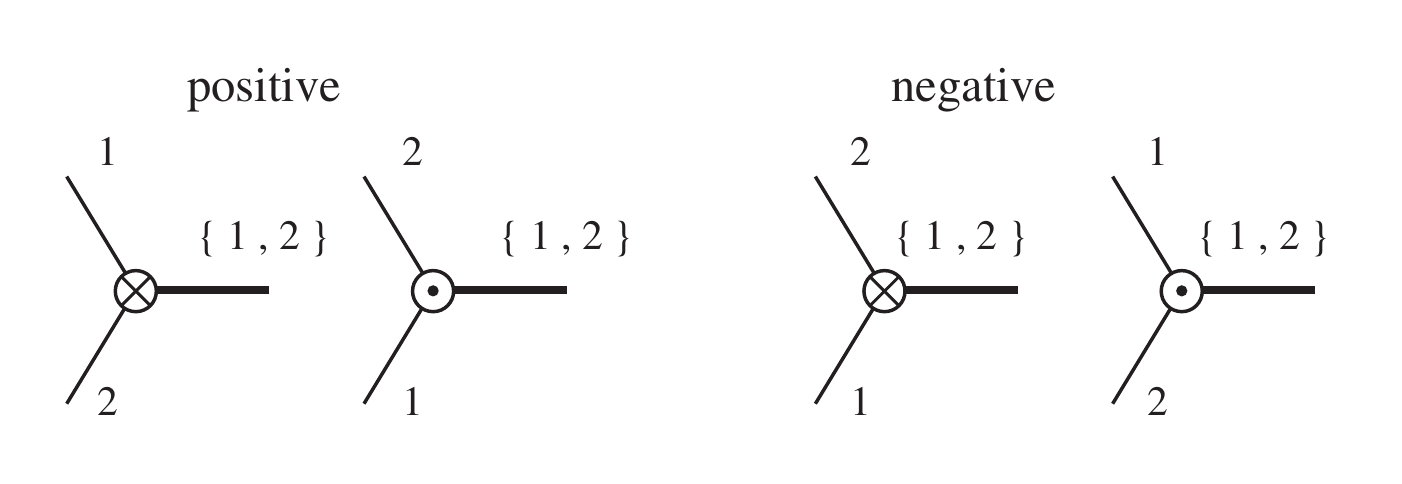}
\caption{\label{fig:theta_signs} $\theta^+_{12}(c)$ or simply $\theta^+(c)$ counts the number of positive circles.}
\end{center}
\end{figure}


\subsection{Deformed evaluation for $GL(2)$ foams}

Modified Robert-Wagner evaluation formula, in the $GL(2)$ case, specializes to
\begin{equation}\label{eq_eval_modify}
 \angf{F,c} \ = \ (-1)^{\theta^+_{12}(c)}
\frac{(x_1-x_2)^{\chi(F_{1\cap 2}(c))}x_1^{d_1(c)}x_2^{d_2(c)}}
{(x_1[-1] x_2)^{\chi_1(c)/2}(x_2[-1] x_1)^{\chi_2(c)/2}}
\end{equation}
Since
\begin{equation} \label{eq_minus_12}
x_1 [-1] x_2 = (x_1-x_2)p_{12}^{-1}, \ \  x_2 [-1] x_1 = (x_2-x_1)p_{21}^{-1},
\end{equation}
we have
\begin{equation} \label{eq:deformed_glN_2}
\frac{(x_1 - x_2)^{\chi_{1\cap 2}(c)}}{(x_1 [-1] x_2)^{\chi_1(c)/2} (x_2 [-1] x_1)^{\chi_2(c)/2}}
= (-1)^{\chi_2(c)/2}\frac{p_{12}^{\chi_1(c)/2}p_{21}^{\chi_2(c)/2}}{(x_1 - x_2)^{\chi_{12}(c)/2}}
\end{equation}
In the 2-color case, $F_{1\cap 2}(c)=F_{1\cap 2}$ is the union of facets of thickness two and does not depend on $c$.  Likewise, $F_{12}(c)=F_{12}$ does not depend on $c$ either. Its Euler characteristic is denoted $\chi_{12}(F)= \chi(F_{12}).$

Equation (\ref{eq_eval_modify}) can be rewritten
\begin{eqnarray*}
\angf{F,c} & = & (-1)^{\theta^+_{12}(c)+\chi_2(c)/2} \
\frac{x_1^{d_1(c)}x_2^{d_2(c)}}
{(x_1- x_2)^{\chi(F_{12})/2} p_{12}^{-\chi_1(c)/2}
p_{21}^{-\chi_2(c)/2}}   \\
& = & (-1)^{\theta^+_{12}(c)+\chi_2(c)/2} \
\frac{x_1^{d_1(c)}x_2^{d_2(c)}p_{12}^{\chi_1(c)/2}
p_{21}^{\chi_2(c)/2}}
{(x_1- x_2)^{\chi(F_{12})/2}
}  \\
& = & (-1)^{s'(F,c)} \frac{P(F,c)}{Q(F,c)}\   p_{12}^{\chi_1(c)/2}
p_{21}^{\chi_2(c)/2} \\
& = & (-1)^{\chi(F)/2} \ang{F,c} \ p_{12}^{\chi_1(c)/2}
p_{21}^{\chi_2(c)/2}.
\end{eqnarray*}
Above, $s'(F,c)=\theta^+_{12}(c)+\chi_2(c)/2$ and $d_1(c)$, $d_2(c)$ is the number of dots on thin facets colored by $1$, resp. $2$ by $c$.
We see that the original evaluation $\ang{F,c}$ is scaled by an invertible element, which is a product of powers of $p_{12}$ and $p_{21}$ and a sign. Also, the power of $x_1-x_2$ in the denominator depends on $F$ only. Let us write down the formula again.

\begin{equation}\label{eq:def_eval_gl2}
\angf{F,c}  \ = \  (-1)^{\theta^+_{12}(c)+\chi_2(c)/2} \
\frac{x_1^{d_1(c)}x_2^{d_2(c)}
p_{12}^{\chi_1(c)/2} p_{21}^{\chi_2(c)/2}}
{(x_1- x_2)^{\chi(F_{12})/2}
}
\end{equation}

We now define
\begin{equation}\label{eq:def_eval_gl3}
    \angf{F} \ = \ \sum_{c} \angf{F,c},
\end{equation}
the sum over all colorings of $F$. Let $E_1=x_1+x_2$, $E_2=x_1x_2$. The symmetric group $S_2$ acts on $\pseries{x_1,x_2}$ by permuting $x_1,x_2$.

\begin{theorem} \label{thm:no_denoms_GL2} $\angf{F}\in \pseries{x_1,x_2}^{S_2} \cong \pseries{E_1,E_2}$ for any $GL(2)$ foam $F$.
\end{theorem}
In other words, $\angf{F}$ is a power series in $x_1, x_2$ that's symmetric under the permutation action of $S_2$ on $x_1,x_2$. Equivalently, it's a power series
in elementary symmetric functions $E_1, E_2$.
Consider the chain of inclusions
\begin{equation}\label{eq:ring_inclusions}
    \kk \ \subset \ \pseries{E_1,E_2} \ \subset \  \pseries{x_1,x_2} \
    \subset \ \pseries{x_1,x_2} \biggl[\frac{1}{x_1-x_2}\biggr].
\end{equation}
Denote these rings by
\begin{eqnarray}
  \wR & = &  \pseries{E_1,E_2}, \label{eq:ring_wR}\\
  R' & = &  \pseries{x_1,x_2}, \label{eq:ring_R_prime}\\
  R'' & = & \pseries{x_1,x_2} \biggl[\frac{1}{x_1-x_2}\biggr],\label{eq:ring_R_dprime}
\end{eqnarray}
resulting in the chain of ring inclusions
\begin{equation}\label{eq:ring_inclusions_2}
    \kk \ \subset \ \wR \  \subset \ R' \ \subset \ R''.
\end{equation}
The theorem above has already been proved in Section~\ref{sec:deformed} for general $N$, see Theorem~\ref{Foam_denominator}. We include a more detailed proof for the special case $N=2$ to make this section independent from Section~\ref{sec:deformed}.

\begin{proof}
The evaluation $\angf{F,c}$ can be written, via (\ref{eq_eval_modify}), as a power series in $x_1,x_2$ with coefficients in $\kk$
divided by a power of $x_1-x_2$, either positive or negative, thus it belongs to the ring $R''$, see above.

Group $S_2$ acts on 2-colorings of $F$ by transposing the colors $1$ and $2$. This action is compatible with the evaluation in the sense that $\sigma(\angf{F,c})= \angf{F, \sigma(c)},$ where $\sigma=(12)$ is the nontrivial element of $S_2$. Therefore, $\angf{F}$ is in the subring $(R'')^{S_2}$ of $S_2$-invariants of $R''$.

$\angf{F}$ potentially has a denominator $(x_1-x_2)^{\chi(F_{12}(c))/2}$. Surface $F_{12}(c)=F_{12}$ is a union of connected components $\Sigma_1,\dots, \Sigma_m$, each one contributing
$(x_1-x_2)^{\chi(\Sigma_k)/2}$ to the product. Only connected components of genus $0$ have positive Euler characteristic, $\chi(\Sigma_k)=2$, and contribute $x_1-x_2$ to the denominator.

Consider one such component $\Sigma$ and a coloring $c$. The Kempe move on $\Sigma$ replaces $c$ with a coloring $c_1=(c,\Sigma)$ which is identical to $c$ outside $\Sigma$ and swaps colors $1,2$ of $c$ on $\Sigma$. We compare $\angf{F,c}$ and $\angf{F,c_1}$ in formula (\ref{eq:def_eval_gl2}).

If there are $t_i$ dots on color $i$ facets of $\Sigma$ under $c$, $i=1,2$, then
\[ x_1^{d_1(c)}x_2^{d_2(c)} =
x_1^{t_1}x_2^{t_2}u, \ x_1^{d_1(c_1)}x_2^{d_2(1)} =
x_1^{t_2}x_2^{t_1}u,
\]
for a monomial $u$ in $x_1,x_2$ counting dots on facets not in $\Sigma$.

If $\Sigma$ has $r$ singular circles, let $\theta^+(c,\Sigma)$ be the number of positive circles on $\Sigma$ under $c$ and $\theta^-(c,\Sigma)$ be the number of negative circles. Under the swap $c\leftrightarrow c_1$, positive circles on $\Sigma$ become negative circles on $\Sigma$ and vice versa, so that $\theta^+(c,\Sigma)+\theta^+(c_1,\Sigma)=r$.

Let $\wchi_i(\Sigma)=\chi(\Sigma\cap F_i(c))$, $i=1,2,$ be the Euler characteristic of the union of color $i$ facets of $\Sigma$, for coloring $c$. We have
\[ \chi_2(c_1)=\chi_2(c)+\wchi_1(\Sigma)-\wchi_2(\Sigma) = \chi_2(c) + \chi(\Sigma) - 2 \wchi_2(\Sigma),
\]
since $\chi(\Sigma) = \wchi_1(\Sigma)+\wchi_2(\Sigma).$
Defining integer $\ell$ by  $2\ell = \widetilde{\chi}_1(\Sigma)-\widetilde{\chi}_2(\Sigma).$, we have
 \[ \chi_2(c_1) \ = \ \chi_2(c) + 2\ell, \ \chi_1(c_1) \ = \ \chi_1(c) - 2\ell.
 \]
 Consequently, one can write
 \[ p_{12}^{\chi_1(c_1)/2}p_{21}^{\chi_2(c_1)/2} =
 p_{12}^{\ell}u', \ \
 p_{12}^{\chi_1(c)/2}p_{21}^{\chi_2(c)/2}= p_{21}^{\ell}u',
  \]
for a suitable monomial $u'$ in $p_{12}, p_{21}$, possibly with negative exponents.

Also,
\[ \frac{(-1)^{s'(F,c_1)}}{(-1)^{s'(F,c)}} = \frac{(-1)^{\theta^+(c_1)+\chi_2(c_1)/2}}
{(-1)^{\theta^+(c)+\chi_2(c)/2}}=
(-1)^{r + \chi(\Sigma)/2 - \wchi_2(\Sigma)}.
\]

When $\Sigma\cong \SS^2$,
\[r+\chi(\Sigma)/2 - \wchi_2(\Sigma) \equiv 1 + r - \wchi_2(\Sigma)\equiv 1 (\mathrm{mod}\ 2),
\]
since $r\equiv \chi_2(\Sigma)(\mathrm{mod}\ 2)$. The last comparison modulo $2$ can be proved by induction on $r$, by removing an innermost singular circle of $\Sigma$. This operation reduces $r$ by $1$ and changes $\wchi_2(\Sigma)$ by $\pm 1$. We see that
$(-1)^{s'(F,c_1)}=-(-1)^{s'(F,c)}$.

Putting these relations together,
\[ \angf{F,c} + \angf{F,c_1} =
s'(F,c)\cdot\bigl( x_1^{t_1}x_2^{t_2} p_{12}^{\ell}
-
x_1^{t_2}x_2^{t_1}p_{21}^{\ell}
\bigr) u u' (x_1-x_2)^{-\chi_{12}(F)/2} .
\]
The expression $( x_1^{t_1}x_2^{t_2} p_{12}^{\ell}
-
x_1^{t_2}x_2^{t_1}p_{21}^{\ell}) $ is divisible by $x_1-x_2$ and  allows to cancel out that term from the denominator.

Repeating this argument simultaneously for all $S^2$-components of $F_{12}$ shows that $\angf{F} \in \kk\llbracket x_1,x_2 \rrbracket.$
Permutation action of $S_2$ on $\kk\llbracket x_1,x_2 \rrbracket$ and on colorings shows that $\angf{F} \in \kk\llbracket x_1,x_2 \rrbracket^{S_2}=\kk\llbracket E_1,E_2\rrbracket.$
\end{proof}

The sum $\chi(F) = \chi_1(c)+\chi_2(c)$ does not depend on the coloring $c$ and is the Euler characteristic of the surface $F$. In particular, in the symmetric case (when $p_{12}=p_{21}$), one has that
\begin{equation} \angf{F} = \ang{F} \cdot (-p_{12})^{\chi(F)/2}
\end{equation}
so that the new evaluation is proportional to the original one with the coefficient that depends only on $\chi(F)$. We expect that non-symmetric case will prove more interesting.


\subsection{Examples} \label{sec:examples}

$\quad$

{\it Example 1:} Let $F=\SS^2_{1,n}$ be the two-sphere of thickness one with $n$ dots (or, equivalently, with a single dot labelled $n$). Here the lower index $(1,n)$ lists thickness followed by the number of dots.

$\SS^2_{1,n}$ has two colorings $c_1$ and $c_2$, where in the coloring $c_i$ the 2-sphere carries color $i$. For the coloring $c_1$
$$ F_1(c_1) \cong \SS^2, \ F_2(c_1) = \emptyset , \ F_{1\cap 2}(c_1) = \emptyset, $$
and likewise for $c_2$, so that
\begin{eqnarray*}
 \angf{\SS^2_{1,n},c_1}  & = & \frac{x_1^n}{x_1 [-1] x_2} = \frac{x_1^n p_{12}}{x_1-x_2}, \\
 \angf{\SS^2_{1,n},c_2}  & = & \frac{x_2^n}{x_2 [-1] x_1} = \frac{x_2^n p_{21}}{x_2-x_1},
\end{eqnarray*}
and
\begin{equation*}
    \angf{\SS^2_{1,n}} = \frac{x_1^n p_{12} - x_2^n p_{21}}{x_1-x_2}.
\end{equation*}
To explicitly cancel $x_1-x_2$ in the denominator, expand $p_{12}$ and $p_{21}$ into power series and then cancel. The result is a power series symmetric in $x_1,x_2$ with coefficients which are polynomials in $\beta_{i,j}$.

We denote
\begin{equation}\label{eq:def_rho_n}
    \rho_n \ = \ \angf{\SS^2_{1,n}} = \frac{x_1^n p_{12} - x_2^n p_{21}}{x_1-x_2}.
\end{equation}

Note that the following relation holds:
\begin{equation}\label{eq:2sphere-inductive}
   \rho_{n+2} - E_1 \rho_{n+1} +
   E_2 \rho_n = 0,
\end{equation}
where, recall, $E_1=x_1+x_2$, $E_2= x_1 x_2$. It follows from the relation
\[  (x_1^{n+2}p_{12} - x_2^{n+2}p_{21}) - (x_1+x_2) (x_1^{n+1}p_{12} - x_2^{n+1}p_{21}) + x_1x_2
(x_1^{n}p_{12} - x_2^{n}p_{21}) = 0.
\]
Relation (\ref{eq:2sphere-inductive}) allows to inductively write $\rho_n=\angf{\SS^2_{1,n}}$
as a linear combination of $\rho_0=\angf{\SS^2_{1,0}}$ and
$\rho_1=\angf{\SS^2_{1,1}}$ with coefficients in $\Z[E_1,E_2]$. The latter are
\begin{equation}\label{eq:2sphere_few_dots}
  \rho_0 = \angf{\SS^2_{1,0}}  =  \frac{ p_{12} - p_{21}}{x_1-x_2}, \ \ \ \
  \rho_1 = \angf{\SS^2_{1,1}}  =  \frac{x_1 p_{12} - x_2 p_{21}}{x_1-x_2}.
\end{equation}


{\it Example 2:} Let the foam $F$ be a thin two-torus $T^2$ with $n$ dots and standardly embedded in $\R^3$ (embedding of a surface does not influence its evaluation).
As in the previous example, there are two colorings, $c_1$ and $c_2$, with
$F_1(c_1)\cong T^2$, $F_2(c_1)= \emptyset$, $F_{1\cap 2}(c_1)= \emptyset,$ and  $\brak{T^2}'= x_1^n + x_2^n.$

\vspace{0.1in}

{\it Example 3:} Closed surface $M$ of genus $g\ge 1$ with $n$ dots.
\begin{equation}\label{eq:genus_g_surface}
\angf{M} =  \frac{x_1^n p_{12}^{1-g}}{(x_1-x_2)^{1-g}} +
\frac{x_2^n p_{21}^{1-g}}{(x_2-x_1)^{1-g}}
    =  (x_1^n p_{12}^{1-g} + (-1)^{g-1}x_2^n p_{21}^{1-g})(x_1-x_2)^{g-1}.
\end{equation}

\vspace{0.1in}

{\it Example 4:} $F$ is  2-sphere $\SS^2$ of thickness two, also denoted $\SS^2_2$. It has a unique coloring $c$, with the facet labelled by $\{1,2\}$ and $F_1(c)=F_2(c)=F_{1\cap 2}(c) \cong \SS^2$ so that
\begin{eqnarray*}
\angf{\SS^2_2} & = & \angf{\SS^2_2,c} = \frac{(x_1- x_2)^2}{(x_1[-1]x_2)(x_2 [-1] x_1)}=
\frac{(x_1- x_2)^2p_{12}p_{21}}{(x_1-x_2)(x_2-x_1)} \\
 & = & - p_{12}p_{21}
 \end{eqnarray*}
 Denote the value of this foam by $\rho$, so that
 \begin{equation}\label{eq:value_S22}
 \rho \ = \ \angf{\SS^2_2} \ = \ - p_{12}p_{21}.
 \end{equation}
 Note that $\rho$ is an invertible element of the ground ring.
In the original case, when $p(x,y)=1$, the double sphere $\SS^2_2$ evaluates to $-1$.

\vspace{0.1in}

{\it Example 5:} An oriented closed surface $M$ of genus $g\ge 1$ and thickness two:

\begin{equation}
\angf{M} = \frac{(x_1-x_2)^{2-2g}p_{12}^{1-g}q_{21}^{1-g}}{(x_1-x_2)^{1-g}(x_2-x_1)^{1-g}} =(-p_{12}p_{21})^{1-g}=\rho^{1-g}.
\end{equation}

In the special case, when $g=1$ so that $M=T^2$ is a two-torus,  $\angf{T^2_2}=1.$

\vspace{0.1in}

{\it Example 6:} The theta-foam $\Theta$ with $n_1$ and $n_2$ dots on thin facets, suitably oriented.

\begin{figure}[h]
\begin{center}
\includegraphics[scale=0.7]{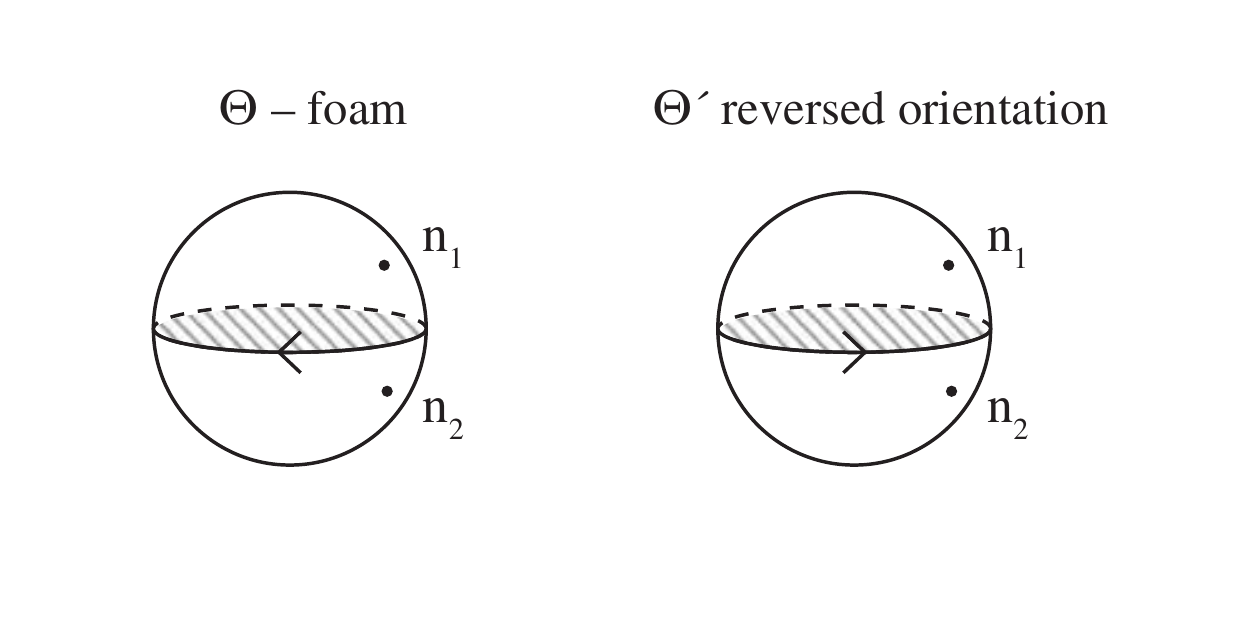}
\caption{\label{fig:theta_foam_1} On the left: $\Theta$-foam. On the right: same foam $\Theta'$ with the reversed orientation.}
\end{center}
\end{figure}

Let $c_1$ be the coloring of $\Theta$ with its top facet colored $1$. Then the bottom facet is colored $2$. Surfaces $F_1(c_1), $ $F_2(c_1),$ and  $F_{12}(c_1)=F_{12}$ are all 2-spheres, with Euler characteristics $2.$ The sign $\theta^+_{12}(c_1)=1$.

\begin{figure}[h]
\begin{center}
\includegraphics[scale=0.8]{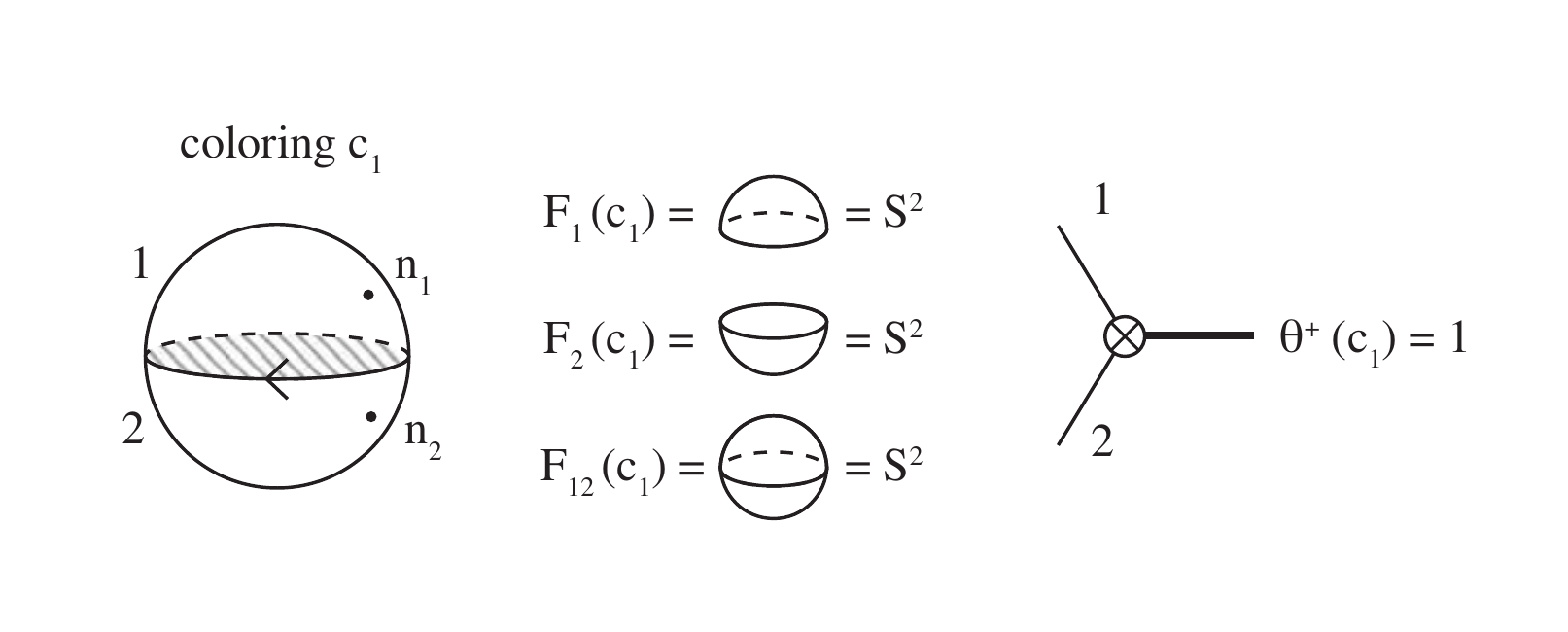}
\caption{\label{fig:theta_foam2} Computing $\angf{\Theta,c_1}$. }
\end{center}
\end{figure}
For the sign,
\[
s'(\Theta,c_1)=\theta^+(c_1)+\chi_2(c_1)/2=1+1=2, \ \ (-1)^{s'(\Theta,c_1)}= 1.
\]
We get
\[
\angf{\Theta,c_1} = \frac{x_1^{n_1}x_2^{n_2}p_{12}p_{21}}{x_1-x_2},
\ \
\angf{\Theta,c_2} = \frac{-x_1^{n_2}x_2^{n_1}p_{12}p_{21}}{x_1-x_2},
\]
where $c_2$ is the other coloring (with the opposite sign in the evaluation and transposed exponents of $x_1,x_2$). Assuming $n_1\ge n_2$,
\begin{equation}\label{eq:theta_foam_value}
    \angf{\Theta} \ = \ (x_1x_2)^{n_2} \frac{x_1^{n_1-n_2}-x_2^{n_1-n_2}}{x_1-x_2} p_{12}p_{21} \ = \ (x_1x_2)^{n_2} h_{n_1-n_2-1}(x_1,x_2) p_{12}p_{21},
\end{equation}
where $h_k(x_1,x_2)=x_1^k + x_1^{k-1}x_2+\dots + x_2^k$ is the $k$-the complete symmetric function in $x_1,x_2$. Note that we can write
\begin{equation}\label{eq:theta_foam_value_2}
    \angf{\Theta} \ = \ - E_2^{n_2} h_{n_1-n_2-1}(x_1,x_2) \rho ,
\end{equation}
and that $h_k(x_1,x_2)$ is a polynomial in $E_1,E_2$, the latter elementary symmetric functions in $x_1,x_2$. Also, $\angf{\Theta}$ is the product of a Schur function for $GL(2)$ and $-\rho$.

Note that if the two thin facets of the theta-foam carry the same number of dots, $n_1=n_2$, then it
evaluates to zero, $\angf{\Theta} =0$. If we reverse the orientation of $\Theta$ to get a foam $\Theta'$, then $\angf{\Theta'} = - \angf{\Theta}.$ In general, if foam $F$ contains $k$ singular circles and $\overline{F}$ is given by reversing the orientation of $F$, then $\angf{\overline{F}} = (-1)^k \angf{F}$.

Recall that our ground ring $\wR$ is the power series  $\pseries{E_1,E_2}$, where $\kk =\Z[\beta_{i,j}]$ is polynomials in various negative degree generators with integer coefficients, see formulas (\ref{eq_ring_kk}), (\ref{eq:ring_wR}). Let $R$ be the subring of $\wR=\pseries{E_1,E_2}$ generated by
$E_1,E_2,  \rho_0, \rho_1, \rho^{\pm 1}$ over $\Z$:
\begin{equation}\label{eq:def_ground_ring}
    R \  = \  \langle E_1, E_2, \rho_0, \rho_1, \rho^{\pm 1} \rangle   \ \subset \ \wR \ = \  \pseries{E_1,E_2}.
\end{equation}
In all the examples above, the foam evaluates to an element of this subring. We'll see soon that this is true for any closed foam and that the ground ring of the theory can be reduced from the rather large power series ring
$\pseries{E_1,E_2}$ to the subring $R$, which is
finitely generated over the image of $\Z$ in $\kk$.

Let us summarize that
\begin{eqnarray}
  \rho_0 & = & \angf{\SS^2_{1,0}}  =  \frac{ p_{12} - p_{21}}{x_1-x_2}, \nonumber \\
  \label{eq:rho_values}
  \rho_1 & = & \angf{\SS^2_{1,1}}  =  \frac{x_1 p_{12} - x_2 p_{21}}{x_1-x_2}, \\
  \rho & = & \angf{\SS^2_{2,0}}  = - p_{12}p_{21} \nonumber
\end{eqnarray}
are the evaluations of the thin 2-sphere with zero dots, with one dot, and the double 2-sphere, respectively. The subring $R$ is graded, with homogeneous generators in degrees

\begin{center}
    \begin{tabular}{|c|c|c|c|c|c|}
    \hline
       generator   & \ $\rho_0$ \ & \ $\rho_1$ \  & \ $\rho$ \ & \ $E_1$ \ & \ $E_2$ \   \\
       \hline
       degree   & -2 & 0  & 0 & 2 & 4  \\
       \hline
    \end{tabular}
\end{center}
Notice that only $\rho_0$ has a negative degree.
Using that $E_1^2-4E_2=(x_1-x_2)^2$, it's easy to compute
\begin{equation}\label{eq:express_rho}
    \rho_1^2 - E_1 \rho_1 \rho_0 + E_2 \rho_0^2 = \frac{E_1^2-4E_2}{(x_1-x_2)^2}\ p_{12}p_{21} = -  \rho.
    \end{equation}
Define the ring
\begin{equation}\label{eq_ring_ovR}
\ovR = \Z[E_1,E_2,\rho_0,\rho_1,(\rho_1^2 - E_1 \rho_1 \rho_0 + E_2 \rho_0^2)^{-1}]
\end{equation}
as the localization of the polynomial ring with generators
$E_1,E_2,\rho_0,\rho_1$ at the element
\begin{equation}
\rho=-(\rho_1^2 - E_1 \rho_1 \rho_0 + E_2 \rho_0^2).
\end{equation}
There is an obvious homomorphism $\ovR\lra R\subset \wR$, and we now prove that it's an isomorphism between $\ovR$ and $R$.

Consequently, we can think of $R$ as the localization,
\begin{equation}\label{eq:def_R}
   R \ \cong \ \Z[E_1,E_2,\rho_0,\rho_1,(\rho_1^2 - E_1 \rho_1 \rho_0 + E_2 \rho_0^2)^{-1}].
\end{equation}
We will show that this localization has a basis over  $\Z[E_1,E_2]$:
\begin{equation} \label{eq_basis_R}
\mathscr{B} := \{ \rho_1^{n_1}\rho_0^{n_2} \rho^{n_3} , \ \
n_1\in \{0,1\}, n_2\in \Z_+, n_3 \in \Z \}
\end{equation}
To establish isomorphism  $\ovR\cong R$ of rings, denote by
\begin{equation}\label{eq_Rminus}
\Rminus := \Rminuslong
\end{equation}
the graded ring of power series in $E_1,E_2$ with coefficients in the ring $\Z[\rho_0, \rho_1^{\pm 1}]$. In this definition, we view $\rho_0,\rho_1$ as additional generators and not as power series.

\begin{lemma} Ring $\Rminus$ is  naturally a subring of $\wR$, via power series expansions (\ref{eq:rho_values}) for $\rho_0$ and $\rho_1$.
\end{lemma}
\begin{proof} We can write the power series
\[ p_{12} = 1 + Ax_1 + Bx_2 + \sum_{i+j>1} \beta_{ij}x_1^ix_2^j, \]
where $A=\beta_{1,0}$ and $B=\beta_{0,1}$, also see formula (\ref{eq_p_series}).
The power series for $\rho_1$ is invertible, since the expansion starts with $1+A(x_1+x_2)=1+A E_1$ followed by higher degree terms in $x_1$ and $x_2$ with coefficients in the variables $\beta_{ij}$ for $i+j>1$,
\[ \rho_1 \longmapsto \frac{x_1 p_{12}-x_2 p_{21}}{x_1-x_2} =1 + A (x_1+x_2) + h.o.t.,
\]
where $h.o.t.$ stands for 'higher order terms'.
Furthermore, the series expansion for $\rho_1$ does not involve the coefficient $B$. The series
\[ \rho_0 \longmapsto \frac{p_{12}-p_{21}}{x_1-x_2} = A - B + h.o.t.
\] begins with the element $A-B$ followed by higher degree terms in $x_1$ and $x_2$ with coefficients only involving the parameters $\beta_{ij}$ for $i+j>1$. Consider the homomorphism
\[ \tau: \Rminus = \Rminuslong \longrightarrow \wR, \]
given by expanding $\rho_0$ and $\rho_1$ as power series, so that
\[ \tau(\rho_0) = A-B + h.o.t., \quad \quad \tau(\rho_1) = 1 + A E_1 + h.o.t.
\]
To show that $\tau$ is injective, compose $\tau$ with the involution of $\wR$ that sends the  generator $B=\beta_{0,1}$ to $\tau(\rho_0)$, and fixes all other generators (generators $E_1,E_2$ and  $\beta_{i,j}$ for $(i,j)\not= (0,1))$. So the question reduces to showing  injectivity of the map
\[ \pi : \Rminus \longrightarrow \wR, \quad \pi(\rho_0) = B, \quad \pi(\rho_1) = \tau(\rho_1) = 1 + AE_1 + h.o.t., \quad \pi(E_i)=E_i, \ i=1,2. \]
Now consider any homogeneous element of degree $2n$ in the kernel of $\pi$
\[ \sum_{i,j,k, \, \, i+2j=n+k} \rho_0^k f_{ijk}(\rho_1) E_1^i E_2^j,  \]
where $f_{ijk}(\rho_1)$ is a Laurent polynomial in $\rho_1$. Mapping to $\wR$ under $\pi$ and observing that the elements $B^k$ are linearly independent  over the subring of $\tilde{R}$ given by power series in $E_1$ and $E_2$ with values in the polynomial algebra $\Z[\beta_{i,j}, (i,j) \neq (0,1)]$, we deduce that for any fixed $k \geq 0$ one has relations
\[ 0 = \sum_{i+2j=2n+k} f_{ijk}(\tau(\rho_1)) E_1^i E_2^j. \]
Notice that for any $k$, the above expression is a finite sum.  So by multiplying by a suitable power of $\tau(\rho_1)$, we may assume that each Laurent polynomial $f_{ijk}(\rho_1)$ is in fact a polynomial in $\rho_1$. The algebraic independence of the classes $\tau(\rho_1), E_1, E_2$ easily implies that each $f_{ijk}(\rho_1)$ must be trivial. In other words, the map $\pi$ is injective, which is what we wanted to prove.
\end{proof}

\begin{cor} The power series homomorphism $\ovR \lra \wR$ takes $\ovR$ isomorphically onto the subring $R$ of $\wR$. Moreover, the ring $R$ has a basis over $\Z[E_1, E_2]$ given by
\[ \mathscr{B} := \{ \rho_1^{n_1}\rho_0^{n_2} \rho^{n_3} , \ \
n_1\in \{0,1\}, n_2\in \Z_+, n_3 \in \Z \}.
\]
\end{cor}

\begin{proof} By definition, the image of $\ovR$ in $\wR$ is equal to the ring $R$. Now both rings $\ovR$ and $R$ are generated as modules over $\Z[E_1, E_2]$ by the set of elements of $\mathscr{B}$.
To be more precise, both rings $\ovR$ and $R$ have a collection of generators $\mathscr{B}(\ovR)$ and $\mathscr{B}(R)$, respectively, as defined above that are compatible under the map from $\ovR$ to $R$.
Hence, to demonstrate the isomorphism between $\ovR$ and $R$, it is sufficient to show that the elements $\mathscr{B}(R)$ are linearly independent over $\Z[E_1, E_2]$ when seen as elements in $R$, thereby showing that the elements $\mathscr{B}(R)$ form a $\Z[E_1, E_2]$-module basis of $R$. It follows from this that the collection $\mathscr{B}(\ovR)$ also forms a $\Z[E_1,E_2]$-module basis of $\ovR$, and consequently, that the map from $\ovR$ to $R$ is an isomorphism.

\medskip
In what follows, we will actually show that the elements $\mathscr{B}(R)$ are linearly independent over $\Z\llbracket E_1, E_2 \rrbracket$ in the larger ring $\Rminus$, once we observe that the ring $R$ is contained in the image of $\Rminus \subset \wR$. For this it suffices to show that $\rho^{-1}$ is in $\Rminus$, which follows from formula
(\ref{eq:express_rho}) that expresses $\rho^{-1}$ as a power series in $E_1$ and $E_2$ with polynomial coefficients in $\rho_0, \rho_1^{\pm 1}$:
\[ \rho^{-1} = -(\rho_1^2 - E_1 \rho_1 \rho_0 + E_2 \rho_0^2)^{-1} = -\rho_1^{-2}(1-E_1  \rho_0\rho_1^{-1} + E_2 \rho_0^2\rho_1^{-2})^{-1}, \]
and then formally expanding the inverse as power series. This shows that the inclusion $R \subset \wR$ factors through the subring $\Rminus$.

It remains to show linear independence of the elements $\mathscr{B}(R)$ over $\Z\llbracket E_1, E_2 \rrbracket$ inside $\Rminus$. Since the set of elements $\{ \rho_0^{n_2} \}$ are linearly independent over $\Z\llbracket E_1, E_2 \rrbracket$, it is sufficient to show that the sub-collection of $\mathscr{B}(R)$ given by the elements $\{ \rho_1^{n_1} \rho^{n_3} \}$ is linearly independent over $\Z[\rho_0]\llbracket E_1, E_2 \rrbracket$. Let us consider a homogeneous relation
\begin{equation} \label{eq:relation}0 = \sum_{n:=(n_1, n_3)} A_n(\rho_0, E_1, E_2) \rho_1^{n_1}\rho^{n_3}, \end{equation}
where the indexing set is some finite subset of distinct pairs $n := (n_1, n_3)$ as above with $A_n(\rho_0, E_1, E_2)$ being a homogeneous element of $\Z[\rho_0]\llbracket E_1, E_2\rrbracket$. Reducing relation (\ref{eq:relation}) mod $\rho_0$ and using equation (\ref{eq:express_rho}), we obtain the relation in $\Z[\rho_1^{\pm 1}]\llbracket E_1, E_2 \rrbracket$
\[ 0 = \sum_{n:=(n_1, n_3)} (-1)^{n_3} A_n(0, E_1, E_2) \rho_1^{n_1 + 2n_3}, \]
which is clearly true only if $A_n(0,E_1,E_2) = 0$ for all $n$. This condition implies that each $A_n(\rho_0,E_1,E_2)$ is divisible by $\rho_0$. We may therefore factor $\rho_0$ out of the entire relation (\ref{eq:relation}), and repeat the argument (note that $\rho_0$ is not a zero divisor). This shows that  $A_n(\rho_0,E_1,E_2)$ must be trivial for all $n$, which is what we needed to establish.

\end{proof}
\begin{remark}
The inclusion $R \subset \Rminus$ is dense in the power series ring topology. In order to show this, it is sufficient to show that $\rho_1^{-1}$ can be described in terms of a power series in $E_1$ and $E_2$, with coefficients that are polynomials in $\rho_0, \rho_1, \rho^{\pm 1}$. This follows from formula (\ref{eq:express_rho}) which implies that
\[ \rho_1^{-1} = -\rho^{-1}(1 + \rho_0^2\rho^{-1} E_2)^{-1} (\rho_1 - \rho_0 E_1). \]
\end{remark}

Notice that in addition to the chain of ring inclusions in formulas (\ref{eq:ring_inclusions})-(\ref{eq:ring_inclusions_2}), there is also a chain of inclusions
\begin{equation}\label{eq:ring_inclusions_3}
    R \ \subset \  \wR \ \subset \ R' \ \subset \ R''.
\end{equation}

\vspace{0.1in}

The example 6 above for the evaluation of the  $\Theta$-foam is straightforward to generalize to $GL(N)$, where $\Theta$-foam has a disk of thickness $N$ with $N$ disks of thickness one attached to it,  carrying $n_1,\dots,n_N$ dots, respectively, where we can assume $n_1\ge n_2\ge \dots \ge n_N$, see Figure~\ref{fig:theta_foam_gln}.

\begin{figure}[h]
\begin{center}
\includegraphics[scale=0.70]{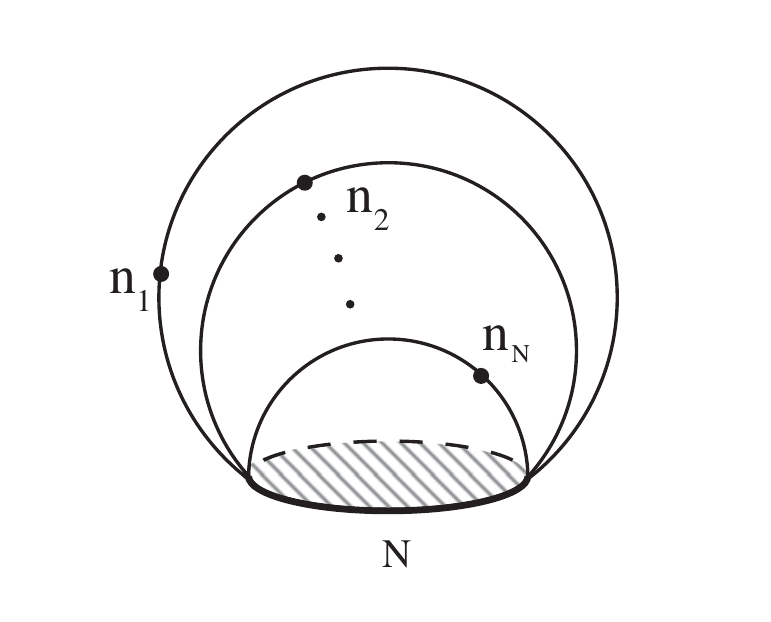}
\caption{\label{fig:theta_foam_gln} $\Theta$-foam for $GL(N)$ }
\end{center}
\end{figure}
Let $\lambda_i=n_i-N+i,$ so that $\lambda=(\lambda_1,\dots, \lambda_N)$ is a partition iff $n_i>n_{i+1}$ for all $i$. Denote this foam by $\Theta_{\lambda}$. One can compute the foam evaluation

\begin{equation}
    \angf{\Theta_{\lambda}} = \pm \sum_{\sigma\in S_N}
    \frac{(-1)^{\ell(\sigma)}\prod_{i=1}^N x_{\sigma(i)}^{n_i}}
    {\prod_{i<j}(x_i-x_j)}
    \prod_{i\not=j} p_{ij} = \pm s_{\lambda}  \prod_{i\not=j} p_{ij},
\end{equation}
where $s_{\lambda}$ is the $GL(N)$ Schur function for the partition $\lambda$.
The last equality holds if $\lambda$ is a partition, otherwise $\angf{\Theta_{\lambda}}=0$. One can argue that our deformation does not go far enough, since it does not deform Schur functions in an interesting way and only scales them by the product of $p_{ij}$'s. At least it does deform the value of the thin 2-sphere with dots and other closed surfaces in a non-trivial way.


\subsection{Skein relations} \label{sec:skein_relations}

\begin{prop}
   The skein relation (\emph{singular neck-cutting relation}) in Figure~\ref{fig:sneck_cut_1} holds.
\end{prop}

\begin{figure}[h]
\begin{center}
\includegraphics[scale=0.6]{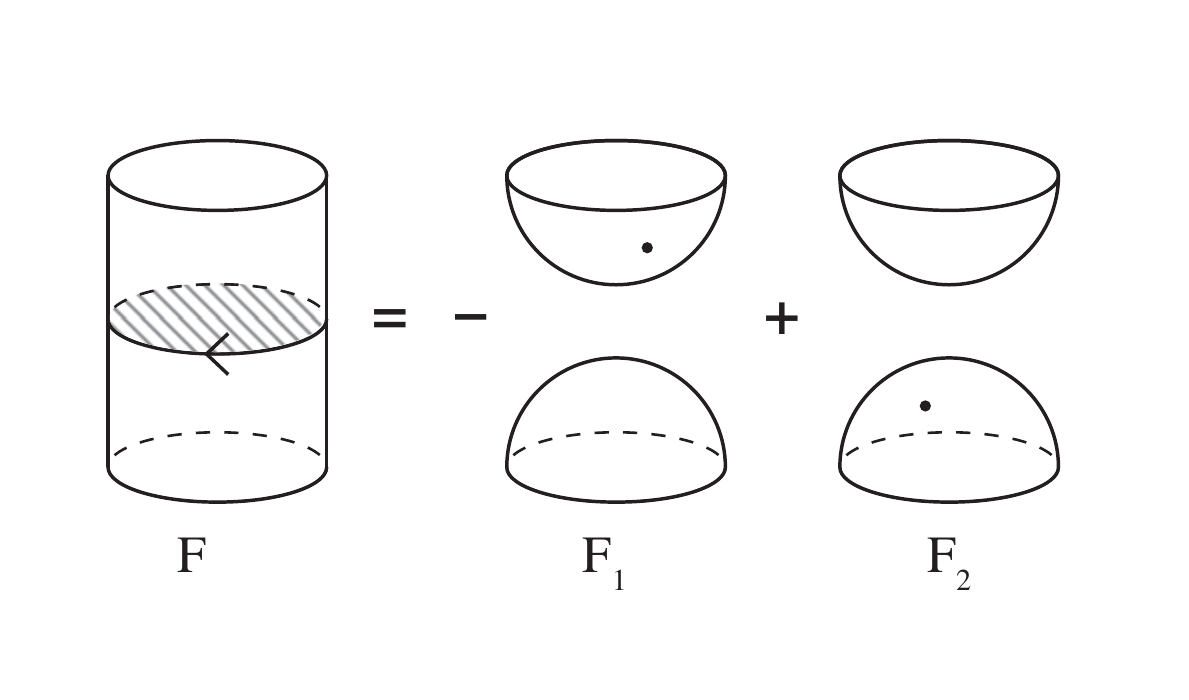}
\caption{\label{fig:sneck_cut_1} Singular neck-cutting relation}
\end{center}
\end{figure}

\emph{Proof:} Coloring $c$ of $F$ induces a coloring $c'$ of $F_1, F_2$ (the latter two foams differ only by dot placement, and we use $c'$ to denote corresponding coloring of both foams). Coloring $c'$ has opposite colors on the two disks of $F_1$ (and $F_2$). If a coloring $c_1$ of $F_1$ and $F_2$ has the same color on the two disks, $\angf{F_1,c_1}=\angf{F_2,c_1},$ since dots will contribute with the same $x_i$, $i\in\{1,2\}$, to the evaluations, and this coloring will not contribute to the difference $\angf{F_2}-\angf{F_1}$. Thus, we can restrict to colorings $c'$ as above, in bijection with colorings $c$ of $F$.

\begin{figure}[h]
\begin{center}
\includegraphics[scale=0.7]{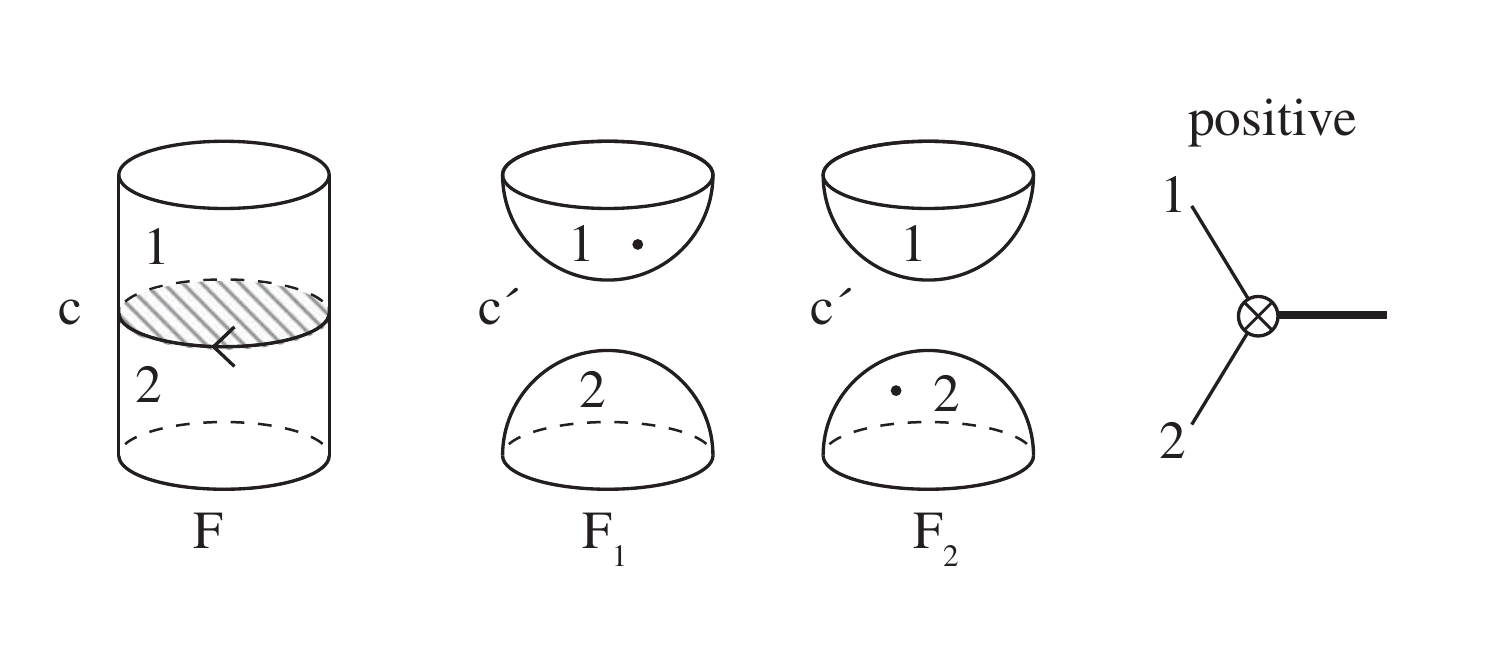}
\caption{\label{fig:sneck_cut_2} When top facet is colored $1$ }
\end{center}
\end{figure}
If the top facet of $c$ is colored $1$, see Figure~\ref{fig:sneck_cut_2}, then the circle of $F$ in the figure is positive and $\theta^+(c)=\theta^+(c')+1$. Also, $\chi_2(F,c)=\chi_2(F_1,c')=\chi_2(F_2,c')$, so that $- (-1)^{s'(F,c)}=(-1)^{'(F_1,c')} = (-1)^{'(F_2,c')}.$

We have $\chi_{12}(F,c)=\chi_{12}(F_1,c')-2$, so that $\angf{F,c}$ has an additional $(x_1-x_2)$
in the numerator, compared to $\angf{F_1,c'}$ and $\angf{F_2,c'}$. Due to a dot on facet colored $1$ there's an extra $x_1$ in $\angf{F_1,c'}$ and an extra $x_2$ in $\angf{F_2,c'}$. More accurately, we can write
\[ \angf{F,c}=-(x_1-x_2)y, \ \angf{F_1,c'} = x_1 y, \ \angf{F_2,c'} = x_2 y
\]
for some $y$, so that $\angf{F,c} = - \angf{F_1,c'} + \angf{F_2,c'}$.

\begin{figure}[h]
\begin{center}
\includegraphics[scale=0.7]{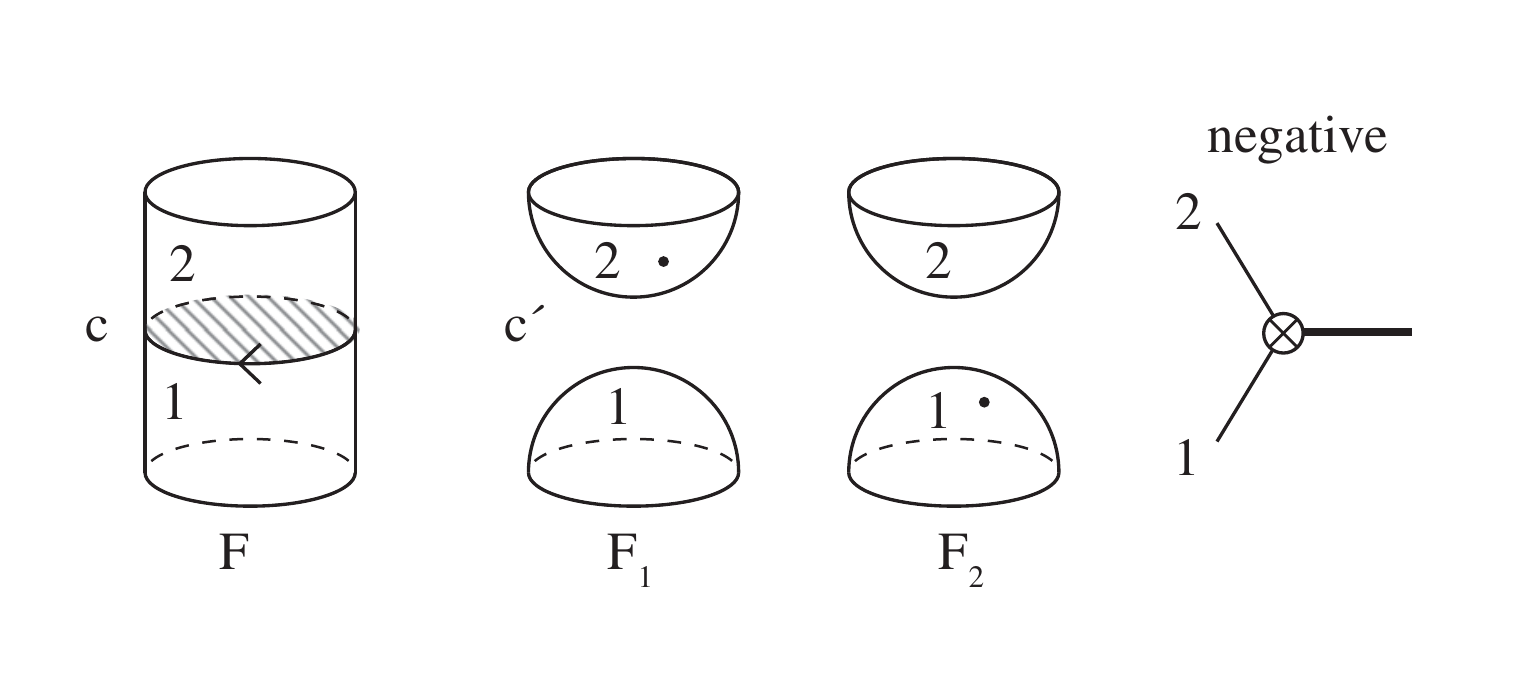}
\caption{\label{fig:sneck_cut_3} When top facet is colored $2$}
\end{center}
\end{figure}

The other case is when the top facet of $F$ is colored $2$ by $c$, see Figure~\ref{fig:sneck_cut_3}. In this case the singular circle of $F$ in the figure is negative for the coloring $c$, so that $\theta^+(c)=\theta^+(c')$ and  $s'(F,c)=s'(F_1,c')=s'(F_2,c').$ This change of sign is balanced by the opposite coloring of the two disks in $F_1,F_2$, so that
\[ \angf{F,c}=(x_1-x_2)y, \ \angf{F_1,c'} = x_2 y, \ \angf{F_2,c'} = x_1 y
\]
for some $y$, and we still have $\angf{F,c} = - \angf{F_1,c'} + \angf{F_2,c'}$. Summing over all $c$ implies the proposition. $\square$

Reversing the orientation of the singular circle (and hence of the entire connected component of $F$) changes the signs in the relation, see Figure~\ref{fig:sneck_cut_4} and Proposition~\ref{prop:reverse_orient} below.

\begin{figure}[h]
\begin{center}
\includegraphics[scale=0.6]{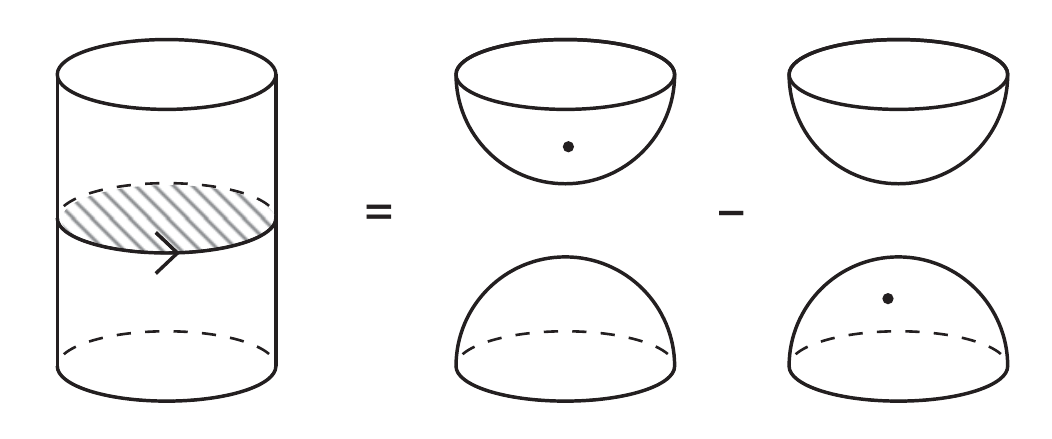}
\caption{\label{fig:sneck_cut_4} Relation for the other orientation}
\end{center}
\end{figure}

\begin{prop} \label{prop_double_disks}
   The skein relation (\emph{canceling double disks}) in Figure~\ref{fig:disk_cancel} holds.
\end{prop}

\begin{figure}[h]
\begin{center}
\includegraphics[scale=0.7]{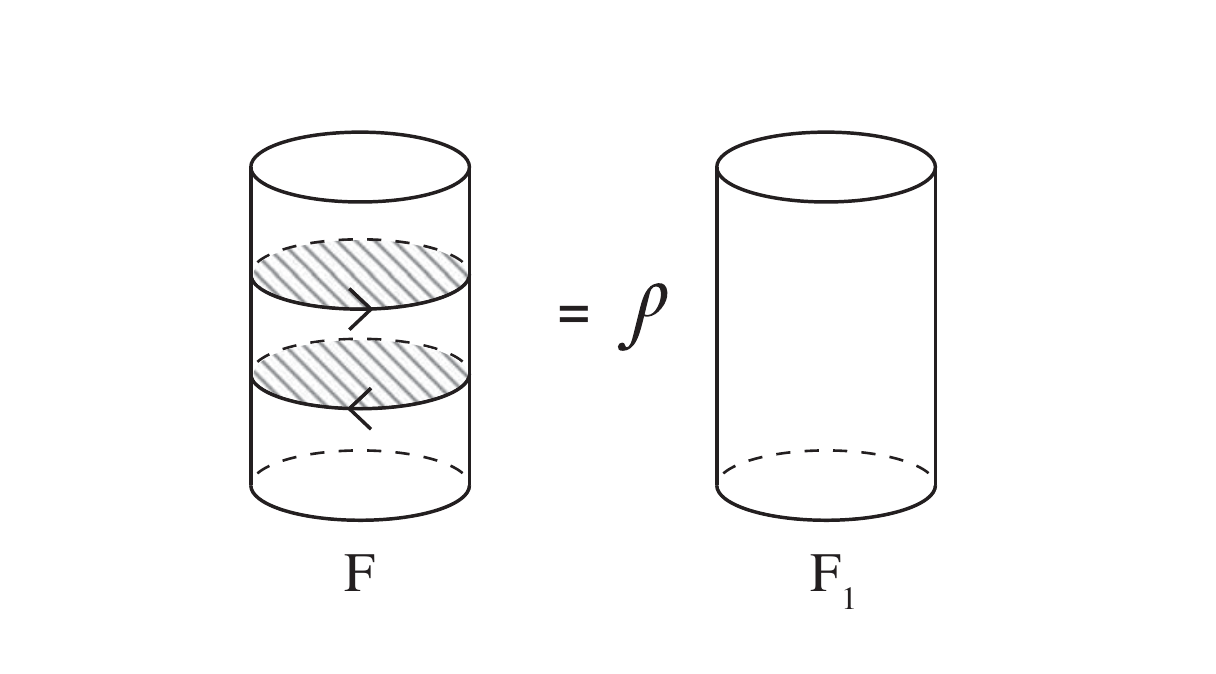}
\caption{\label{fig:disk_cancel} Canceling parallel double disks}
\end{center}
\end{figure}

\emph{Proof:} There is a bijection between colorings $c$ of $F$ and colorings $c_1$ of $F_1,$ see Figure~\ref{fig:color_dd}.

\begin{figure}[h]
\begin{center}
\includegraphics[scale=0.7]{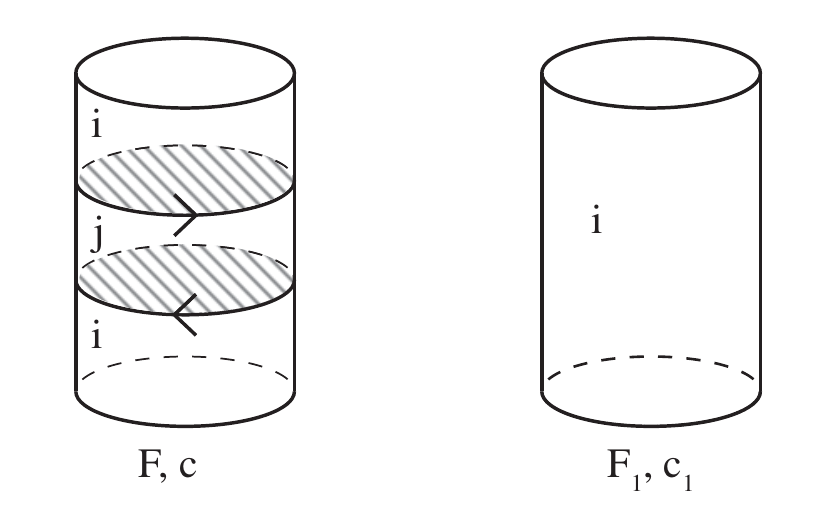}
\caption{\label{fig:color_dd} A coloring $c$ of $F$ and the corresponding coloring $c_1$ of $F_1$}
\end{center}
\end{figure}

One checks that $\chi_k(F_1,c_1)=\chi_k(F,c)+2$, $k=1,2$, and $\chi_{12}(F,c) = \chi_{12}(F_1,c_1)$. For any coloring, $\theta^+(c')\equiv \theta^+(c) (\mathrm{mod}\ 2),$ since the two singular circles in $F_1$  have the same parity, and
$(-1)^{s'(F,c)} = - (-1)^{s'(F_1,c_1)}$. Comparing the contributions,
\[ \angf{F,c} = - p_{12}p_{21} \angf{F_1,c_1} =
\rho \angf{F_1,c_1}.
\]
Summing over all $c$, the result follows.
$\square$

Since $\rho$ is invertible, this relation shows that either of the two foams in Figure~\ref{fig:disk_cancel} can be written as the other foam times $\rho^{\pm 1}$.

Reversing orientation of the two singular circles on the left hand side of Figure~\ref{fig:disk_cancel} gives a similar skein relation, with no sign added since the parity of the number of singular circles is the same on both sides of the relation.

\begin{prop}\label{prop:nc}
   The skein relation (\emph{neck-cutting relation}) in Figure~\ref{fig:neck_cut_1} holds.
\end{prop}

\begin{figure}[h]
\begin{center}
\includegraphics[scale=0.75]{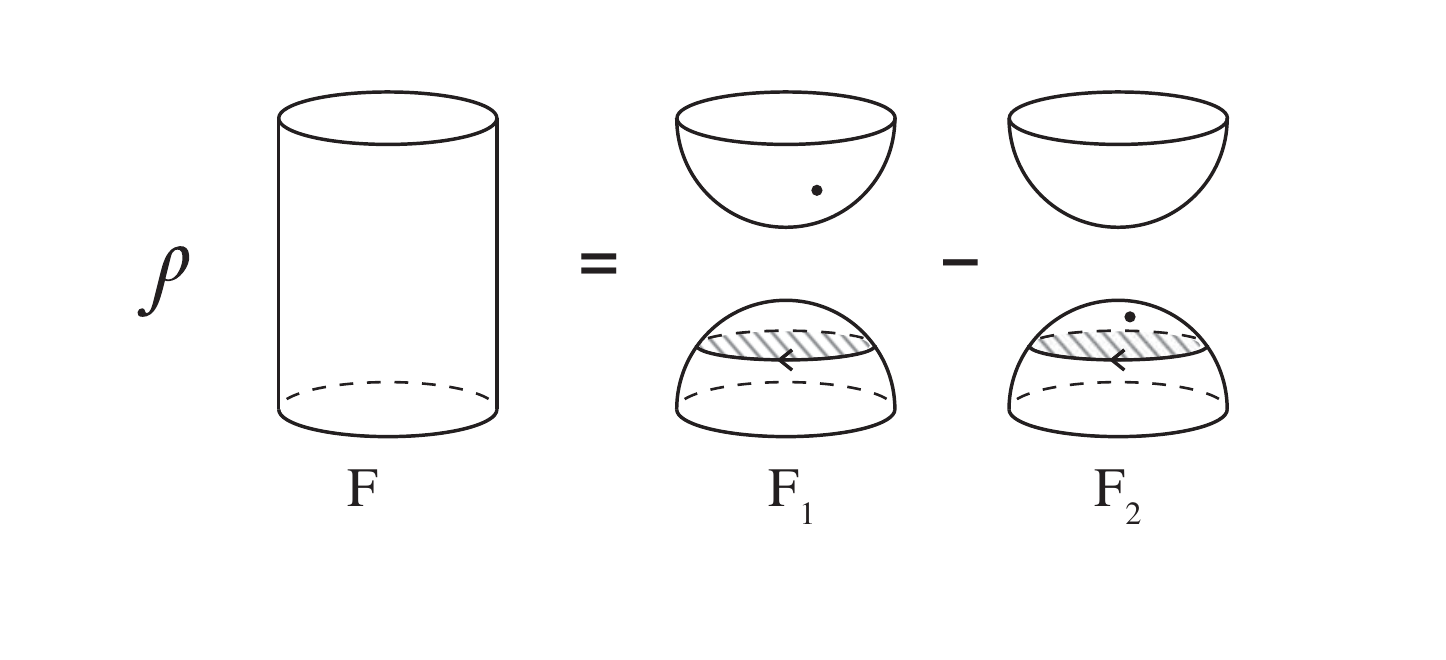}
\caption{\label{fig:neck_cut_1} Neck-cutting relation }
\end{center}
\end{figure}

Again, that $\rho$ is invertible, and the relation allows us to do a surgery on an annulus which is part of a thin facet of $F$.

\emph{Proof:} Apply Figure~\ref{fig:disk_cancel} relation to pass to a tube with two double disks and then use Figure~\ref{fig:sneck_cut_4}  relation to do surgery on the top double disk.
$\square$

Doing the surgery on the bottom double disk using Figure~\ref{fig:sneck_cut_1} results in a similar relation, depicted in Figure~\ref{fig:neck_cut_2} where singular disks now appear at the top rather than the bottom on the right hand side.

\begin{figure}[h]
\begin{center}
\includegraphics[scale=0.75]{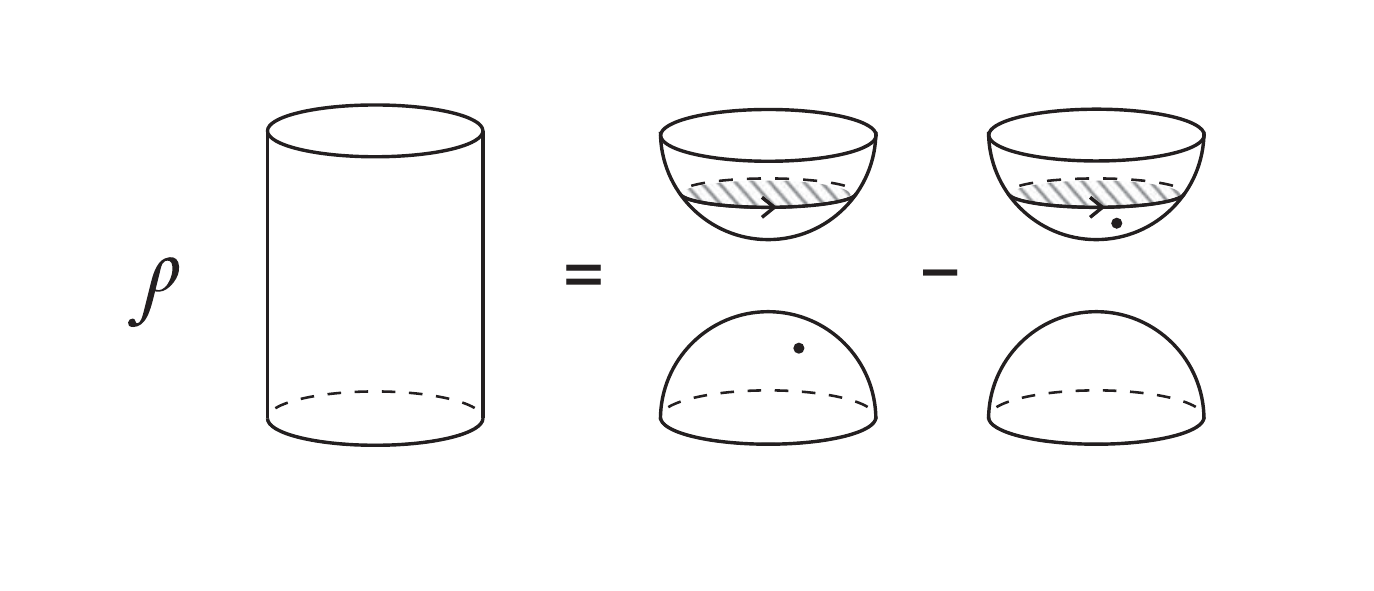}
\caption{\label{fig:neck_cut_2} Neck-cutting relation with double disks at the top}
\end{center}
\end{figure}

\begin{prop} \label{prop_disk_flip} If a double disk $D^2$ bounding a singular circle in a foam $F$ can be completed to a 2-sphere without additional interections with $F$, denote by $F_1$ the foam given by removing the 2-disk from $F$ and adding its complement in $\SS^2$, see Figure~\ref{fig:ddisk_flip}. Then $\angf{F}= - \angf{F_1}$.
\end{prop}

\begin{figure}[h]
\begin{center}
\includegraphics[scale=0.70]{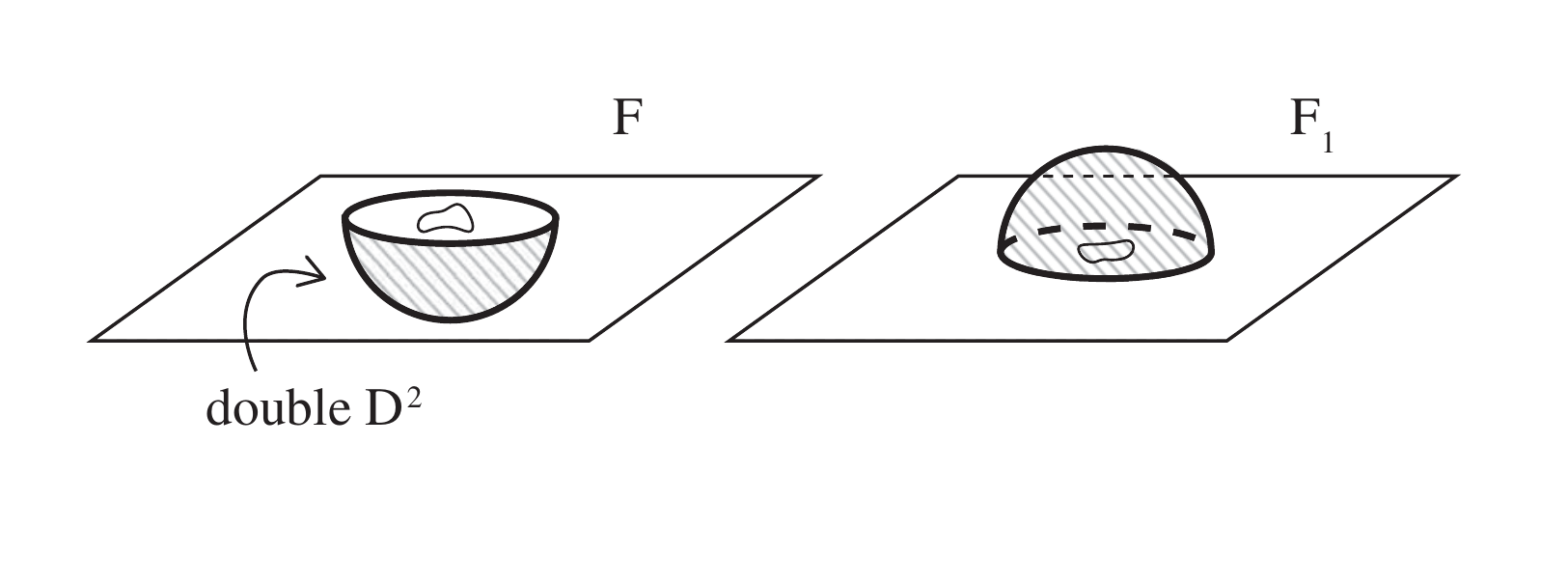}
\caption{\label{fig:ddisk_flip} Double disk flipping}
\end{center}
\end{figure}

\emph{Proof:} There is a bijection between colorings $c$ of $F$ and colorings $c_1$ of $F_1$, with the only difference in evaluations coming from the type of the singular circle, so
that $s'(F,c) = s'(F_1,c_1)\pm 1$ and
$\angf{F,c} = - \angf{F_1,c_1}.$

$\square$

\begin{prop} \label{prop:reverse_orient}
If $\overline{F}$ is a foam $F$ with the reversed orientation of all facets, then $\angf{\overline{F}}=(-1)^k \angf{F}$, where $k$ is the number of singular circles of $F$.
\end{prop}

\emph{Proof:} Each coloring $c$ of $F$ is a coloring of $\overline{F}$ as well, and $\angf{\overline{F},c}=(-1)^k\angf{F,c}$, since $s'(\overline{F},c)=(-1)^k s'(F,c)$ as the type (positive or negative) of each singular circle of $F$ is reversed in $\overline{F}$. Summing over $c$ implies the proposition. $\square$

This proposition can be applied, for instance, to the neck-cutting relation in Figure~\ref{fig:neck_cut_1}. Reversing the orientation of singular circles in $F_1,F_2$ reverses the orientation of all facets as well. Since $F$ has one less singular circle than $F_1,F_2$, there'll be an additional overall minus sign, which can go either to the left or right hand side.

\begin{prop} For a foam $F_2$ with a facet with two dots, the relation
\[ \angf{F_2} = E_1 \angf{F_1} - E_2\angf{F_0}
\]
hold, where $F_1$ and $F_0$ are the foams with one fewer and two fewer dots on the same facet, see Figure~\ref{fig:dot_reduction}.
\end{prop}

\begin{figure}[h]
\begin{center}
\includegraphics[scale=0.70]{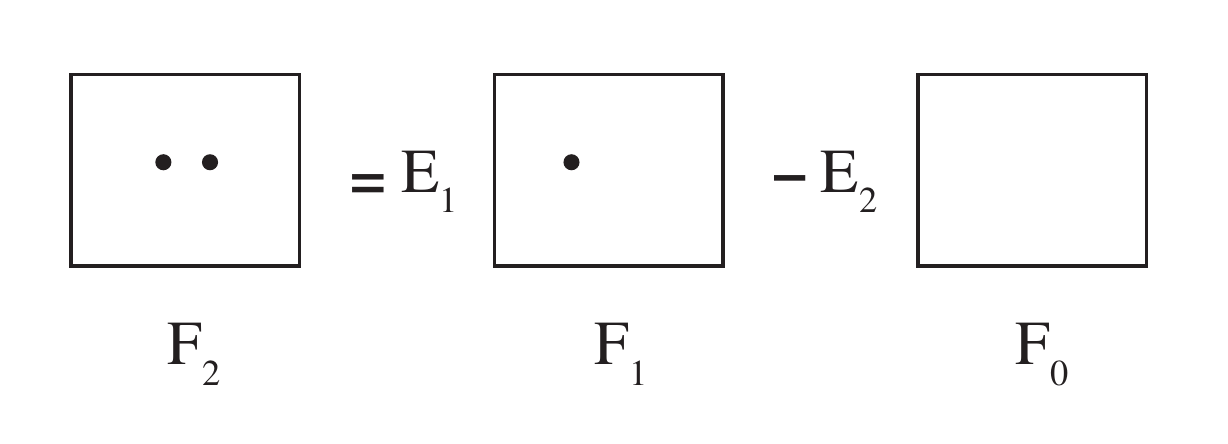}
\caption{\label{fig:dot_reduction} Dot reduction relation}
\end{center}
\end{figure}

\emph{Proof:} Follows, since $x_i^2=E_1 x_i - E_2$ for $i=1,2$. $\square$

\begin{prop} \label{prop_double_nc} (Double facet neck-cutting relation) Evaluations of foams $F$ and $F_1$ in Figure~\ref{fig:neck_cutting_d} satisfy
$\angf{F} = \rho \angf{F_1}$.
\end{prop}

\begin{figure}[h]
\begin{center}
\includegraphics[scale=0.65]{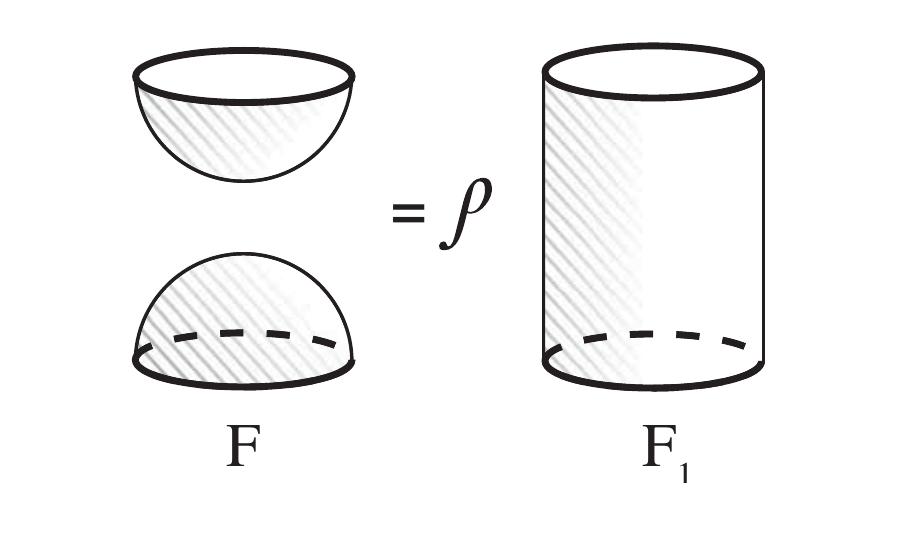}
\caption{\label{fig:neck_cutting_d} Neck cutting on a double facet}
\end{center}
\end{figure}
\emph{Proof:} Again, there's a bijection between colorings $c$ of $F$ and colorings $c_1$ of $F_1$. Difference in the Euler characteristics
$\chi_i(F,c)=\chi_i(F_1,c_1)+2,$ for $i=1,2$, contributes the term $-p_{12}p_{21}=\rho$ to the evaluation of $\angf{F,c}$ compared to that of $\angf{F_1,c_1}$. Summing over $c$ implies the proposition.
$\square$

\begin{prop} (Dot migration relations)
\begin{enumerate}
\item Evaluations of foams $F_1, F_2$ and $F_1$ in Figure~\ref{fig:dot_migr_1} satisfy
$\angf{F_1} + \angf{F_2} = E_1 \angf{F}$.
\item Figure~\ref{fig:dot_migr_2} relation holds.
\end{enumerate}
\end{prop}
\begin{figure}[h]
\begin{center}
\includegraphics[scale=0.7]{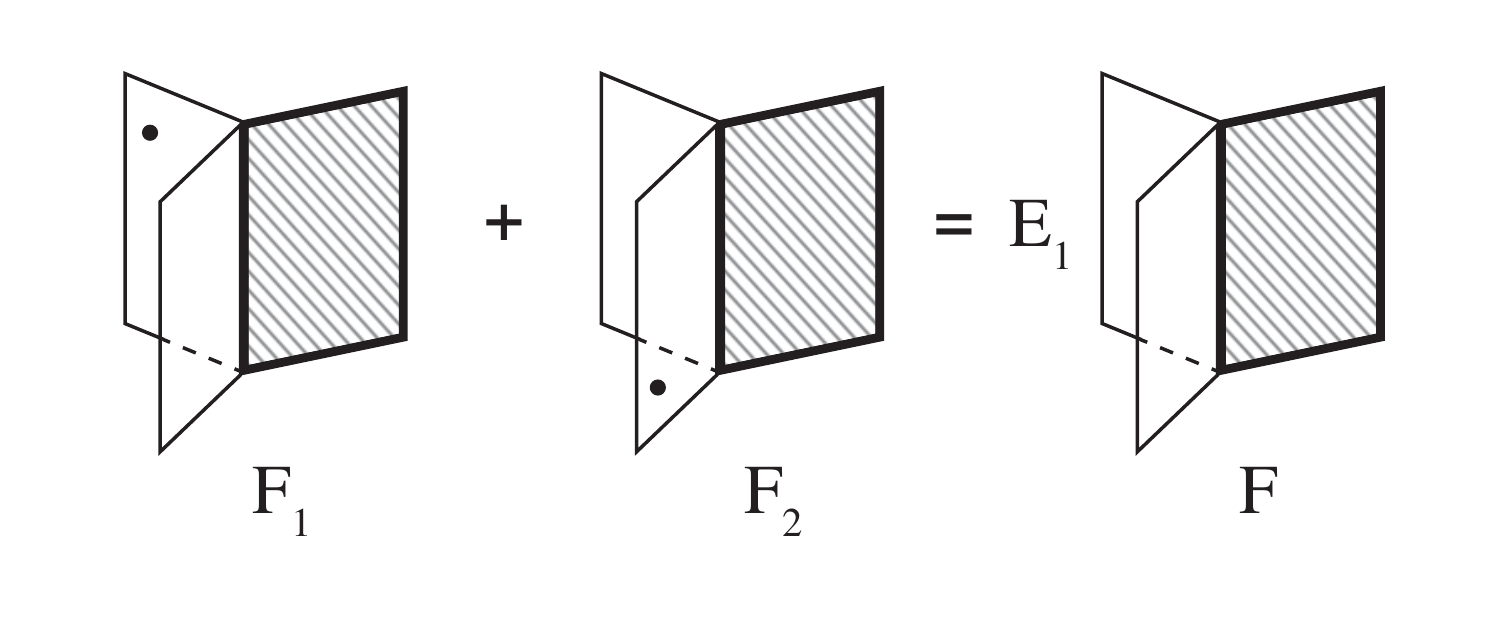}
\caption{\label{fig:dot_migr_1} Dot migration relation}
\end{center}
\end{figure}

\begin{figure}[h]
\begin{center}
\includegraphics[scale=0.70]{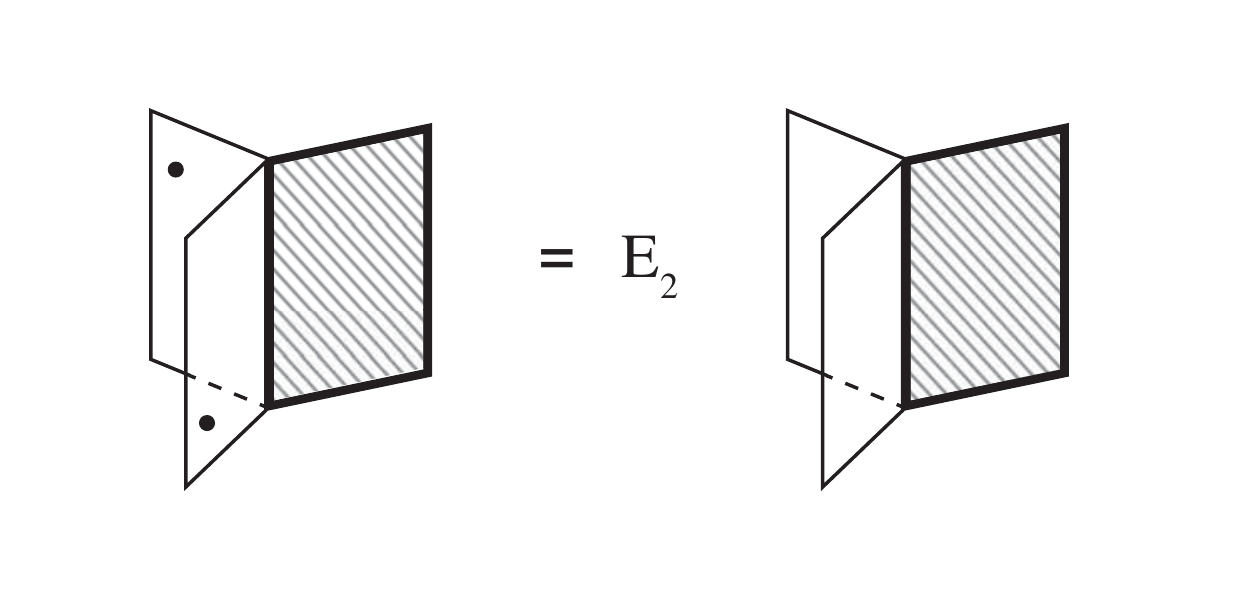}
\caption{\label{fig:dot_migr_2} Second dot migration relation }
\end{center}
\end{figure}

\emph{Proof:} Follows, since these foams differ only by dot placement, any for any coloring $c$ the two facets with dots carry opposite colors. These dots contribute $x_1$ and $x_2$ to the evaluation. Consequently,
\[\angf{F_1,c}+\angf{F_2,c}=(x_1+x_2)\angf{F_3,c} = E_1 \angf{F_3,c}
\]
The same argument implies the second relation.
$\square$

\begin{prop}  Skein relation in Figure~\ref{fig:digon_decomp_1} holds.
\end{prop}
\begin{figure}[h]
\begin{center}
\includegraphics[scale=0.65]{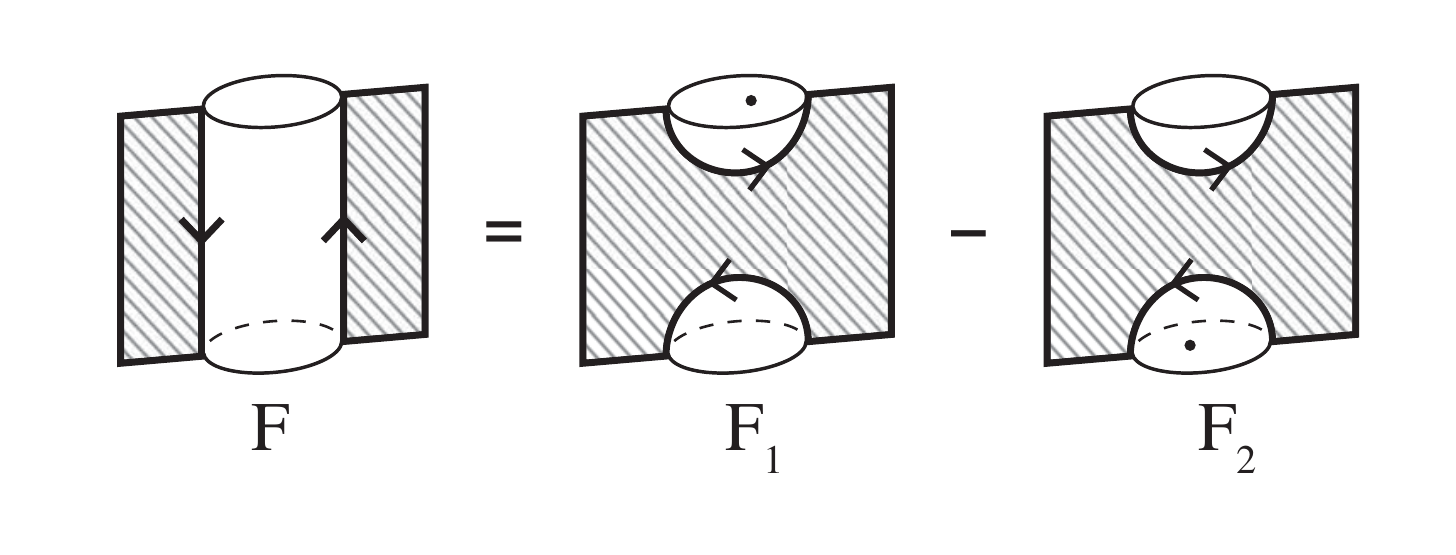}
\caption{\label{fig:digon_decomp_1} Cutting a tube with two singular edges}
\end{center}
\end{figure}

\emph{Proof:} Colorings $c'$ of $F_1,F_2$ that don't come from colorings of $F$ have the property that the front thin bottom and back thin top facets are colored by the same color, see Figure~\ref{fig:digon_decomp_2} left.

\begin{figure}[h]
\begin{center}
\includegraphics[scale=0.65]{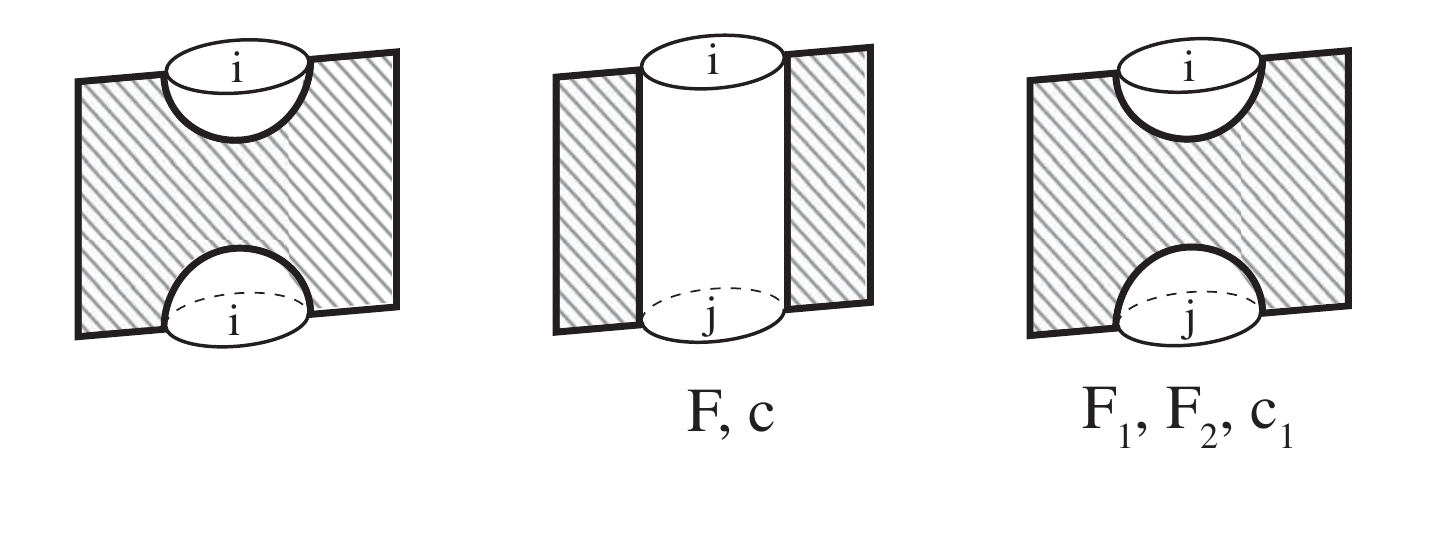}
\caption{\label{fig:digon_decomp_2}  Left: a coloring of $F_1,F_2$ that does not come from a coloring of $F$. Right: a coloring $c'$ of $F$ and induced coloring $c_1$ of $F_1,F_2$.}
\end{center}
\end{figure}
The dots on $F_1,F_2$ will have the same color and these terms will cancel out from the difference, with  $\angf{F_1,c'}-\angf{F_2,c'}=0$.

The remaining colorings $c_1$ of $F_1,F_2$ are in bijection with colorings $c$ of $F$, see Figure~\ref{fig:digon_decomp_2} right. For these colorings we have
\[ \chi_{12}(F,c) = \chi_{12}(F_k,c_1)-2, \
\chi_{\ell}(F,c) = \chi_{\ell}(F_k,c_1), \ \ell,k=1,2.
\]
The rest of the computation is similar to that in the proof of Proposition~\ref{prop:nc}. If $i=1,j=2$, one
checks that $\theta^+(c_1)=\theta^+(c)$ and
the signs $(-1)^{s'}$ are the same in the three evaluations. Due to difference in $\chi_{12}$, the evaluation $\angf{F,c}$ will acquire $(x_1-x_2)$ in the numerator compared to the other two foams. This will be matched by the dots, contributing $x_1$ to $\angf{F_1,c_1}$ and $x_2$ to $\angf{F_2,c_1}$, correspondingly.

If $i=2,j=1$, there will be sign difference
$(-1)^{s'(F,c)}=-(-1)^{s'(F_k,c_1)},$ $k=1,2$. Dots will now contribute $x$'s with the opposite indices to the evaluations of $F_1,F_2$, canceling the sign difference, so that again $\angf{F,c}=\angf{F_1,c_1}-\angf{F_2,c_1}$.

$\square$

This relation with the opposite singular circles orientation, see Figure~\ref{fig:digon_decomp_3}, can be obtained from that in Figure~\ref{fig:digon_decomp_1} by looking at foams there from the opposite side of the plane.
Furthermore, rotating foams in Figure~\ref{fig:digon_decomp_1} by $180^{\circ}$ (or using dot migration relation  in  Figure~\ref{fig:dot_migr_1} twice) yields a similar to Figure~\ref{fig:digon_decomp_1} relation but with a different distribution of dots across thin facets.
\begin{figure}[h]
\begin{center}
\includegraphics[scale=0.65]{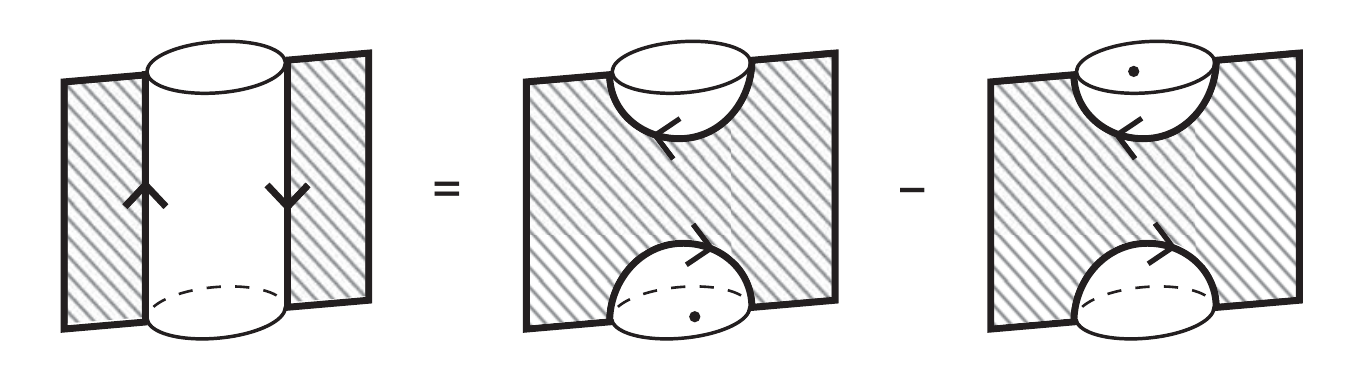}
\caption{\label{fig:digon_decomp_3} Tube-cutting with the  other orientation.  }
\end{center}
\end{figure}

In a foam $F$, let $\gamma$ be a curve that connects two points on singular lines and lies in a single facet of $F$, see Figure~\ref{fig:gamma_move_1} left. Let us call such a curve a \emph{proper} curve. The foam $F_1=m(\gamma,F)$ on that figure on the right is called the modification of $F$ along $\gamma$.
In the undeformed case, when $p(x,y)=1$,  $GL(2)$ foam evaluation satisfies $\brak{F}=\pm \brak{F_1}$, with the sign depending on orientation of singular edges of $F$, see~\cite[Equation (2.10)]{BHPW}.

\begin{figure}[h]
\begin{center}
\includegraphics[scale=0.70]{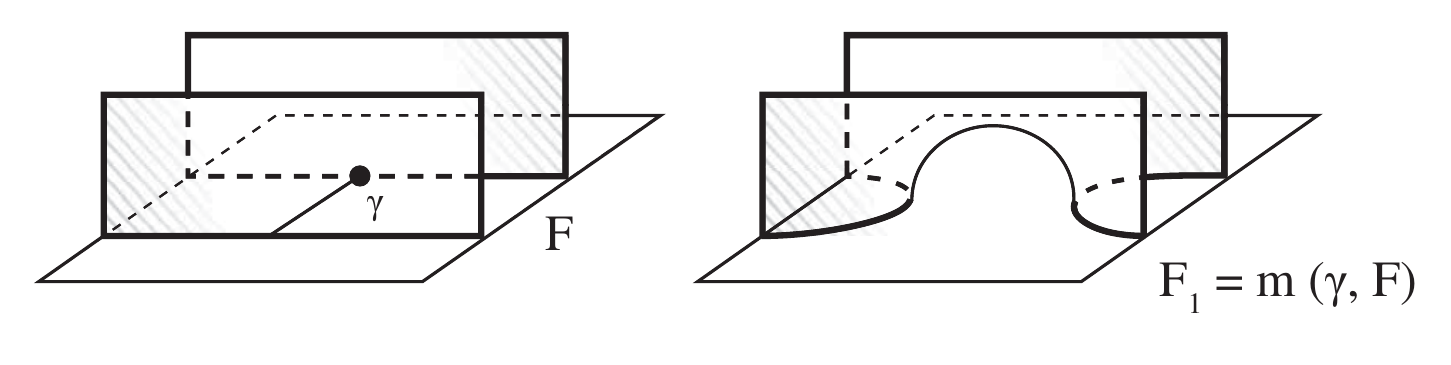}
\caption{\label{fig:gamma_move_1} Modifying foam $F$ along curve $\gamma$ in a thin facet. $\gamma$ connects two points on the singular set of $F$.  }
\end{center}
\end{figure}

The relation is more subtle in our case. We start with orientations of singular edges as shown on Figure~\ref{fig:gamma_move_2}; note that choosing orientation of one edge forces the orientation of the other edge of $F$ shown.
Choose a coloring $c$ of $F$ and denote by $c'$ the corresponding coloring of $F_1$ (there's a bijection between colorings of $F$ and $F_1$), see Figure~\ref{fig:gamma_move_2}.

\begin{figure}[h]
\begin{center}
\includegraphics[scale=0.7]{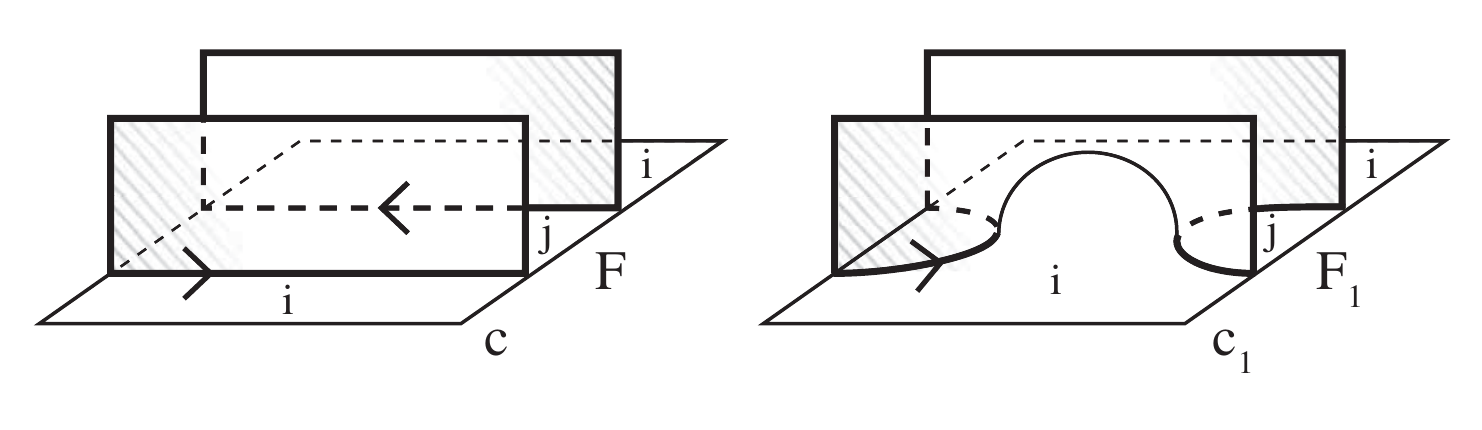}
\caption{\label{fig:gamma_move_2} Notice orientation of singular edges. }
\end{center}
\end{figure}

We have
\[ \chi_{12}(F_1,c_1)=\chi_{12}(F,c), \
\chi_i(c_1) = \chi_i(c)-2, \ \chi_j(c_1)= \chi_j(c).
\]
The number of singular circles of $F_1$ is one less or more of that of $F$, depending on whether the two singular edges shown in $F$ belong to different or the same singular circle.

If $i=1,j=2$ then these circles are negative, they make no contribution to $\theta^+(c)$ and $\theta^+(c_1)$, and
$ s'(F,c)=s'(F_1,c_1),$ since $\chi_2(c)=\chi_2(c_1)$ and $\theta^+(c)=\theta^+(c_1)$. If $i=2,j=1$, the circles are positive and $(-1)^{\theta^+(c)}=-(-1)^{\theta^+(c_1)}.$ Also, the Euler characteristics $\chi_2(c), \chi_2(c_1)$ differ by two and contribute a sign to the differece as well. We again get
$ s'(F,c)=s'(F_1,c_1),$ so that for any coloring
$ s'(F,c)=s'(F_1,c_1).$

Combining these computations,
\begin{equation}\label{eq:gamma_1}
    \angf{F,c}\ = \ p_{ij} \angf{F',c'}.
\end{equation}
Note that $p_{ij}$ is not an element of our ground ring $\pseries{E_1,E_2}$ and summing this equality over all colorings $c$ of $F$ will not get an immediate relation between evaluations of $F$ and $F'$.

We now look at the oppositely oriented case, see Figure~\ref{fig:gamma_move_3}. Circles now carry opposite signs from that of the previous case, and one can check that $s'(F,c) = - s'(F_1,c_1)$ in each of the cases $(i,j)=(1,2)$ and $(i,j)=(2,1).$
\begin{figure}[h]
\begin{center}
\includegraphics[scale=0.7]{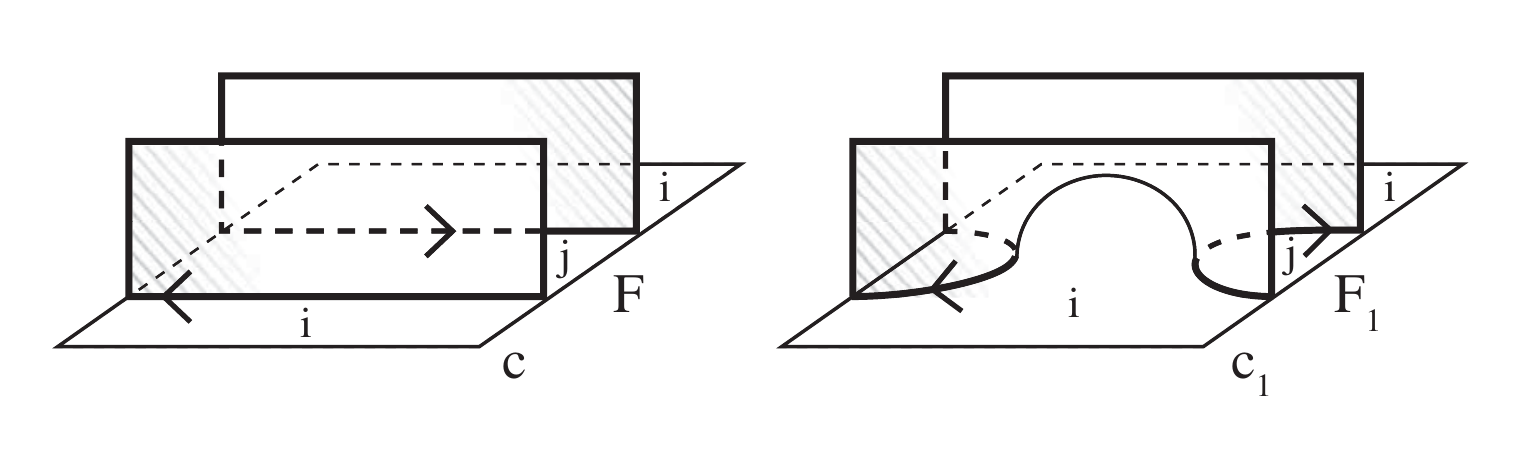}
\caption{\label{fig:gamma_move_3} Singular edges have the opposite orientation from that in the previous figure}
\end{center}
\end{figure}
A similar computation now gives
\begin{equation}\label{eq:gamma_2}
    \angf{F,c} = -p_{ij} \angf{F_1,c_1},
\end{equation}
which is similar to $(\ref{eq:gamma_1})$ but with an additional sign.

Consider the thin surface $F_{12}$ of $F$ and choose a connected component $\Sigma$ in it. Recall that we are looking at modifications of $F$ along proper curves $\gamma$ and now restrict to $\gamma$ on a component $\Sigma$. Notice that the double facets at the endpoints of $\gamma$ are pointing in the same direction relative to $\Sigma$, either both outward or both inward.
Also, if we were to redraw $F$ in Figure~\ref{fig:gamma_move_2} keeping orientations of the singular edges but drawing double facets on the opposite side of $\Sigma$ ('below', rather than 'above'), the type of the diagram would change to the one in Figure~\ref{fig:gamma_move_3}, and vice versa.
Proper disjoint curves or arcs $\gamma_1,\gamma_2\in \Sigma$ are called \emph{complementary} if for a coloring $c$ of $F$ they lie in differently colored regions. Thi property does not depend on the choice of $c$.

To a pair $(\gamma_1,\gamma_2)$ of complementary arcs we assign a sign $s(\gamma_1,\gamma_2).$ Namely, consider the four double facets of $F$ at the endpoints of $\gamma_1$ and $\gamma_2$. If these four facets all point into the same connected component of $\R^3\setminus \Sigma$, we set
$s(\gamma_1,\gamma_2)=1$. Otherwise we define
$s(\gamma_1,\gamma_2)=-1$.

An example when $s(\gamma_1,\gamma_2)=1$ is shown in Figure~\ref{fig:double_gamma_5}. In general, $\gamma_1,\gamma_2$ don't have to have an endpoint on the same singular circle.

\begin{figure}[h]
\begin{center}
\includegraphics[scale=0.80]{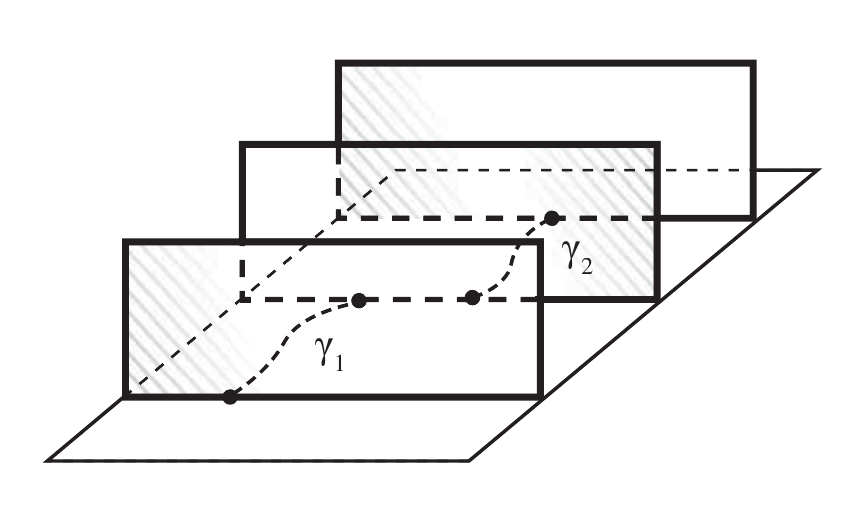}
\caption{\label{fig:double_gamma_5} Complementary proper arcs $\gamma_1,\gamma_2$ with $s(\gamma_1,\gamma_2)=1$}
\end{center}
\end{figure}

Given complementary proper arcs $\gamma_1,\gamma_2$ in $\Sigma$, we can do commuting modifications along $\gamma_1,\gamma_2$ to get from $F$ to the foam
\[ F_2=m(\gamma_1,m(\gamma_2,F)) =m(\gamma_2,m(\gamma_1,F)).
\]
\begin{prop} \label{prop:gammas} For $F$ and $F_2$ as above,
  $\angf{F} = s(\gamma_1,\gamma_2) \rho\cdot  \angf{F_2}.$
\end{prop}

\emph{Proof:} For a coloring $c$ of $F$ curves $\gamma_1$ and $\gamma_2$ lie in differently colored regions of $\Sigma$, say $i$ and $j$-colored regions, $\{i,j\}=\{1,2\}$. When $s(\gamma_1,\gamma_2)=1$, orientations on singular edges will make one of curves $\gamma$ the type in Figure~\ref{fig:gamma_move_2} and the other in Figure~\ref{fig:gamma_move_3}, with $(i,j)$ replaced by $(j,i)$ in one of these two cases.
Using equations (\ref{eq:gamma_1}) and (\ref{eq:gamma_2}), we obtain
$\angf{F,c}= - p_{12}p_{21}\angf{F_2,c_2} = \rho \angf{F_2,c_2}$ for the corresponding coloring $c_2$ of $F_2$.

When $s(\gamma_1,\gamma_2)=-1$, orientations on singular edges will make both $\gamma_1,\gamma_2$ either the type in Figure~\ref{fig:gamma_move_2} or the type in Figure~\ref{fig:gamma_move_3}, with $(j,i)$ in place for $(i,j)$ for one of $\gamma_1,\gamma_2$. This will introduce minus sign, with $\angf{F,c}= - \rho \angf{F_2,c_2}$.

$\square$

This proposition may be generalized in  some cases when one of $\gamma_1,\gamma_2$ is not a proper arc. One would need $\gamma_2$ to be a proper arc in $m(\gamma_1,F)$, in the region of color opposite to that of $\gamma_1$, with a coloring of $F$ naturally converted to a coloring of $m(\gamma_1,F)$. We provide an example of such pair of arcs in Figure~\ref{fig:double_gamma_0} and leave the details to the reader.

\begin{figure}[h]
\begin{center}
\includegraphics[scale=0.70]{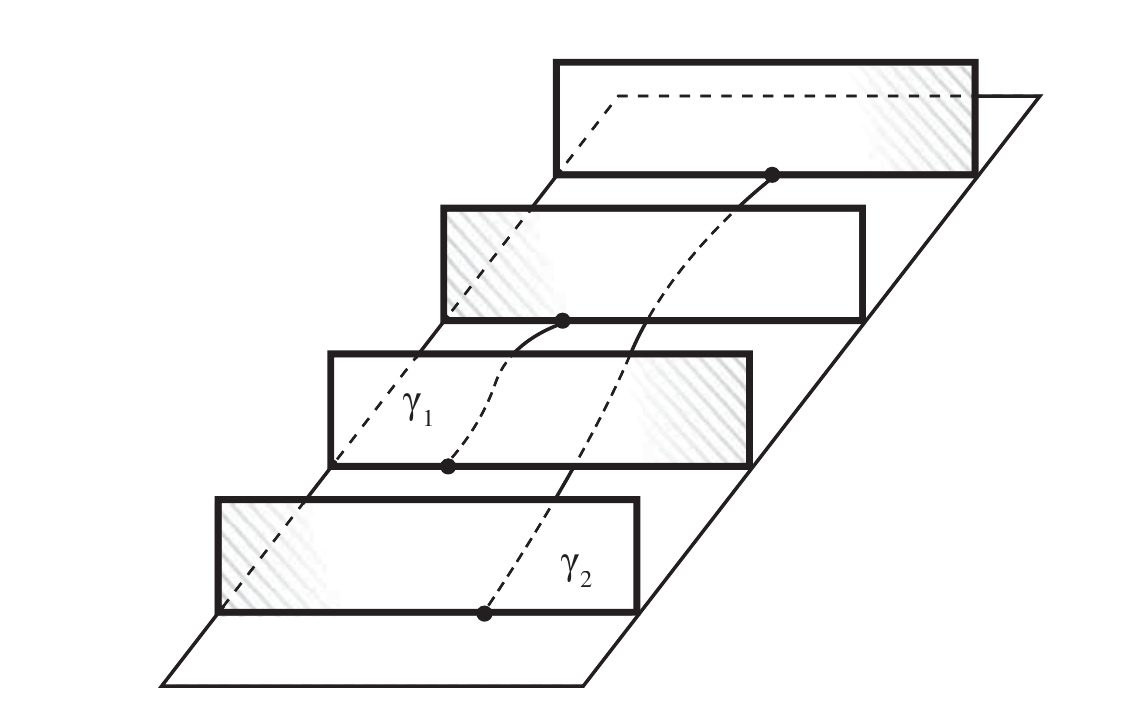}
\caption{\label{fig:double_gamma_0} Arc  $\gamma_2$ is not proper but becomes proper in $m(\gamma_1,F)$}
\end{center}
\end{figure}

\begin{cor}
    $\angf{F} = \rho \angf{F_1}$ for foams $F,F_1$ in Figure~\ref{fig:double_gamma_1}.
\end{cor}

\begin{figure}[h]
\begin{center}
\includegraphics[scale=0.70]{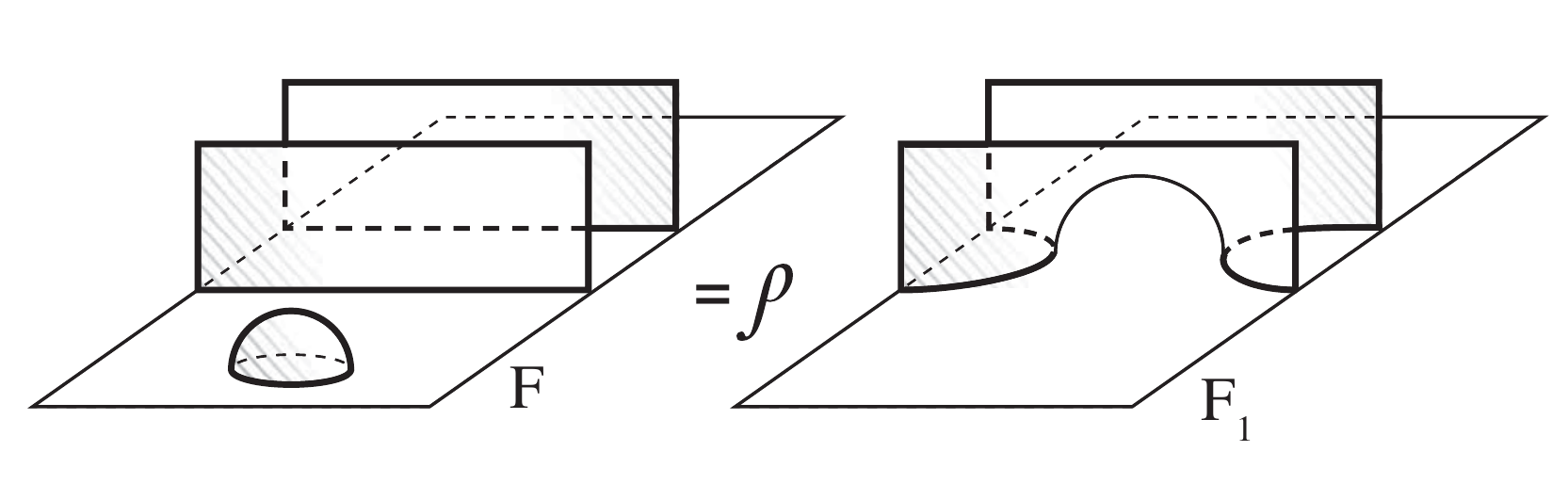}
\caption{\label{fig:double_gamma_1} Notice additional double cap on the foam $F$, used to create a pair of complementary proper arcs on it.}
\end{center}
\end{figure}

This follows from Proposition~\ref{prop:gammas} using the pair of arcs in Figure~\ref{fig:double_gamma_2} with $s(\gamma_1,\gamma_2)=1$. $\square$

\begin{figure}[h]
\begin{center}
\includegraphics[scale=0.70]{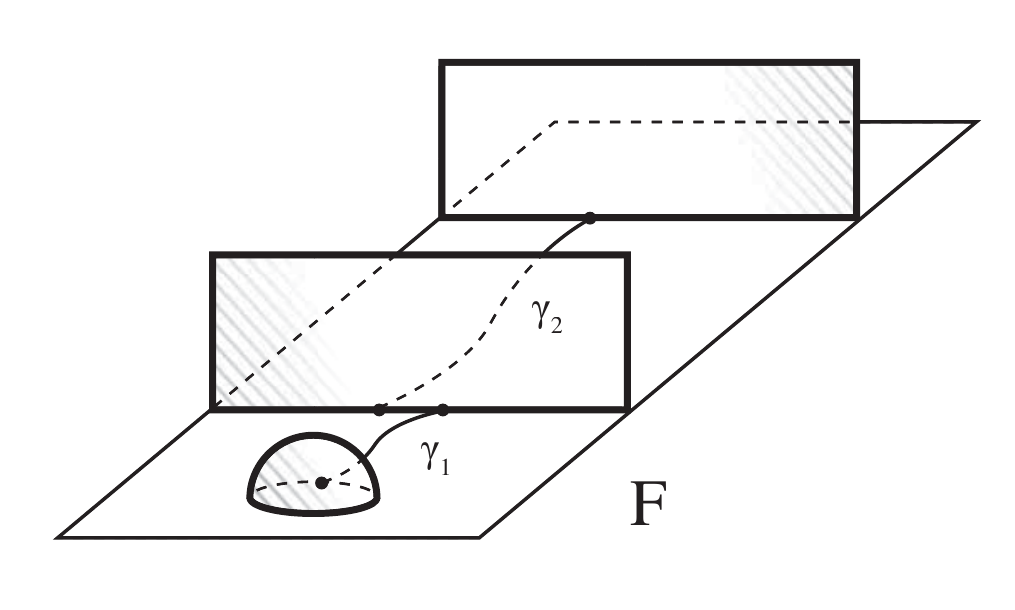}
\caption{\label{fig:double_gamma_2} Arcs $\gamma_1$,$\gamma_2$ are complementary proper with $s(\gamma_1,\gamma_2)=1$  }
\end{center}
\end{figure}

\begin{cor}
    Figure~\ref{fig:double_gamma_3} relation on foam evaluations holds.
\end{cor}
\begin{figure}[h]
\begin{center}
\includegraphics[scale=0.60]{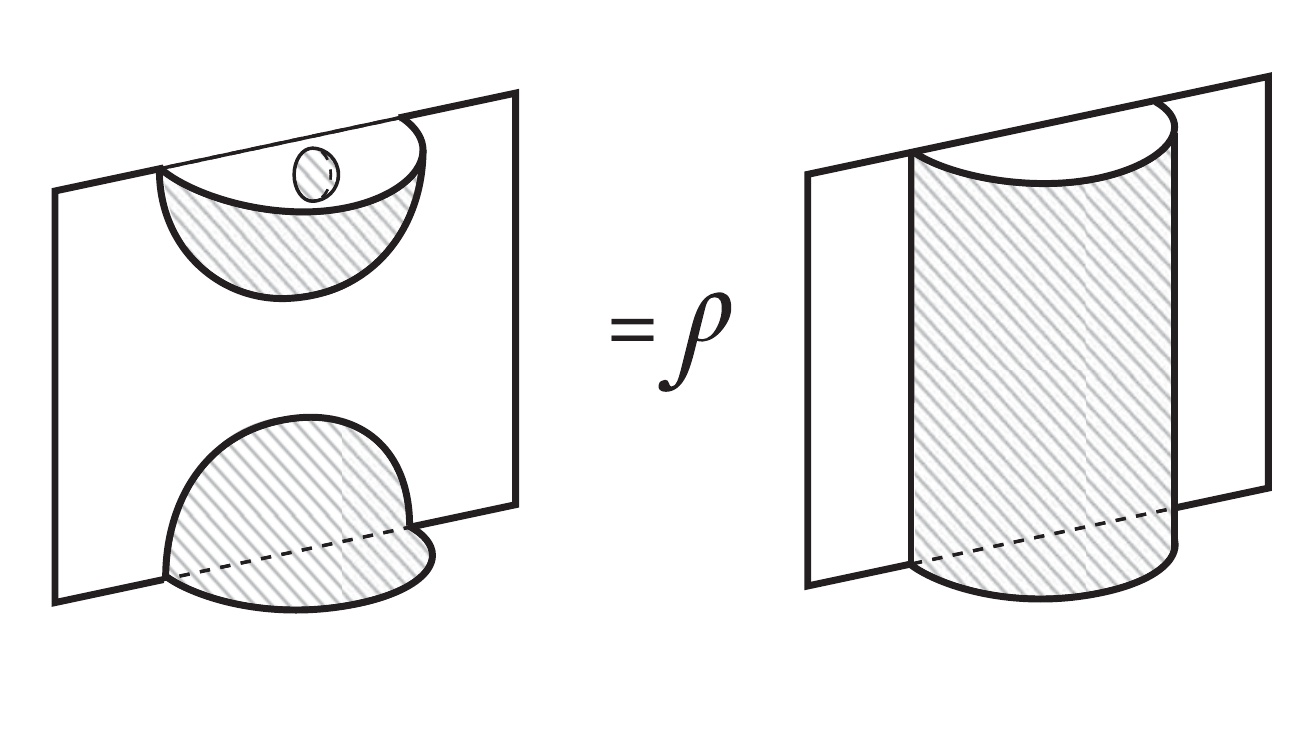}
\caption{\label{fig:double_gamma_3} The bubble on the thin plane on the LHS foam points its double facet toward us, being on the same side of the thin plane as portions of the other double facets shown on the LHS. }
\end{center}
\end{figure}
The corollary follows at one from the previous one. $\square$


\subsection{Prefoams and ground ring reduction}
\label{sec:prefoams}
$\quad $

To prove Proposition~\ref{prop:eval_in_subring} below, it's convenient to introduce the notion of $GL(2)$ prefoam and its evaluation. An (oriented) $GL(2)$ prefoam (or pre-foam) $F$ has the same local structure as a $GL(2)$ foam, but without an embedding into $\R^3$. It has oriented thin and double facets, with facets orientations
compatible along singular edges as in Figure~\ref{fig:gl2_orientations}. In particular, orientations of facets induce orientations of singular circles. Vice versa, an orientation of a singular circle in a connected component of a prefoam will induce orientation on all facets of that component.

Along each singular edge a preferred facet out of two adjacent thin facets is specified. One can encode this choice by an arrow (a normal direction) out of the singular edge and into the thin surface of the pre-foam (the union of its thin facets).
 A pre-foam may carry dots on its thin facets.

 A $GL(2)$ foam $F$ gives rise to a $GL(2)$ prefoam, also denoted $F$. Embedding of foam $F$ in $\R^3$ together with orientation of singular circles induces an order on the two thin facets attached to a given singular circle. Namely, look in the direction of the orientation on the circle and choose the first thin facet counterclockwise starting from the double facet attached to the circle. This is then the preferred facet for the singular circle in the underlying pre-foam $F$.

Coloring $c$ of a pre-foam is defined in the same way as for foams. For each coloring $c$ surfaces $F_1(c)$ and $F_2(c)$ inherit orientations from the facets of $F$ they contain.
Surface $F_{12}(c)$ is orientable as well, say with orientation matching that of thin facets of $F_{12}(c)$ colored $1$ and opposite to that of thin facets colored $2$.

Orientation requirements for facets ensure that each connected component of the thin surface $F_{12}$ of a prefoam $F$ will admit two checkerboard colorings, so that a prefoam $F$ will admit $2^k$ colorings, where $k$ is the number of connected components of $F$.

Given a coloring $c$ of $F$, the preferred thin facet at a singular circle $u$ allows to label the circle \emph{positive} or \emph{negative}, as in Figure~\ref{fig:theta_signs}. Namely, if the preferred facet is colored $1$, the circle is \emph{positive}. If the preferred facet is colored $2$, the circle is negative.

Define $\theta^+(c)=\theta^+_{12}(c)$ as the number of positive singular circles for the coloring $c$.

Thus, in a pre-foam $F$, each singular circle $u$ comes with both an orientation (induced from the orientation of attached facets and, vice versa, determining them) and a choice of preferred thin facet (normal direction to the thin surface $F_{12}$) along $u$. The evaluation of $F$, though, will only depend on the choice of preferred facet at each singular circle, not on its orientation.

Unlike the foam case, in a $GL(2)$ pre-foam we can reverse the thin normal direction (reverse the choice of preferred thin facet) at any subset of its singular circles without making any other changes, such as reversing orientations of facets or singular circles, changing the embedding into $\R^3$, etc.
In a foam, the analogous operation of reversing the cyclic order of facets at a single circle via a simple modification of the embedding is possible only sometimes, see Figure~\ref{fig:ddisk_flip}  for an example.

Recall the chain of inclusions of rings $R\subset \wR
\subset R' \subset R''$ defined in formulas (\ref{eq:ring_inclusions})-(\ref{eq:ring_inclusions_2}) and (\ref{eq:def_ground_ring}).


Now, to a coloring $c$ of a prefoam $F$ we assign an element $\angf{F,c}\in R''$ using the  formula (\ref{eq:def_eval_gl2}).  Furthermore,
 define $\angf{F}$ via the formula (\ref{eq:def_eval_gl3}).

\begin{prop}\label{prop:prefoam_eval_in_subring_1}
   Evaluation $\angf{F}$ of any $GL(2)$ prefoam $F$ belongs to the subring  $\pseries{E_1,E_2}$ of $R''$.
\end{prop}

\emph{Proof:} Our proof of this result for foams, Theorem~\ref{thm:no_denoms_GL2}, extends to prefoams without change. $\square$

\begin{prop}\label{prop:prefoam_eval_in_subring_2}
   Evaluation $\angf{F}$ of any $GL(2)$ prefoam $F$ belongs to the subring $R$ of $\wR=\pseries{E_1,E_2}.$
\end{prop}

\emph{Proof:} Evaluations of surfaces and theta-foams, with dots, in Section~\ref{sec:examples} depend only on the pre-foam structure, not on an embedding in $\R^3$. Skein relations described in Section~\ref{sec:skein_relations} extend, with suitable care, to pre-foams. In Figure~\ref{fig:sneck_cut_1} relation, a pre-foam on the LHS induces pre-foam structures on terms on the right, with orientations of facets in the RHS coming from those of the LHS. With this convention, Figure~\ref{fig:sneck_cut_1} relation holds for pre-foams, where in the pre-foam $F$ on the LHS one also remembers the cyclic order of the facets

Relations in Figures~\ref{fig:sneck_cut_4}, \ref{fig:disk_cancel} extend likewise. In Figure~\ref{fig:neck_cut_1} choice of orientations of all foams (respectively, pre-foams) is encoded in the orientation of the singular circle on the RHS (equivalently, of the cyclic order of the 3 facets at the circle).

Analogue of Proposition~\ref{prop:reverse_orient} for prefoams is that $\angf{F'}=(-1)^k\angf{F}$, where $F'$ is obtained from $F$ by reversing the cyclic order of facets at some $k$ singular circles of $F$.

Figure~\ref{fig:dot_reduction} relation obviously extends to prefoams. In the double facet neck-cutting relation in Figure~\ref{fig:neck_cutting_d} relation prefoam $F_1$ on the right induces an orientation on the prefoam $F$ on the left. With this convention, Figure~\ref{fig:neck_cutting_d} relation extends to pr-foams. Dot migration relations in Figures~\ref{fig:dot_migr_1} and \ref{fig:dot_migr_2} as well as the tube-cutting relation  in Figure~\ref{fig:digon_decomp_1} extend to prefoams.

Modification $m(\gamma,F)$ in Figure~\ref{fig:gamma_move_2} can be done to a prefoam $F$, assuming compatible orientations and cyclic orders along the two singular edges of $F$. Proposition~\ref{prop:gammas} will hold for prefoams as well, again assuming compatibility of the orientations and cyclic orders along the three singular edges shown in Figure~\ref{fig:double_gamma_5}.

Starting with a prefoam $F$, look at the thin surface $F_{12}$. It may have several connected components, some of which are connected in $F$ by double facets.
Applying the double neck-cutting relation in Figure~\ref{fig:neck_cutting_d}, using multiplicativity of $\angf{F}$ on the disjoint union of prefoams, and the evaluation of closed double surfaces (Examples 4, 5 in Section~\ref{sec:examples}), we can reduce the evaluation to the case when $F_{12}$ is connected and each double facet is a disk.
Applying the singular neck-cutting relation in Figure~\ref{fig:sneck_cut_1} along each singular circle of $F$, the evaluation reduces to that of a closed thin surface, possibly with dots, see Examples 1-3 in Section~\ref{sec:examples}. All coefficients in the  skein relations and in the evaluation of closed surfaces belong to the ring $R$, implying the proposition.
$\square$

The proposition implies the next result.

\begin{prop}\label{prop:eval_in_subring}
   Evaluation $\angf{F}$ of any closed foam $F$ coincides with the evaluation of the associated prefoam. In particular, it belongs to the subring $R$ of $\wR=\pseries{E_1,E_2}.$
\end{prop}

\emph{Proof:} Foam $F$ lives in $\R^3$, but to evaluate it using the formulas  (\ref{eq:def_eval_gl2}) and (\ref{eq:def_eval_gl3}) we can pass to the associated prefoam and evaluate it instead.
$\square$

Consequently, evaluations of all closed foams belong to the subring $R$ of $\wR$. It can then be chosen as the ground ring of the theory instead of $\wR$, in the $GL(2)$ case.


\subsection{$GL(2)$ webs, their state spaces, and direct sum decompositions}\label{subsec:direct_sum_dec}

$\quad$

We define $GL(2)$ closed webs $\Gamma$ as generic intersections of $GL(2)$ foams with planes $\R^2$ in $\R^3$. A $GL(2)$ web $\Gamma$ is a plane trivalent oriented graph with \emph{thin} and \emph{thick} (or \emph{double}) edges and vertices as in Figure~\ref{fig:web}.

\begin{figure}[h]
\begin{center}
\includegraphics[scale=0.40]{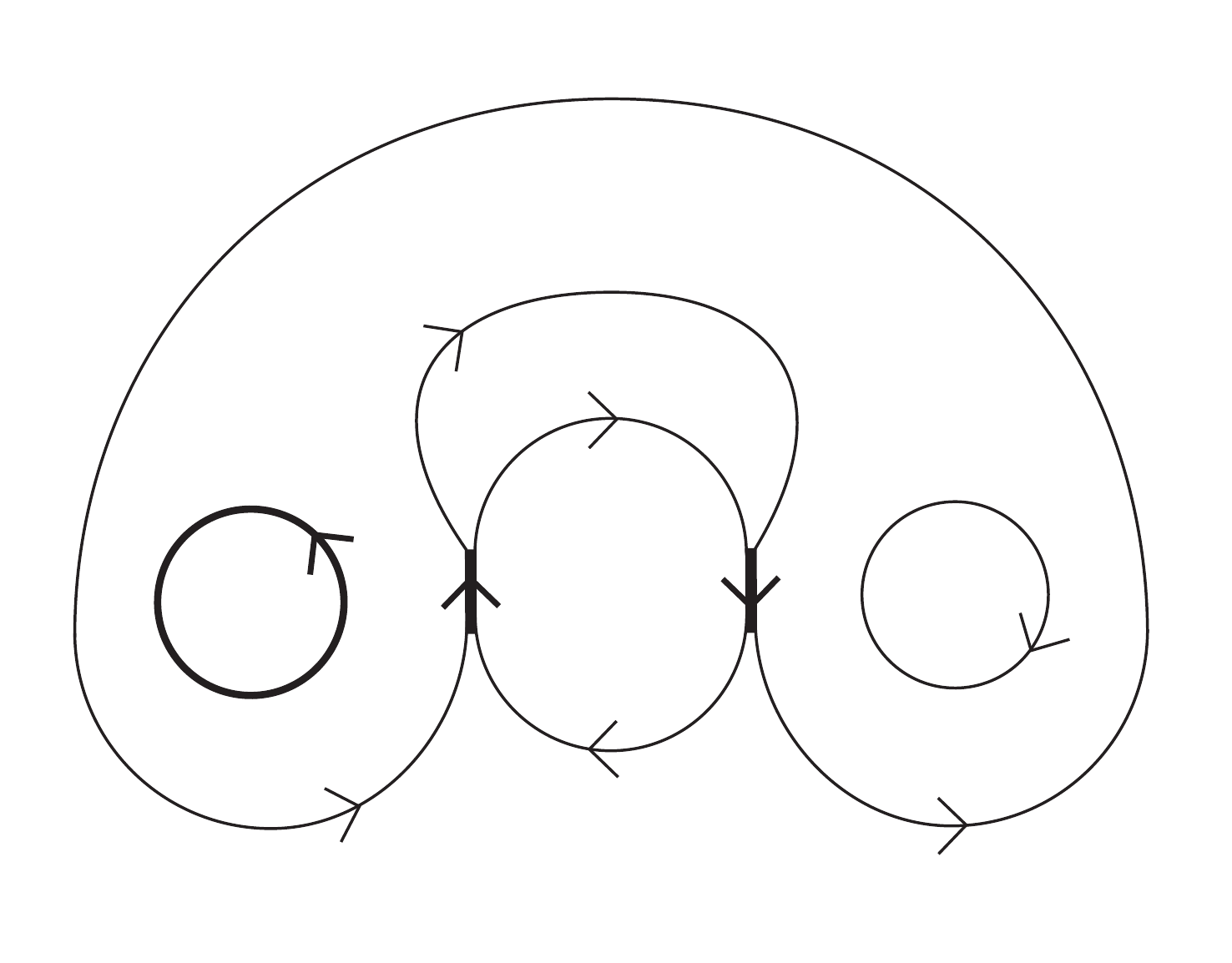}
\caption{\label{fig:web} A $GL(2)$ web $\Gamma$ with  two thick edges, four thin edges, one thin and one thick circle.}
\end{center}
\end{figure}

 Vertices of $GL(2)$ foams may be of two types. In one type, a pair of oriented thin edges flows into the vertex and a double edge flows out. In the other type, a double edge flows in and a pair of oriented thin edges flows out of the vertex. The web in Figure~\ref{fig:web} has two vertices of each type.

Single and double closed loops are allowed, as well as the empty web.
The union of thin edges of $\Gamma$ is called the \emph{thin one-manifold} of $\Gamma$, or the \emph{thin cycles} of $\Gamma$ and denoted $\Gamma_{(1)}$. For $\Gamma$ in Figure~\ref{fig:web}, the thin one-manifold $\Gamma_{(1)}$ has three connected components.

One defines $GL(2)$ foams with boundary a $GL(2)$ web $\Gamma$ in the usual way. We use Figure~\ref{fig:gl2_orientations} as the convention for the induced orientation of the web that's the boundary of a $GL(2)$ foam. Note that $GL(2)$ foams $F$ in
$\R^2\times [0,1]$ with the boundary $(-\Gamma_0)\sqcup \Gamma_1$, where $\Gamma_i = F
\cap (\R^2\times \{i\})$, $i=0,1$,  may be viewed as cobordisms between $\Gamma_0$ and $\Gamma_1$.

Define $\Foamt$ as the  category  where objects are $GL(2)$ webs $\Gamma$ and morphisms from $\Gamma_0$ to $\Gamma_1$ are
isotopy classes (rel boundary) of $GL(2)$ foams with the boundary $(-\Gamma_0)\sqcup \Gamma_1$. Composition is the concatenation of foams.

Define the degree of a foam $F$, not necessarily closed, as
\begin{equation}\label{eq:foam_degree}
\deg(F) = - \chi(F_{12}) + 2|d(F)|,
\end{equation}
where $d(F)$ is the number of dots of $F$.
Thin surface $F_{12}$ of $F$ is well-defined for foams with boundary. The boundary of $F_{12}$ is the union of thin circles on the boundary of $F$.

For closed foams $F$, $\deg(F)$ equals the degree of $\angf{F}$, viewed as a homogeneous element of either $\wR$ or its subring $R$.
Degree of a foam is additive under composition of foams.

We define the state space $\angf{\Gamma}$ of a $GL(2)$ web $\Gamma$ using the universal construction as in~\cite{BHMV,Kh2}.

First, let $\Free(\Gamma)$ be the free graded $R$-module with a basis $\{[F]\}_F$, over all
foams $F$ from the empty web to $\Gamma$. The degree of the generator $[F]$ is defined to be $\deg(F)$. Define a bilinear form on $\Free(\Gamma)$ by
\begin{equation} \label{eq:bilin_form}
([F],[G]) = \angf{w(G)F},
\end{equation}
where $w(G)$ is the reflection of $G$ in the horizontal plane together with the orientation reversal of all facets of $G$ to make $F$ and $w(G)$ composable along $F$. The foam $w(G)F$ is closed and can be evaluated to an element of $R$.
Given a closed foam $H$, reflecting it about a plane into a foam $H'$ may add sign to the evaluation, $\angf{H'} = (-1)^k \angf{H}$, where $k$ is the number of singular circles of $H$.
To get rid of the sign, reverse orientation of all facets of $H'$ to get a foam $w(H)$ with $\angf{w(H)} = \angf{H}$.
A similar argument works for non-closed foams. Consequently, the bilinear form (\ref{eq:bilin_form}) is symmetric.

Define the state space $ \angf{\Gamma} $ as the quotient of $\Free(\Gamma)$ by the kernel of the bilinear form $(,)$.
The state space $\angf{\Gamma}$ is a graded $R$-module, via the degree formula (\ref{eq:foam_degree}). As usual in the universal construction, a foam $F$ with boundary $(-\Gamma_0)\sqcup \Gamma_1$ induces a homogeneous $R$-module map
\[ \angf{F} \ : \ \angf{\Gamma_0} \lra \angf{\Gamma_1}
\]
of degree $\deg(F)$ taking an element $\angf{G}\in \angf{\Gamma_0}$ associated to a foam $G$ with boundary $\Gamma_0$ to the element $\angf{FG}$ associated to the foam
$FG$ with boundary $\Gamma_1$. These maps assemble into a functor from the category of $GL(2)$ foams to the category of graded $R$-modules and homogeneous $R$-module maps. The results below imply that the functor is monoidal.

The state space of the empty web is naturally isomorphic to the free rank one module over $R$ with a generator in degree zero, $\angf{\emptyset}\cong R$.

Let $\Gamma'$ denote the web $\Gamma$ with an innermost thin circle (with one of the two orientations) added in a region of $\Gamma$. Thus, $\Gamma'$ depends on the choice of a region of $\Gamma$ and the  orientation of the circle.

\begin{prop} There are natural isomorphisms of graded $R$-modules
\[ \angf{\Gamma'} \ \cong \ \angf{\Gamma}\{1\} \oplus \angf{\Gamma}\{-1\},
\]
for $\Gamma,\Gamma'$ as above and $\{m\}$ the grading shift up by $m$.
\end{prop}

\begin{figure}[h]
\begin{center}
\includegraphics[scale=0.50]{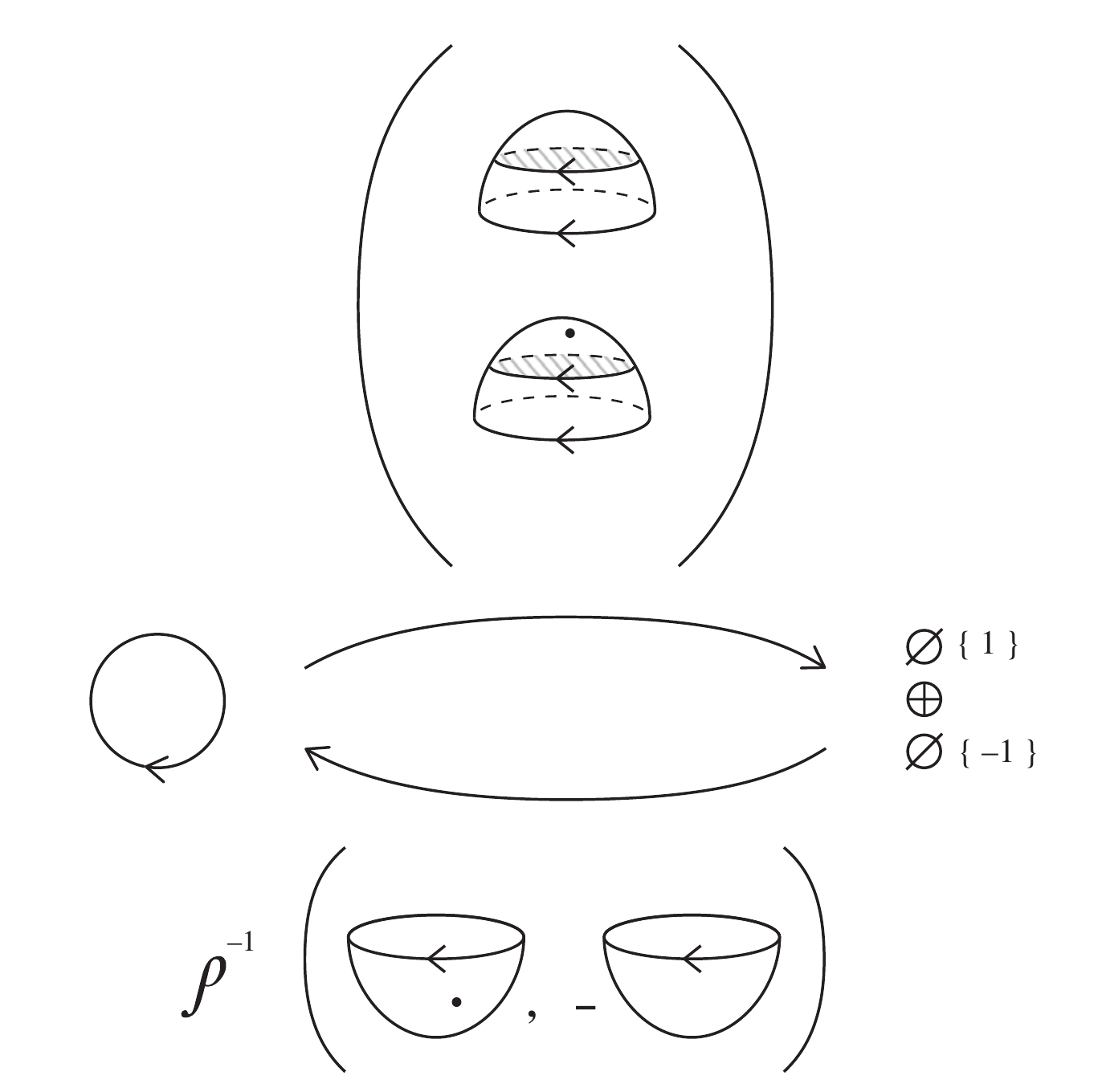}
\caption{\label{fig:dir_sum_3}  Direct sum decomposition for an innermost thin circle with the clockwise orientation. For the opposite orientation of the circle, reverse the orientation of the singular circle as well and add an overall minus sign to one of the two maps. }
\end{center}
\end{figure}

\emph{Proof:} Foam cobordisms that deliver this direct sum decomposition are shown in Figure~\ref{fig:dir_sum_3}. The composition of the maps in either order is the identity, as follows from $\Theta$-foam evaluations in section~\ref{sec:examples} and neck-cutting relation in Figure~\ref{fig:neck_cut_1}. $\square$

\begin{prop} The saddle cobordism in a thick facet induces a grading-preserving isomorphism between the state spaces of its two boundary webs, see Figure~\ref{fig:dir_sum_1}. The inverse isomorphism is given by the adjoint saddle cobordism scaled by $\rho^{-1}$.
\end{prop}

\begin{figure}[h]
\begin{center}
\includegraphics[scale=0.50]{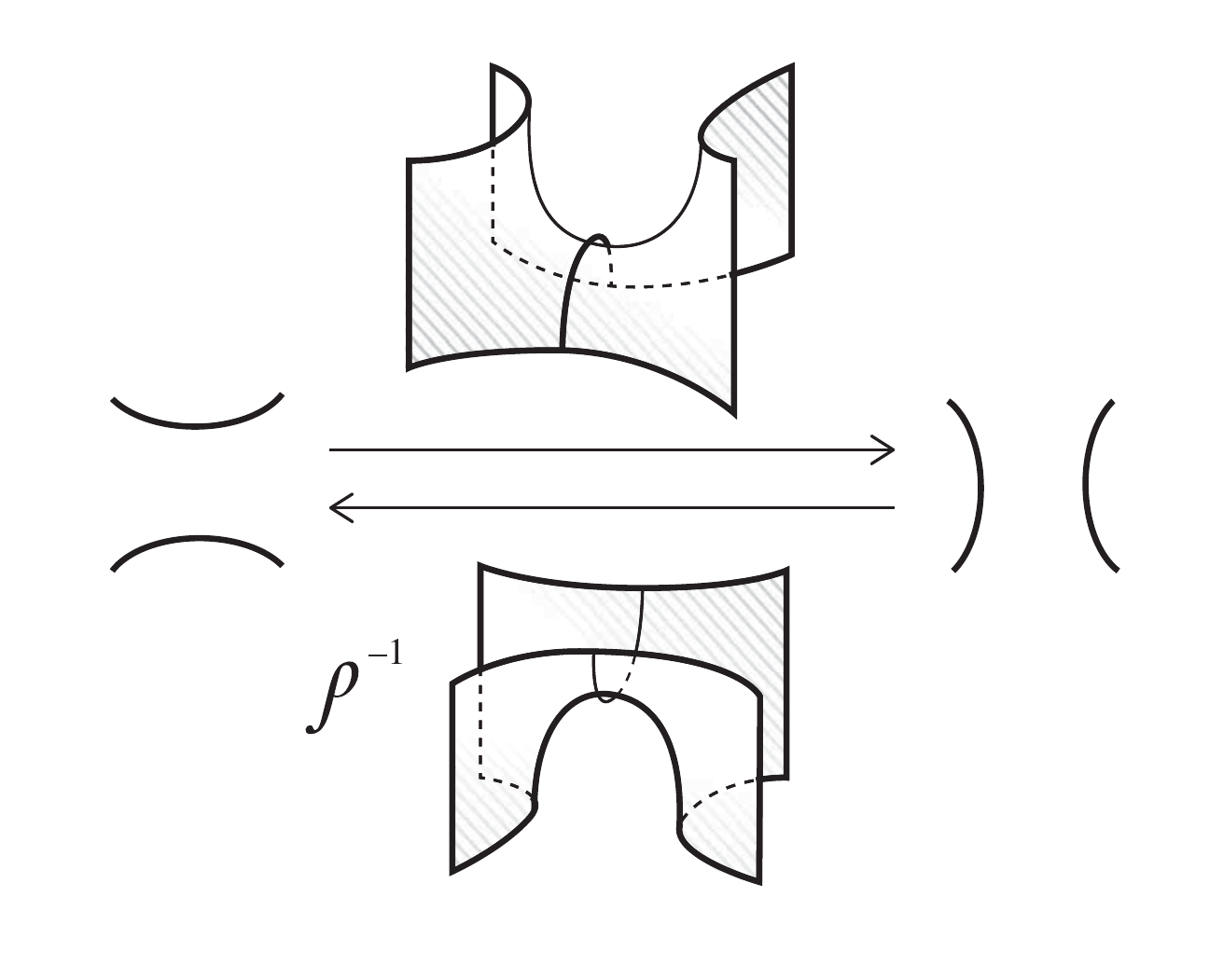}
\caption{\label{fig:dir_sum_1} Saddle isomorphism on a double facet. Double edges on the left and right must carry compatible orientation (that is, extendable to the orientation of the surface).}
\end{center}
\end{figure}

\emph{Proof:} This follows from the thick neck-cutting relation in Figure~\ref{fig:neck_cutting_d}. $\square$

\begin{prop} Let $\Gamma$ be a web and $\Gamma'$ be $\Gamma$ with added innermost thick circle. There is a canonical degree zero isomorphism of state spaces
\[ \angf{\Gamma'} \ \cong \ \angf{\Gamma}
\]
given by the cobordisms in  Figure~\ref{fig:dir_sum_2}.
\end{prop}

\begin{figure}[h]
\begin{center}
\includegraphics[scale=0.50]{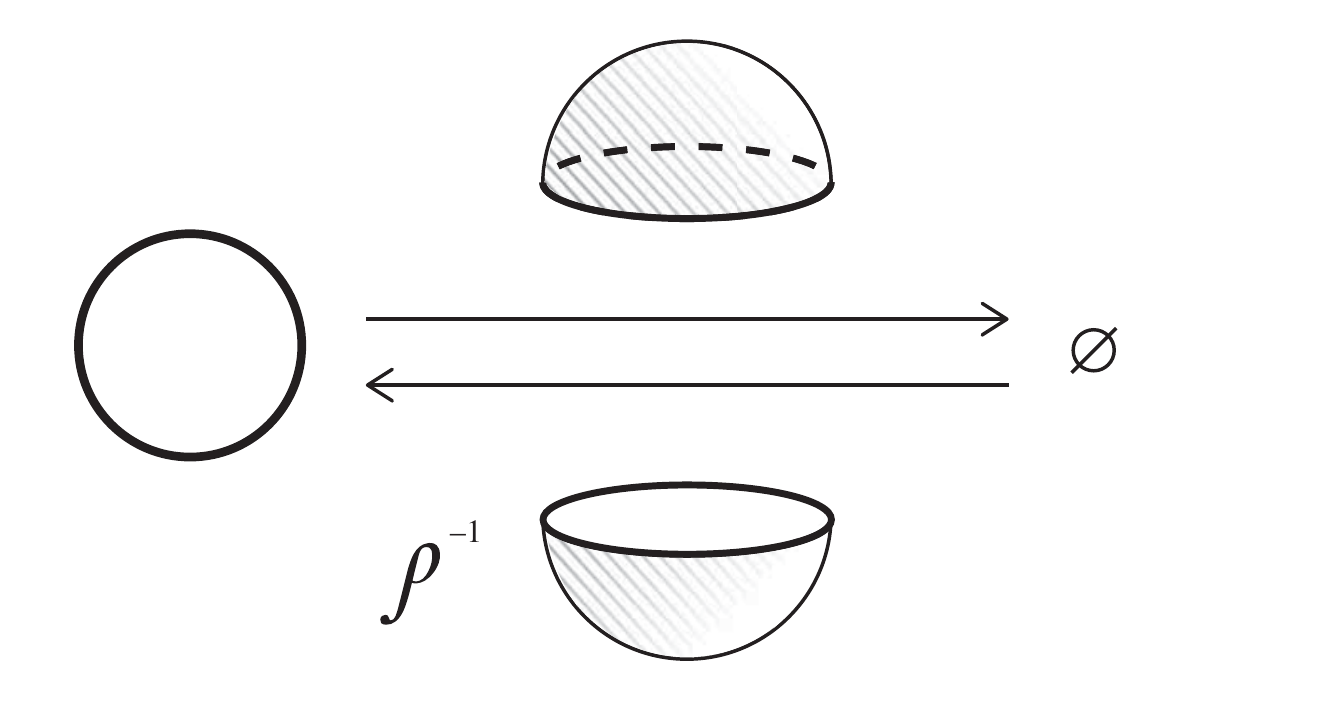}
\caption{\label{fig:dir_sum_2} An isomorphism between a diagram with an innermost double circle and the diagram without it, via double cup and cap cobordisms.}
\end{center}
\end{figure}

\begin{prop} \label{prop_digon_d} Let web $\Gamma$ have a thin edge and denote by $\Gamma'$ the web $\Gamma$ with an attached double edge along the thin edge. The state spaces of $\Gamma$ and $\Gamma'$ are naturally isomorphic as graded $R$-modules via the  maps given in
Figure~\ref{fig:dir_sum_4}.
\end{prop}

\begin{figure}[h]
\begin{center}
\includegraphics[scale=0.45]{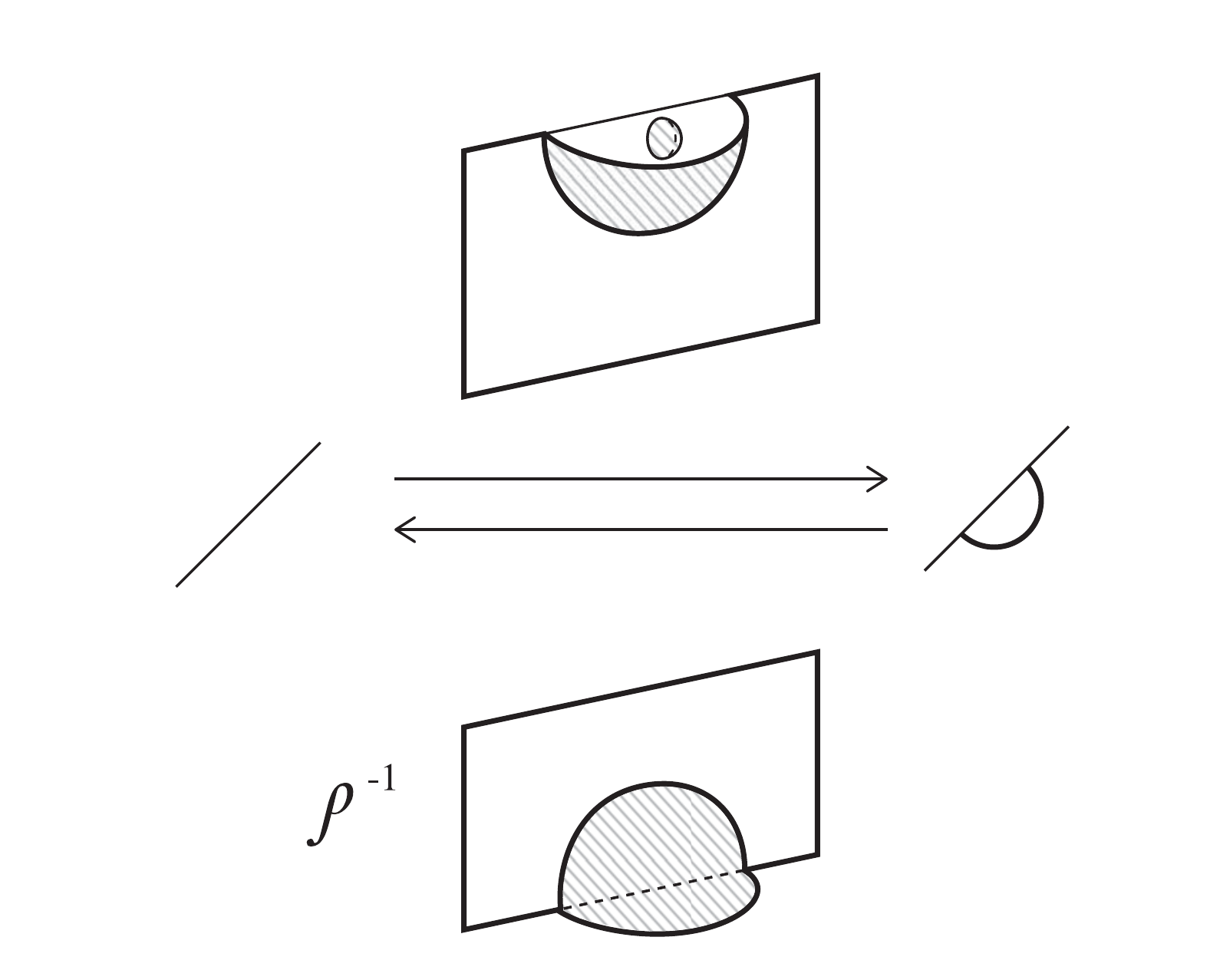}
\caption{\label{fig:dir_sum_4} Mutually-inverse isomorphisms between a thin edge and a thin edge with an attached double edge.}
\end{center}
\end{figure}

\emph{Proof:} This follows from relations in Figures~\ref{fig:double_gamma_3} and \ref{fig:disk_cancel}. $\square$

\begin{theorem}
   $\angf{\Gamma}$ is a free graded $R$-module of graded rank $[2]^{m}$, where $m$ is the number of components (circles) of the thin one-manifold $\Gamma_{(1)}$ and $[2]=q+q^{-1}$.
\end{theorem}

\emph{Proof:} This can be proved by induction on $m$. An innermost thin or double circle of $\Gamma$, see Figure~\ref{fig:innermost_circles}, can be removed using isomorphisms in Figures~\ref{fig:dir_sum_3} and \ref{fig:dir_sum_2}, respectively.

\begin{figure}[h]
\begin{center}
\includegraphics[scale=0.35]{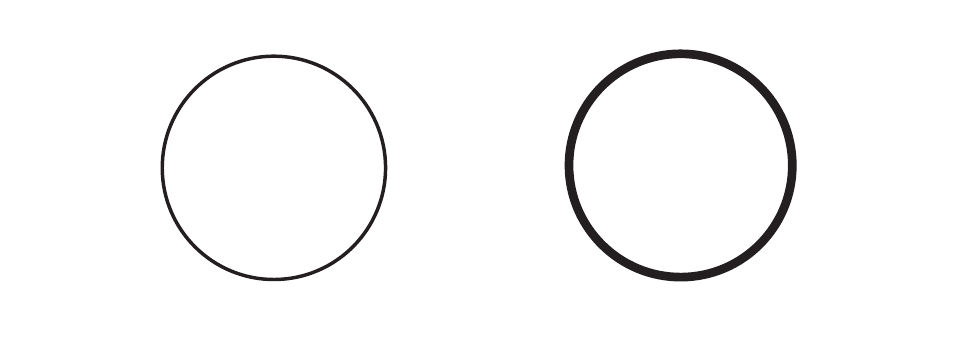}
\caption{\label{fig:innermost_circles} Thin and double innermost circles.}
\end{center}
\end{figure}
Now look at $\Gamma_{(1)}$ and choose an innermost circle $\alpha$ in it. We distinguish between innermost circles of $\Gamma$ and those of $\Gamma_{(1)}$. The latter correspond to thin circles in $\Gamma$ which may contain vertices and thus have attached double edges. $\alpha$ bounds a disk $D^2$ in $\R^2$. Double edges emanating out of $\alpha$ split into those inside and outside of $D^2$. Repeatedly applying the double saddle isomorphism in Figure~\ref{fig:dir_sum_1}, we can reduce to the case when each of these double edges has both endpoints on $\alpha$. Going along $\alpha$ one encounters $2n$ vertices (an even number due to orientation reversal along $\alpha$ at each vertex). If at two consecutive vertices double edges both point in or out of $D^2$, one can apply an  isomorphism in Figure~\ref{fig:dir_sum_1} followed by an isomorphism in Figure~\ref{fig:dir_sum_4} to reduce from $2n$ to $2n-2$ vertices along $\alpha$. A configuration where such a pair of vertices does not  exist is impossible for $n>0$, for then the $n$ ends of  double edges pointing into $D^2$ from  $\alpha$ would all have the same orientations and there would be no room for the other $n$ ends of these edges to land. This concludes the inductive argument. $\square$

\begin{cor}
Associating the state space $\angf{\Gamma}$ to a $GL(2)$ web $\Gamma$ and the map $\angf{F}$ of
state spaces to a foam $F$ with boundary
is a monoidal functor from the category of
$GL(2)$ foams to the category of free graded $R$-modules of  finite rank.
\end{cor}


\section{Reidemeister moves invariance and link homology}
\label{sec:reidemeister}

With the state spaces $\angf{\Gamma}$ of $GL(2)$ webs $\Gamma$ defined, we can associate homology groups to a generic projection $D$ of an oriented link $L\subset \R^3$, as follows.
Let $D$ has $n$ crossings. We resolve each crossing into two resolutions, $0$- and $1$-resolutions, as in Figure~\ref{fig:two_resolutions}.

\begin{figure}[h]
\begin{center}
\includegraphics[scale=0.5]{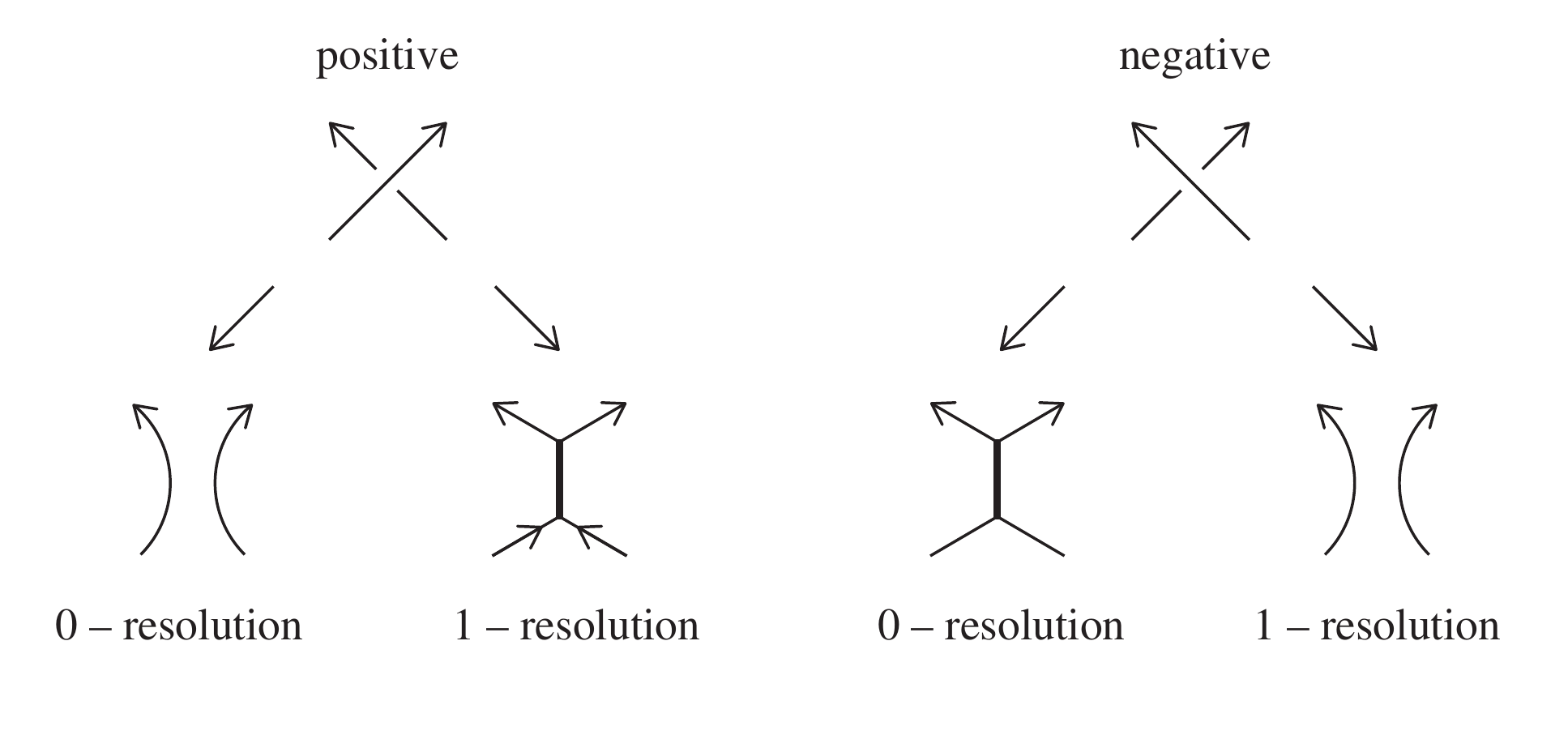}
\caption{\label{fig:two_resolutions} Resolutions of a positive and a negative crossings.}
\end{center}
\end{figure}
One of the resolutions consists of two disjoint thin edges, the other contains a double edge and four adjoint thin edges. All the edges are oriented. Choose a total order on crossings of $D$.
Doing this procedure over all crossings results in $2^n$ resolutions  of $D$ into $GL(2)$ webs $D(\ovmu)$, for $\ovmu=(\mu_1,\dots,\mu_n)$, with $\mu_i\in\{0,1\}$. In a web $D(\ovmu)$ the $i$-th crossing is resolved according to $\mu_i$.

To a crossing  now associate a complex of two webs with boundaries and the differential induced by the "singular saddle" cobordism between them, see Figure~\ref{fig:complexes_for_crossings} which sets us the terms in the complex, and Figure~\ref{fig:differential} which depics "singular saddle" foams inducing the differential.
\begin{figure}[h]
\begin{center}
\includegraphics[scale=0.60]{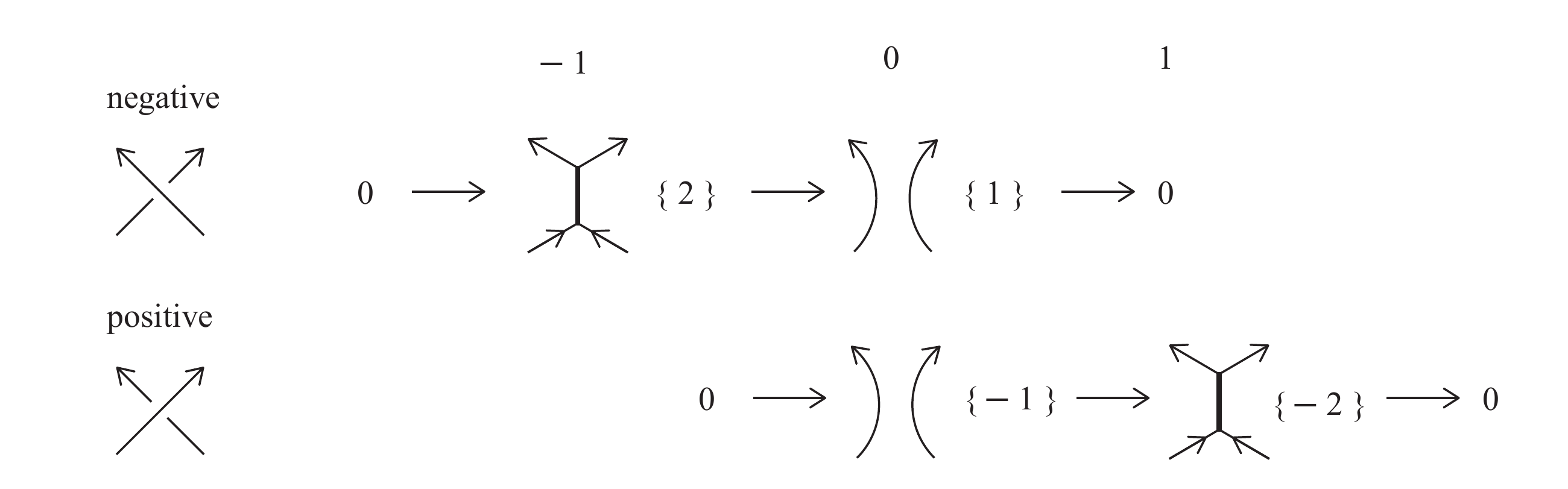}
\caption{\label{fig:complexes_for_crossings} Complexes associated to positive and negative crossings. Numbers at the top show homological gradings of the terms. Resolution into two edges is always in homological degree $0$.}
\end{center}
\end{figure}
\begin{figure}[h]
\begin{center}
\includegraphics[scale=0.40]{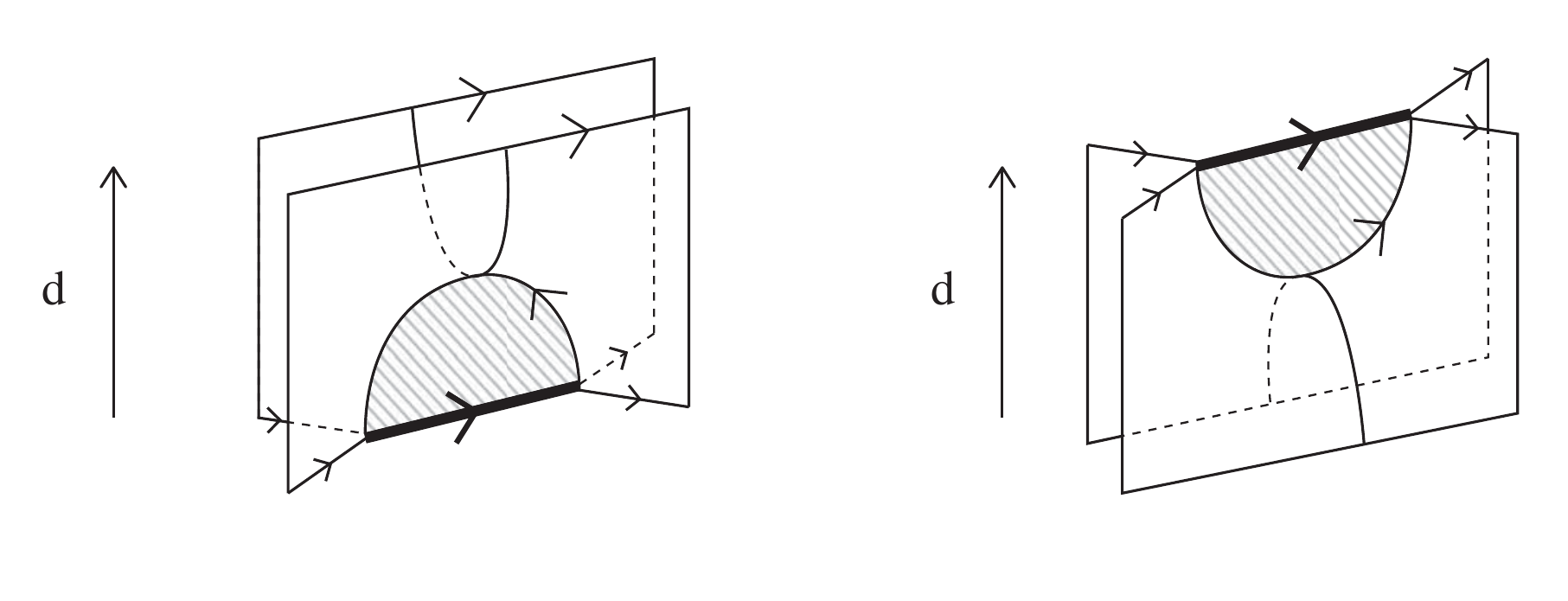}
\caption{\label{fig:differential} Foams that induce the differential in the complexes for a positive and negative crossings. Upward-pointing arrows next to the foams indicate the 'direction' of the differential.}
\end{center}
\end{figure}
These complexes make sense whenever the two webs are closed on the outside into two closed $GL(2)$ webs.
Grading shifts are inserted to make the map induced by the  "singular saddle" cobordism grading-preserving (and, later, to have full invariance under the Reidemeister I  move, rather than  an invariance up to an overall grading shift).

In this way, one can form a commutative $n$-dimensional cube which has the graded $R$-module $\angf{D(\ovmu)}$ in its vertex labelled by the sequence $\ovmu$ and maps induced by "singular  saddle" foams associated to oriented edges of the cube. The  maps commute for every square of the cube.

This setup with "singular saddle" cobordisms goes back to Blanchet~\cite{B}, and is also visible in the earlier papers of Clark-Morrison-Walker~\cite{CMW} and Caprau~\cite{Ca1,Ca2}, where the double facet is not there, but its boundary, a singular edge along the foam, together with a choice of normal direction, is present.

The commuting  cube of graded $R$-modules $\angf{\Gamma(\ovmu)}$ and grading-preserving homomorphisms between them collapses, in the standard way upon adding minus signs, to a complex of graded $R$-modules with a degree-preserving differential. This complex starts in the homological degree -- minus  the number of negative crossings of $D$ and ends in the homological degree which is the number of positive crossings of $D$.

Denote this complex by $F(D)$.

\begin{theorem} \label{thm_homotopy_invariance} For two diagrams $D_1$ and $D_2$ of an oriented link $L$, complexes $F(D_1)$ and $F(D_2)$ are chain homotopy equivalent as complexes of graded $R$-modules.
\end{theorem}

\emph{Proof:} Consider the Reidemeister move R1, undoing a positive curl in Figure~\ref{fig:R1_move}.

\begin{figure}[h]
\begin{center}
\includegraphics[scale=0.60]{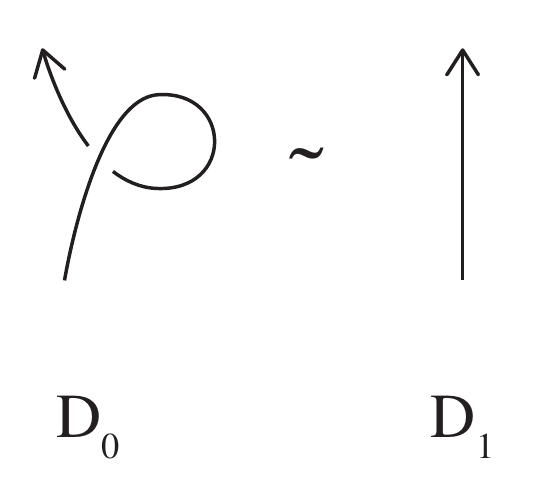}
\caption{\label{fig:R1_move} Reidemeister move R1, for a positive twist.}
\end{center}
\end{figure}

\begin{figure}[h]
\begin{center}
\includegraphics[scale=0.55]{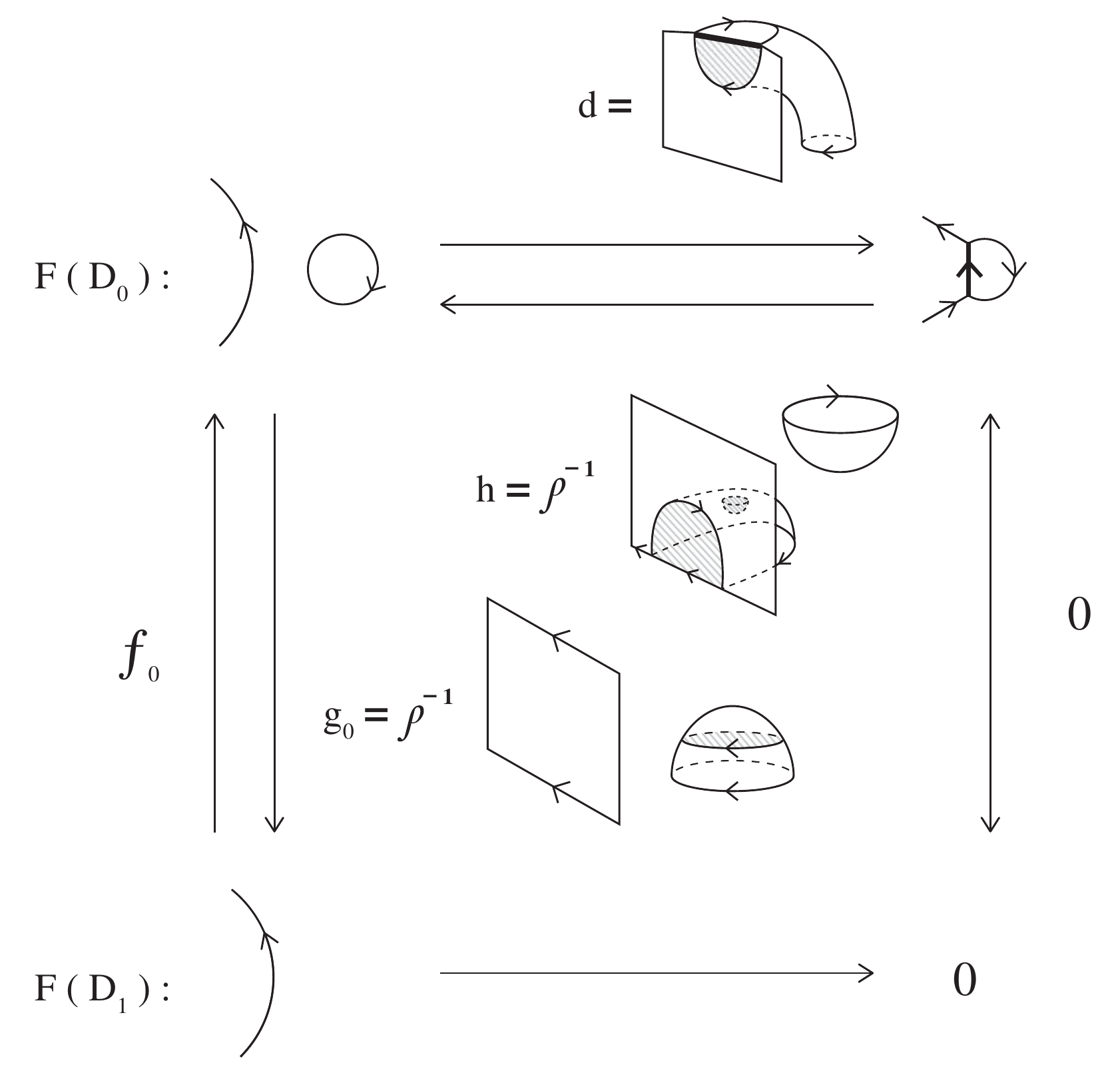}
\caption{\label{fig:solving_R1} Top row, together with the right-pointing arrow $d$, encodes the complex $F(D_0)$. Top left-pointing arrow $h$ is a self-homotopy of $F(D_0)$. Down and up arrows $h_0$ and $g_0$ are maps of
complexes $F(D_0)$ and $F(D_1)$. Map $f_0$ is given in the next figure.}
\end{center}
\end{figure}

\begin{figure}[h]
\begin{center}
\includegraphics[scale=0.50]{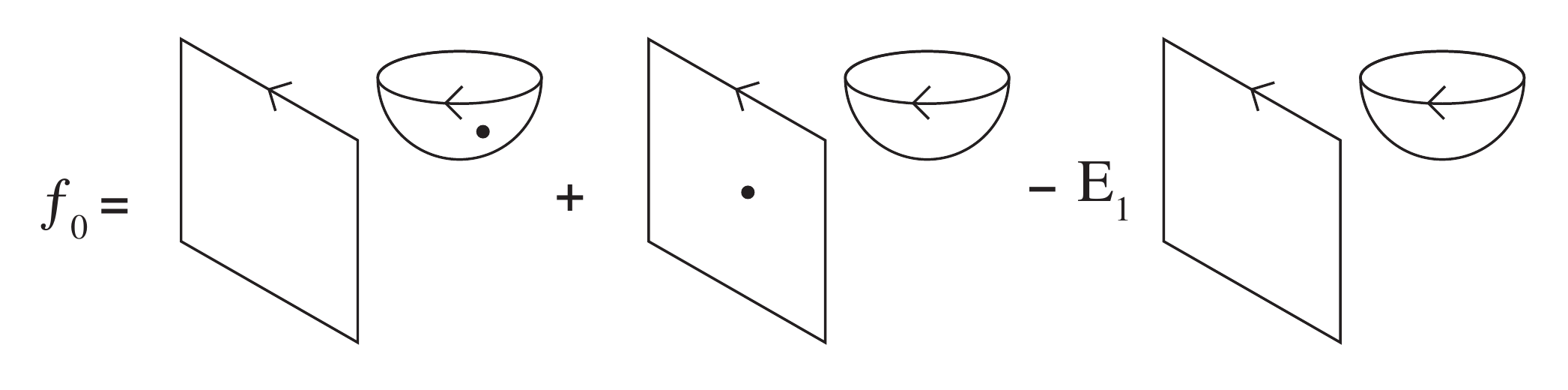}
\caption{\label{fig:h_null} Map $f_0: F(D_1)\lra F(D_0)$ of complexes. }
\end{center}
\end{figure}

\begin{prop} The following relations hold on maps $f_0,g_0,h$ and $d$ in Figures~\ref{fig:solving_R1}, \ref{fig:h_null}:
\begin{eqnarray}
    d  h & = & \mathrm{id}, \\
    d f_0 & = & 0 , \\
    g_0  f_0 & = & \mathrm{id}_{F(D_1)}, \\
    \mathrm{id} & = & f_0 g_0 + h d.
\end{eqnarray}
The map $\mathrm{id}$ in the  first equation is the identity of the complex $F(D_0(1))$, associated to the diagram in the top right corner of Figure~\ref{fig:solving_R1}, while $\mathrm{id}$ in the last equation is the identity of  the complex $F(D_0(0))$ associated to the diagram the top left corner of the figure.
\end{prop}

\emph{Proof} is a direct computation using skein relations derived in Section~\ref{sec:skein_relations}.
This proof is very  similar to the proof of the invariance under the Reidemeister move in~\cite{MSV}, that does it in the non-equivariant $GL(N)$ case, in particular see Figure 8 there.
$\square$

\begin{cor} Complexes $F(D_0)$ and $F(D_1)$, for diagrams in Figure~\ref{fig:R1_move}, are chain homotopy equivalent as complexes of graded $R$-modules.
\end{cor}

\begin{prop}
  For each pair of the diagrams $D_0$, $D_1$ in Figure~\ref{fig:R23move}, which shows Reidemeister moves R2 and R3, complexes $F(D_0)$ and $F(D_1)$ are chain homotopy equivalent as complexes of graded $R$-modules.
\end{prop}

\begin{figure}[h]
\begin{center}
\includegraphics[scale=0.60]{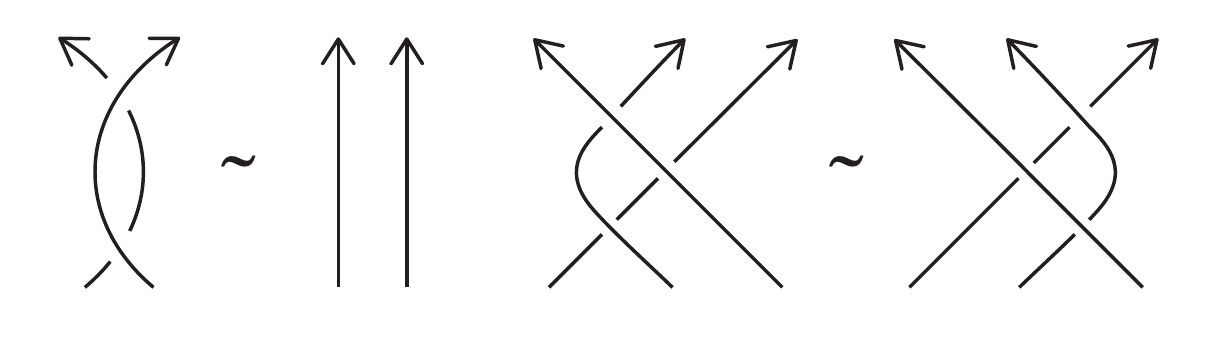}
\caption{\label{fig:R23move} Reidemeister moves R2 and R3.}
\end{center}
\end{figure}

\emph{Proof:} For the Reidemeister R2 move, relation (\ref{fig:digon_decomp_3}) used in the direct sum of decomposition of a web $\Gamma_1$ with a digon facet into the sum of two copies of the simpler web $\Gamma_0$ is no different from the corresponding decomposition in the usual $SL(N)$ graphical calculus, for an arbitrary $N$ (see~\cite[Proposition 8]{Kh2} for the analogous decomposition in the non-equivariant $SL(3)$ case). As one of the relations for this decomposition, the relation of removing a bubble on a double facet with at most one dot on one of the two thin facets is identical with the corresponding relation in the usual $SL(N)$ foam calculus, whether for the standard calculus or the equivariant  one. Bubble removal relation follows from the combination of theta foam evaluation in  Example 6 in Section~\ref{sec:examples} for $n_1,n_2\le 1$ and Proposition~\ref{prop_double_nc}. For essentially the same relations  in the $SL(3)$ case see, for instance, the top two relations in~\cite[Figure 18]{Kh2}.

For this reason, the usual proof of  the Reidemester R2 relation, when both strands are oriented in the same direction, as in Figure~\ref{fig:R23move} left, repeats without any changes in our case, see for instance~\cite[Section 5.2]{Kh2}, \cite[Theorem 7.1]{MSV}, and many other sources.

\begin{figure}[h]
\begin{center}
\includegraphics[scale=0.60]{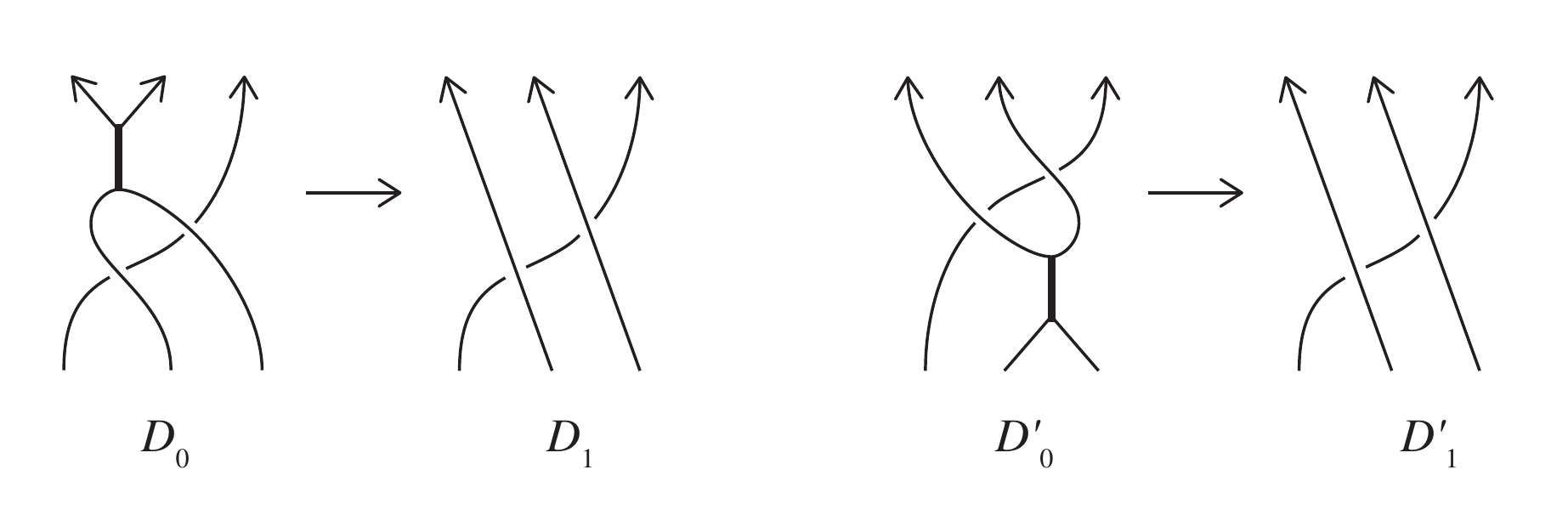}
\caption{\label{fig:R23_partial_res} Two partial resolutions of each of $D$ and $D'$. Note that $D_1$ and $D_1'$ are identical diagrams.}
\end{center}
\end{figure}

Consider the Reidemeister R3 move in Figure~\ref{fig:R23move}. Denote by $D$ and $D'$ the diagrams on the left and right of this move.

We start by resolving a single crossing in each of $D$ and $D'$, see Figure~\ref{fig:R23_partial_res}. Complexes $C(D)$ and $C(D')$ are isomorphic to cones of maps $C(D_0)\lra C(D_1)$ and $C(D'_0)\lra C(D'_1)$ built out of foams between complete resolutions of these diagrams.

Tangle diagrams $D_1$ and $D_1'$ are canonically isomorphic, and their resolutions result in the total complex of the square shown in Figure~\ref{fig:R23_d1_res} with the differential coming from the four foams associated to the arrows of the diagram, with each foam a standard singular saddle in the appropriate position.

\begin{figure}[h]
\begin{center}
\includegraphics[scale=0.60]{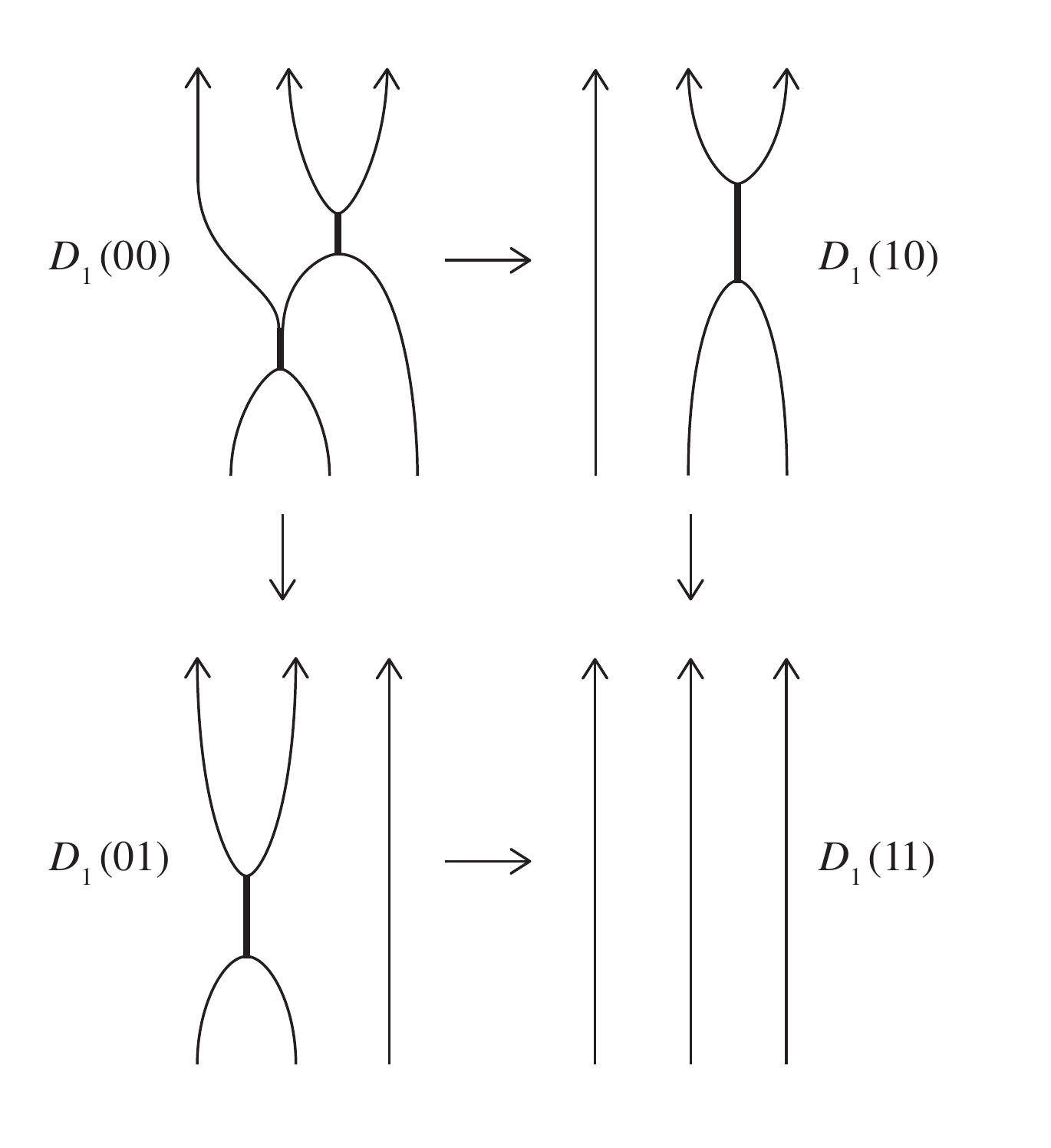}
\caption{\label{fig:R23_d1_res} Resolution of the diagram $D_1 \cong D_1'$.}
\end{center}
\end{figure}

Consider now the diagram $D_0$ and its resolution in Figure~\ref{fig_R23_d0_res}. Maps $\psi_k$, $k=1,\dots, 4$ are homomorphisms between state spaces of web induced by appropriate foams (singular saddle foams). Summing over all possible resolutions of crossings of $D_0$ not shown in the diagram gives homomorphisms, also denoted $\psi_1, \dots, \psi_4,$ of corresponding complexes.

\begin{figure}[h]
\begin{center}
\includegraphics[scale=0.60]{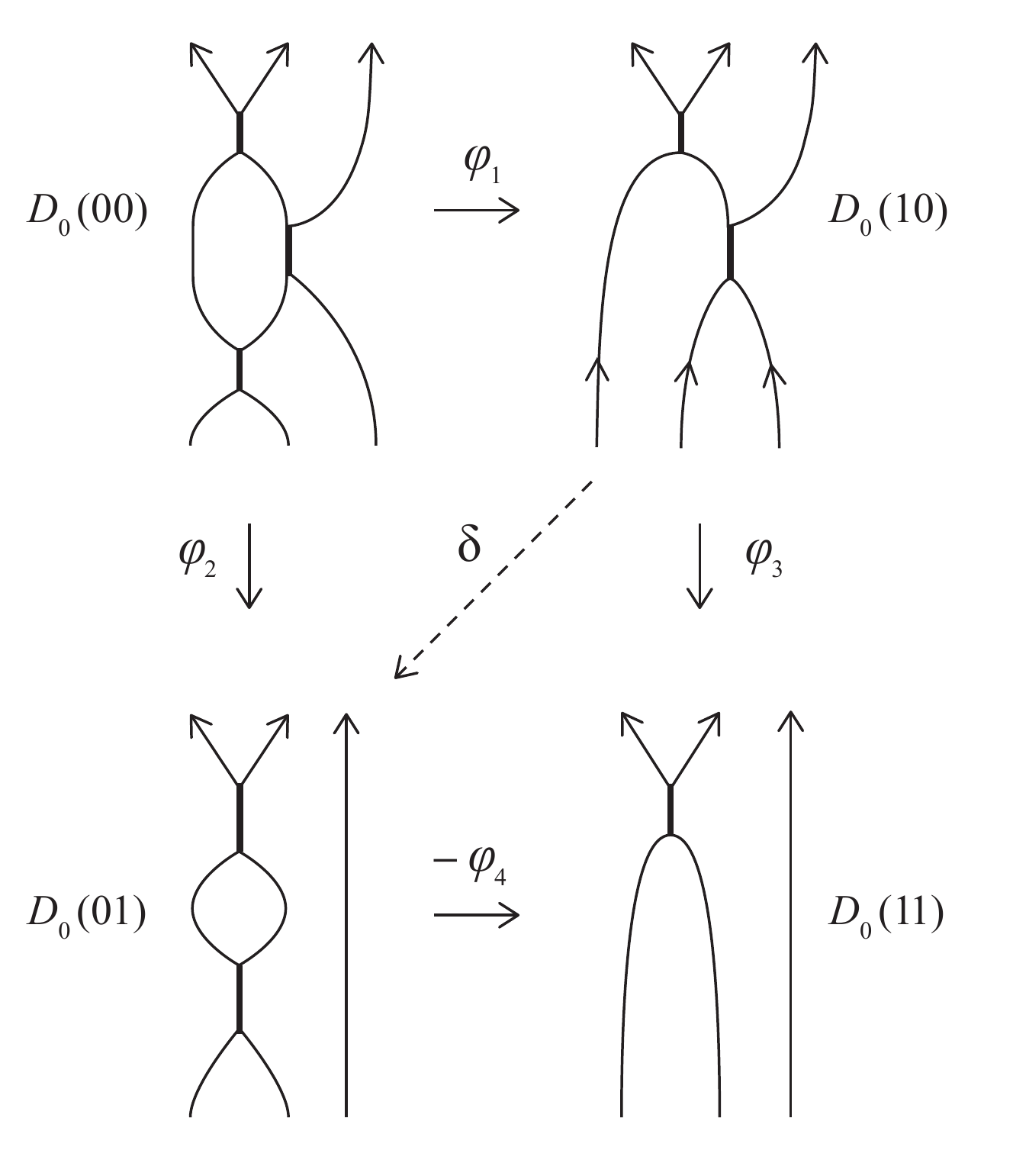}
\caption{\label{fig_R23_d0_res} Resolution of the diagram $D_0$.}
\end{center}
\end{figure}

The four terms $C(D_0(k\ell))$, $k,\ell\in\{0,1\}$, will also map to the corresponding four terms $C(D_1(k\ell))$ in $C(D_1)$ in Figure~\ref{fig:R23_d1_res} to constitute a 3-dimensional cube diagram (not shown).

The complex  $C(D_0(00))$ of the diagram in the upper left of Figure~\ref{fig_R23_d0_res} is isomorphic (and not just homotopy equivalent) to the complex $C(D_2)$ of the diagram $D_2$ shown in
Figure~\ref{fig:R23_simplify} left.

\begin{figure}[h]
\begin{center}
\includegraphics[scale=0.60]{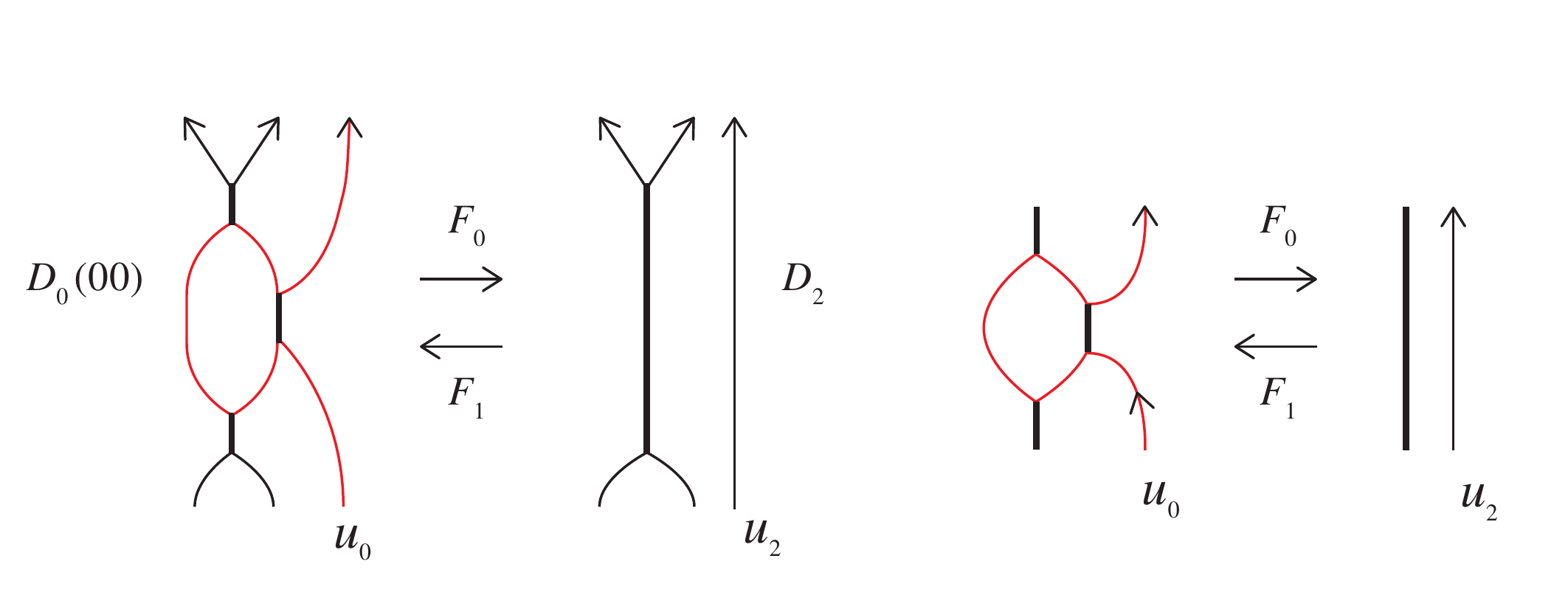}
\caption{\label{fig:R23_simplify} Diagrams $D_0(00)$ and $D_2$ have isomorphic state spaces for any resolution of these diagrams. Complexes $C(D_0(00))$ and $C(D_2)$, with the differentials  induced by various foams between their resolutions, are canonically isomorphic,
$C(D_0(00))\cong C(D_2)$ in the abelian category of complexes (before factoring by homotopies).}
\end{center}
\end{figure}
Foam $F_0$ going from $D_0(00)$ to $D_2$ 'straightens out' the long thin arc $u_0$ of $D_0(00)$ by canceling in pairs the four vertices on this arc where double edges meet $u_0$. Arc $u_0$ becomes the rightmost arc $u_2$ of $D_2$. Seam edges that cancel the four vertices in pairs are shown in Figure~\ref{fig_R23_d0_flat} as two arcs in the upper half of the diagram. The upper half shows the thin facet of $F_0$ where singular vertices along $u_0$ are cancelled in pairs. These cancellations are done via singular arcs, shown in Figure~\ref{fig_R23_d0_flat} top, along which double facets are attached to the thin facet.

\begin{figure}[h]
\begin{center}
\includegraphics[scale=0.40]{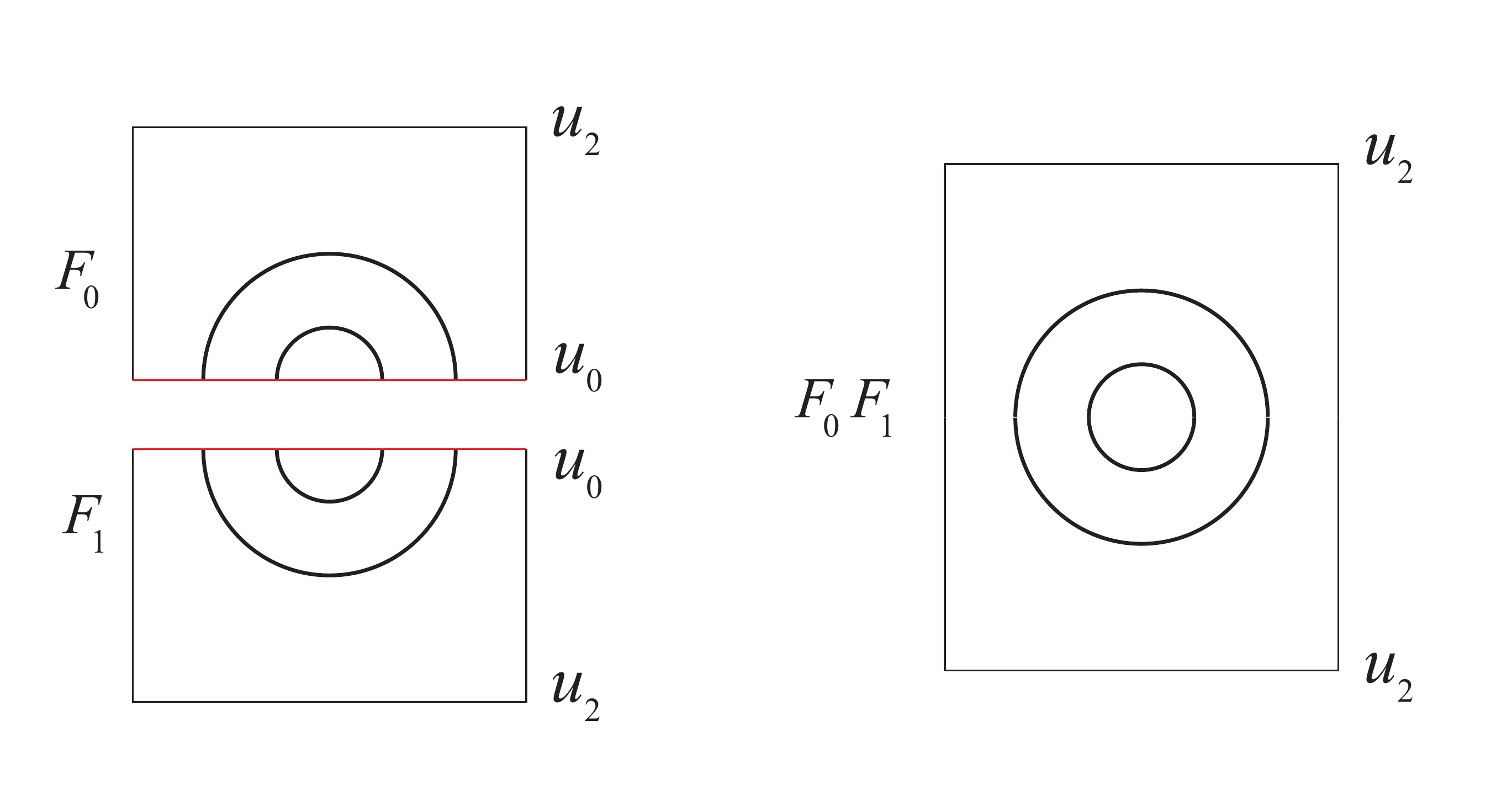}
\caption{\label{fig_R23_d0_flat}  Flattened thin facets of $F_0$ and $F_1$ containing arcs $u_0,u_2$. Composition $F_0F_1$ contains thin surface $S$ (shown on the right) given by gluing the two thin surfaces along the common arc $u_0$. This surface has two singular circles where double facets attach.}
\end{center}
\end{figure}
Foam $F_1$ goes back from $D_2$ to $D_0(00)$ and is given by reflecting $F_0$ in the horizontal plane.
The thin facet of $F_1$ is shown as the lower half of Figure~\ref{fig_R23_d0_flat}. Semicircles depict singular edges along the thin facet.

Denote the maps $F_0,F_1$ induce on state spaces and on complexes built out of the state spaces of all resolutions of $D_0(00)$ and $D_2$ by
\begin{equation}
    \tau_0: \ C(D_0(00)) \lra C(D_2), \ \ \tau_1: C(D_2)\lra C(D_0(00)).
\end{equation}
We know that both $\tau_0$ and $\tau_1$ are isomorphisms of the state spaces and corresponding complexes, since annihilating a digon facet with a thick edge is an isomorphism, see Proposition~\ref{prop_digon_d}.

More precisely, $\tau_0\tau_1 = - \rho^{-1}\Id$. Indeed, the  composition $\tau_0\tau_1$ is an endomorphism of the state space $\angf{\Gamma}$ for each web resolution $\Gamma$ of $D_2$ and the induced endomorphism of the complex $C(D_2)$. The map
$\tau_0\tau_1: \angf{\Gamma}\lra \angf{\Gamma}$ transforms arc $u_2$ of the diagram $D_2$ to the arc $u_0$ of $D_0(00)$ and back, via the composition of foams $F_0F_1$.

Consider the thin surface $S$ bounded by $u_2$ at the top and bottom of the cobordism $F_0F_1$. It can be visualized by gluing the two thin surfaces for $F_0$ and  $F_1$ shown in Figure~\ref{fig_R23_d0_flat} along the common arc $u_0$, shown in red. Surface $S$
contains two nested singular circles, where double facets of $F_0F_1$ meet $S$. Double facets at these two circles attach to $S$ from opposite sides, as one can glean from Figure~\ref{fig:R23_simplify}. This corresponds to having two double edges attached to arc $u_0$ on one side and the other double edge attached to $u_0$ on the other side of the plane, at both endpoints, see the leftmost diagram in Figure~\ref{fig:R23_simplify}.

Apply Proposition~\ref{prop_double_nc} at each of these attached double facets to simplify the non-trivial part of the foam $F_0F_1$ to the surface $S$ with two double disks attached to it from the  opposite sides along the two singular circles, with an  additional factor $\rho^{-2}$. We then apply Proposition~\ref{prop_disk_flip} to flip one of the disks to the opposite side, gaining a minus sign, and then use Proposition~\ref{prop_double_nc} to reduce to the identity foam times $-\rho^{-1}$.

Consequently, maps $\tau_0$ and $-\rho\tau_1$ are mutually-inverse isomorphisms.

Note that diagrams $D_0(11)$ and $D_2$ are isotopic and their complexes are canonically isomorphic. Complex $C(D_0(01))$ decomposes into direct sum of two copies of $D_0(11)$ in the usual way. The composition $\psi_2\tau_1: C(D_2)\lra C(D_0(01))$ is a split inclusion into one of these copies.Since $\tau_1$ is an isomorphism, this composition allows to split off contractible  summand
\[ 0 \lra C(D_0(00)) \stackrel{\cong}{\lra}
\mathrm{im}(\psi_2) \lra 0
\]
from the total complex of $D_0$, also see Figure~\ref{fig_R23_d0_res}. The map $\psi_4$ induces an isomorphism from the complementary direct summand of $C(D_0(01))$, also isomorphic to $C(D_2)$, to $C(D_0(11))$, allowing to split the second contractible summand from $C(D_0)$. After removing these contractible summands, the entire complex $C(D_0)$ in Figure~\ref{fig_R23_d0_res} is downsized to $C(D_0(10))$. The inclusion
$C(D_0(10))\subset C(D_0)$ realizing this chain homotopy equivalence is given in coordinates by  $(\mathrm{id},\delta),$ see Figure~\ref{fig_R23_d0_res} with $\delta$ the diagonal map induced by the simplest cobordism from $D_0(10)$ to $D_0(01)$, with the property $\psi_3=\psi_4\delta$.

\begin{figure}[h]
\begin{center}
\includegraphics[scale=0.40]{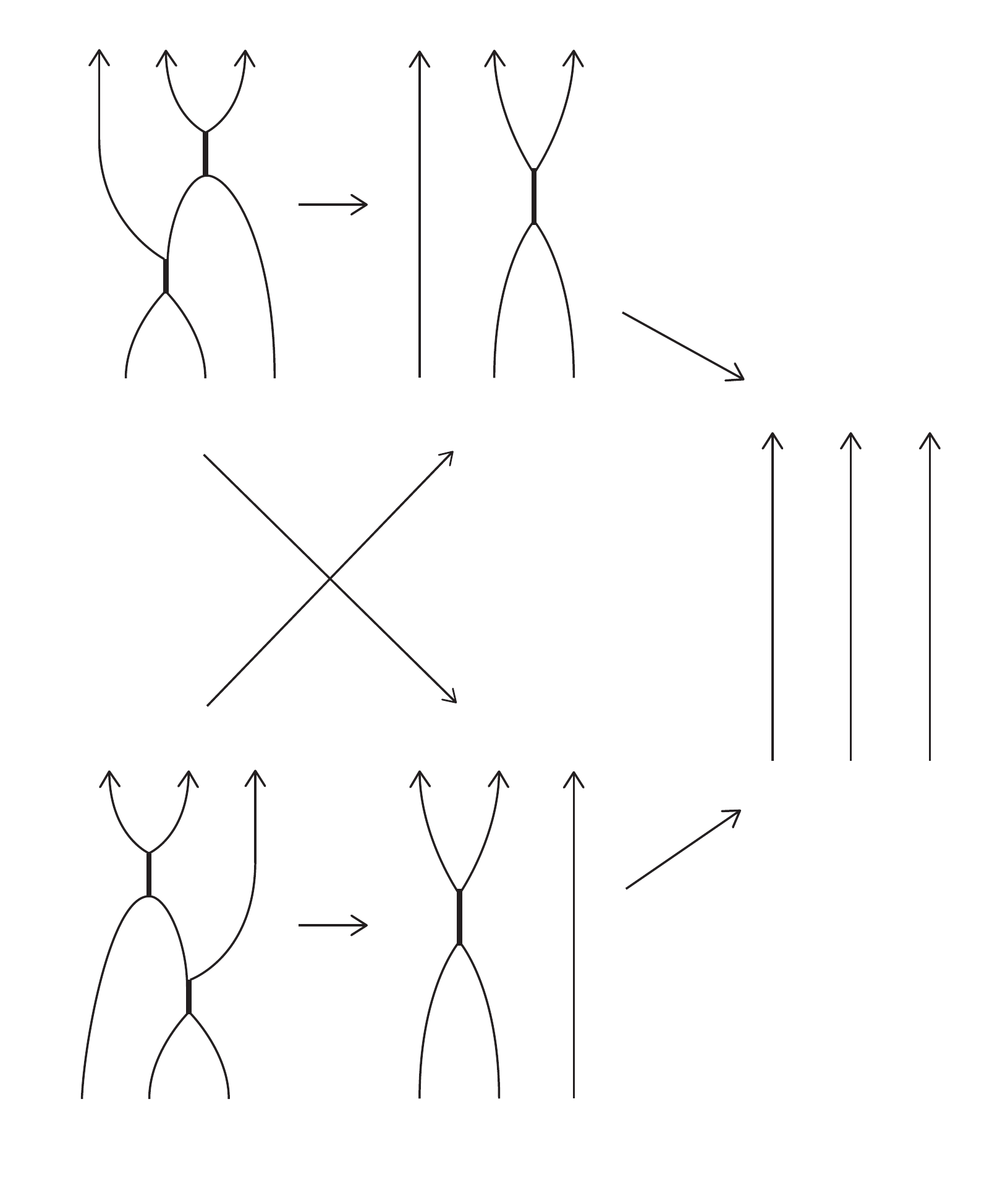}
\caption{\label{fig:R23_final} Common reduction of $C(D)$ and $C(D')$.}
\end{center}
\end{figure}

Reducing the map of complexes $C(D_0)\lra C(D_1)$ to the map $C(D_0(10))\lra C(D_1)$ via the above inclusion of complexes results in the complex shown
in Figure~\ref{fig:R23_final}, with all arrows given by maps induced by the elementary foams between these webs. Signs need to  be added to make each square anticommute, but the isomorphism class of the complex does not depend on the distribution of signs. This complex has an obvious symmetry given by reflecting all diagrams and foams about the vertical axis (or plane, in case of foams) and permuting top and bottom terms in the complex.

The cone of the map $C(D_0')\lra C(D_1')$ in Figure~\ref{fig:R23_partial_res} right reduces to the isomorphic complex, by removing contractible summands of $C(D_0')$ in the same fashion as for $C(D_0)$.
$\square$

This completes the proof of Theorem~\ref{thm_homotopy_invariance}. $\square $

Our proofs of the Reidemeister R2 and R3 relations, for upwards orientations and in $N=2$ case, are essentially identical to those in the usual equivariant case, when $p_{ij}=1$. This observation mirrors our earlier Theorem~\ref{theorem_nilHecke} and Remark~\ref{remark_surprise}  that our deformation does not change the nilHecke algebra relation.
This makes it likely that our $p(x,y)$ deformation does not modify the Soergel category and that the Soergel category will act in the deformed situation as well, with the proofs of Reidemeister R2 and R3 moves for upward orientations identical to that in the $p(x,y)=1$ case. Then $p(x,y)$ deformation would only modify the first Reidemeister move and Reidemeister moves R2 and R3 for non-braid orientations of strands. This expectation mirrors our observation that $p(x,y)$ may only contribute to the deformation of the Frobenius structure, not of multiplication. In the $N=2$ case, similar deformations can be hidden at the level of link homology, see~\cite{V,Kh4}.



\end{document}